\documentclass[a4paper]{amsart}

\usepackage{fullpage}

\usepackage[utf8]{inputenc}

\usepackage{lmodern,amssymb,latexsym,mathtools,slashed,nicefrac,comment, amsrefs, mathrsfs, graphicx,bbold}
\usepackage[all]{xy}
\usepackage{todonotes}

\usepackage[normalem]{ulem} 
\usepackage{amsrefs}

\usepackage{amscd}
\usepackage{enumitem}
\usepackage[colorlinks=true,allcolors=blue]{hyperref}

\newcommand{\Lloc}{L_{\rm{loc}}} 
\newcommand{\Wloc}{W_{\rm{loc}}} 
\DeclareMathOperator{\Brue}{\mathcal B}
\DeclareMathOperator{\Keg}{\mathcal{K}}
\DeclareMathOperator{\spec}{\sigma}

 \makeatletter
 \DeclareFontFamily{OMX}{MnSymbolE}{}
 \DeclareSymbolFont{MnLargeSymbols}{OMX}{MnSymbolE}{m}{n}
 \SetSymbolFont{MnLargeSymbols}{bold}{OMX}{MnSymbolE}{b}{n}
 \DeclareFontShape{OMX}{MnSymbolE}{m}{n}{
     <-6>  MnSymbolE5
    <6-7>  MnSymbolE6
    <7-8>  MnSymbolE7
    <8-9>  MnSymbolE8
    <9-10> MnSymbolE9
   <10-12> MnSymbolE10
   <12->   MnSymbolE12
 }{}
 \DeclareFontShape{OMX}{MnSymbolE}{b}{n}{
     <-6>  MnSymbolE-Bold5
    <6-7>  MnSymbolE-Bold6
    <7-8>  MnSymbolE-Bold7
    <8-9>  MnSymbolE-Bold8
    <9-10> MnSymbolE-Bold9
   <10-12> MnSymbolE-Bold10
   <12->   MnSymbolE-Bold12
 }{}

\let\llangle\@undefined
\let\rrangle\@undefined
\DeclareMathDelimiter{\llangle}{\mathopen}%
                     {MnLargeSymbols}{'164}{MnLargeSymbols}{'164}
\DeclareMathDelimiter{\rrangle}{\mathclose}%
                     {MnLargeSymbols}{'171}{MnLargeSymbols}{'171}
\makeatother

\newtheorem{theorem}{Theorem}[section]
\newtheorem{lemma}[theorem]{Lemma}
\newtheorem{proposition}[theorem]{Proposition}
\newtheorem{corollary}[theorem]{Corollary}

\theoremstyle{definition}
\newtheorem{definition}[theorem]{Definition}

\newtheorem{geometric setting}[theorem]{Geometric setting}
\newtheorem{setup}[theorem]{Setup}
\newtheorem{reminder}[theorem]{Reminder}

\theoremstyle{remark}
\newtheorem{remark}[theorem]{Remark}

\numberwithin{equation}{section}

\newcommand{\loc}{\mathrm{loc}}

\newcommand{\cc}{\mathrm{c}}

\newcommand{\Lip}{\textnormal{Lip}}

\newcommand{\tensor}{\otimes}

\newcommand{\susp}{\mathcal{S}} 
\newcommand{\abs}[1]{\left\lvert#1\right\rvert}

\newcommand{\norm}[1]{\left\|#1\right\|}

\newcommand{\R}{\mathbb{R}}

\newcommand{\vertiii}[1]{{\left\vert\kern-0.25ex\left\vert\kern-0.25ex\left\vert #1 
    \right\vert\kern-0.25ex\right\vert\kern-0.25ex\right\vert}}

\newcommand{\boundedops}{\mathcal{L}}
\newcommand{\vol}{\mathop\mathrm{vol}}
\newcommand{\rk}{\mathop\mathrm{rk}}

\renewcommand{\rm}{\mathrm}
\newcommand{\innerprod}[1]{\langle#1\rangle}

\DeclareMathOperator{\id}{id}

\DeclareMathOperator{\dom}{dom}

\newcommand{\Cl}{\mathrm{Cl}}

\DeclareMathOperator{\ch}{ch}
\DeclareMathOperator{\RR}{\mathbb{R}}
\DeclareMathOperator{\reals}{\mathbb{R}}

\DeclareMathOperator{\CC}{\mathbb{C}}
\DeclareMathOperator{\complexs}{\mathbb{C}}
\DeclareMathOperator{\NN}{\mathbb{N}}

\DeclareMathOperator{\End}{End}
\DeclareMathOperator{\ind}{index}

\DeclareMathOperator{\scal}{scal}

\DeclareMathOperator{\supp}{supp}

\DeclareMathOperator{\Dirac}{\mathcal{D}} 

\DeclareMathOperator{\dist}{dist}
\DeclareMathOperator{\spinor}{\slashed{\mathfrak S}}

\DeclareMathOperator{\thet}{\vartheta}

\DeclareMathOperator{\EE}{\mathbf{E}}
\DeclareMathOperator{\FF}{\mathbf{F}}
\DeclareMathOperator{\LL}{\mathbf{L}}
\DeclareMathOperator{\HH}{\mathbf{H}}
\DeclareMathOperator{\bulk}{\rm{bulk}}
\DeclareMathOperator{\cone}{\rm{cone}}
\DeclareMathOperator{\link}{\rm{link}}
\DeclareMathOperator{\PP}{\mathcal{P}}
\DeclareMathOperator{\PPP}{\mathscr{P}}
\DeclareMathOperator{\tr}{tr}
\DeclareMathOperator{\conic}{\mathscr C}

\DeclareMathOperator{\SSS}{\mathbb{S}}

\DeclareMathOperator{\eps}{\varepsilon}
\DeclareMathOperator{\specflow}{sf}

\newcommand{\Spin}{\mathrm{Spin}}

\newcommand{\myicon}{$\,\,\,\triangleright$}

\makeatletter
\let\c@equation\c@theorem

\makeatother

\begin{document}

\title{Abstract cone operators and Lipschitz rigidity for scalar curvature on singular manifolds}

\author{Simone Cecchini}
\address{Department of Mathematics, Texas A\&M University, College Station, TX, USA}
\email{cecchini@tamu.edu}

\author{Bernhard Hanke}
\address{Institut f\"ur Mathematik,
Universit\"at Augsburg,  Augsburg, Germany}
\email{bernhard.hanke@math.uni-augsburg.de}

\author{Thomas Schick}
\address{Mathematisches Institut,
Georg-August-Universit\"at, 
G{\"o}ttingen,
Germany}
\email{thomas.schick@math.uni-goettingen.de}

\author{Lukas Sch\"onlinner}
\address{Institut f\"ur Mathematik,
Universit\"at Augsburg,  Augsburg, Germany}
\email{lukas.schoenlinner@uni-a.de}

\begin{abstract}
Using the index theory for twisted Dirac operators acting on sections of Lipschitz bundles over non-compact manifolds, we prove  Llarull-type comparison results in scalar curvature geometry. 
They apply to spin Riemannian manifolds with cone-like singularities and Lipschitz comparison maps to spheres.

We use the  language of abstract cone operators which are introduced and studied in a general functional analytic setting and which may be of independent interest.

Applying our discussion to spherical suspensions of odd-dimensional closed manifolds, we  generalize a  Lipschitz rigidity  result of the first three named authors from even to odd dimensions.
Under stronger conditions, this has already been shown by Lee-Tam using geometric flows and by B\"ar using an upper estimate for the smallest Dirac eigenvalue.
\end{abstract}

\keywords{Abstract cone operator, scalar curvature bounds, Lipschitz maps, twisted Dirac operators, low regularity metrics, cone-like singularities}

\subjclass[2000]{Primary: 51F30, 53C23, 53C24 ;  Secondary: 30C65, 53C27, 58J20 }

\maketitle
\tableofcontents

\section{Introduction and summary}

What quantitative bounds on positive scalar curvature can one place on a given smooth manifold?
To prevent Riemannian metrics from scaling, it is useful to compare them with a standard background metric and then look for extremality and rigidity results.
A particularly striking example is Llarull's rigidity theorem, which also inspired the present work.

In the following  let $\SSS^n$ denote the unit sphere in $\R^{n+1}$ with the induced Riemannian metric, distance function and a fixed orientation.

\begin{theorem}[\cite{Lla98}*{Theorem B}] \label{Llarull_Classic}
Let $n \geq 2$, let $(M,g)$ be a closed connected $n$-dimensional smooth spin Riemannian manifold with scalar curvature $\scal_g \geq n(n-1)$.
Let  $f\colon (M,g)\to \SSS^n$ be a smooth $1$-Lipschitz map of non-zero degree.
Then $f$ is a Riemannian isometry.

If $n \geq 3$, then the same conclusions hold under the weaker (infinitesimal) area non-increasing condition that $\left| \Lambda^2d_xf \right| \le 1$ for all $x\in M$.
\end{theorem}

Subsequent work has led to similar results for  other comparison manifolds. 
Goette and Semmelmann  generalized the above results to many compact symmetric spaces in \cite{GS02} and  to compact K\"ahler manifolds of positive Ricci curvature in \cite{GS01}.
Lott \cite{Lott_bound} proved Llarull type comparison results for smooth compact manifolds with boundary.

In his effort to understand scalar curvature as a purely metric concept, Gromov in \cite{Gromov4}*{Section 4.5} asked whether in Llarull's theorem, the smoothness of the map $f$ can be dropped.
The following result resolved this question for even-dimensional manifolds, including a lower regularity assumption for $g$.

\begin{theorem}[\cite{CHS}*{Theorem A}]    \label{thm_CHS}
Let $M$ be a closed smooth connected  oriented manifold of even dimension $n \geq 2$ which admits a spin structure, let $g$ be a $W^{1,p}$-regular  Riemannian metric for some $p>n$ with distributional scalar curvature $\scal_g \geq n(n-1)$.
Let $f \colon (M,d_g) \to \SSS^n$ be a  Lipschitz continuous map of non-zero degree.
Furthermore, if $n = 2$, assume that $f$ is $1$-Lipschitz, and if $n \geq 4$, assume that $df$ is area non-increasing almost everywhere.

Then  $f$ is a metric isometry. 
\end{theorem}

\begin{reminder}
  Let $M$ and $N$ be smooth manifolds with continuous Riemannian metrics $g$ and $h$ and associated distance functions $d_g$ and $d_h$.
  Let $f\colon (M, d_g) \to (N,d_h)$ be Lipschitz continuous.  
  If $f$ is (totally) differentiable at $x \in M$, we denote by
  \[
    d_x f\colon T_x M \to T_{f(x)} N
  \]
  the differential of $f$ at $x$.  
  By Rademacher's theorem, $f$ is differentiable almost everywhere on $M$ with differential of regularity $\Lloc^{\infty}$.

We say that $df$ is {\em area non-increasing a.e.} if for almost all $x \in M$ where $f$ is differentiable the operator norm of the induced map
  \[
    \Lambda^2 d_x f \colon \Lambda^2 T_x M \to \Lambda^2 T_{f(x)} N
  \]
 on the second exterior power of $T_x M$ satisfies $| \Lambda^2 d_x f| \leq 1$.

Next, let  $M$ be a smooth manifold and let  $g$ be a continuous Riemannian metric on $M$ satisfying
\[
    g \in \Wloc^{1,p}( M , T^*M \otimes T^*M)
\]
for some $p>n$.
Recall here that,  by the Sobolev embedding theorem, each section in $\Wloc^{1,p}(M, T^*M \otimes T^*M)$, $p > n$, has a unique continuous representative. 
The case $p=\infty$ is precisely the case of Lipschitz continuity.

For such $g$,  the scalar curvature  is defined as a distribution 
\[
    \scal_g \colon C^{\infty}_c(M) \to \R, 
\]
see  Lee-LeFloch \cite{LL15} and \cite{CHS}*{Definition 3.3}.
In particular, we can define lower scalar curvature bounds in the distributional sense by setting $\scal_g\ge T$ for another distribution $T$ if and only if $\llangle \scal_g,\varphi \rrangle \ge \llangle T,\varphi \rrangle$ for  every smooth test function $\varphi\colon M\to \R$ which takes only non-negative values. 
If $T$ is a constant (the case most relevant to us) this means of course that
$\llangle \scal_g,\varphi \rrangle \ge T\int_M \varphi$ for all smooth test functions $\varphi\colon M\to [0,\infty)$.
\end{reminder}

In \cite{CHS}, to prove Theorem \ref{thm_CHS}, by establishing the index and spectral theories of twisted Dirac operators for low-regularity bundles and metrics.
Notably, a priori non-existence results of harmonic spinors arise from low regularity versions of the Schrödinger–Lichnerowicz formula.
Appropriate approximation arguments then allow to show, by following Llarull's line of thought, that in the situation of Theorem \ref{thm_CHS}, the differential $d_x f$  is an isometry almost everywhere.

In Proposition \ref{P:infinitesimal_isometry} of the present article, this reasoning is generalized to odd-dimensional $M$ of dimension at least $3$.
To do so, we combine a spectral flow argument similar to the one in the work of Li, Su and Wang \cite{LiSuWang} and B\"ar \cite{Baer_2024} with the methods developed in \cite{CHS}.

For non-differentiable $f$, the inverse function theorem cannot be applied, so a new argument is needed to conclude that $f$ is a homeomorphism.
In \cite{CHS} we use the theory of quasi-regular maps in the sense of Reshetnyak to show that $f$ is a homeomorphism under the additional assumption that $d_x f$ is almost everywhere orientation preserving.
Intuitively, this  assumption  avoids the folds and wrinkles that are generally responsible for the large flexibility of maps that are only Lipschitz regular.
To obtain this result, we have made decisive use of the fact that $M$ is even-dimensional and the resulting chiral decomposition of the spinor bundle on $M$, leaving open the case of odd $n$ at this point.

After \cite{CHS} had appeared, Theorem \ref{thm_CHS} has been generalised to all $n \geq 2$ in the following cases.
\begin{enumerate}[label=\myicon]
 \item  If $f$ is $1$-Lipschitz by  Lee-Tam \cite{ML22}*{Theorem 1.1} using a combination of the Ricci and harmonic map heat flows to reduce the claim to Theorem \ref{Llarull_Classic}.
\item If $f$ is $1$-Lipschitz and $g$ is smooth by B\"ar \cite{Baer_2024}*{Theorem 1}, using an upper bound for the smallest Dirac eigenvalue in terms of the  hyperspherical radius, see \cite{Baer_2024}*{Main Theorem}.
\end{enumerate}
Combining \cite{ML22}*{Theorem 1.1} with our Proposition \ref{P:infinitesimal_isometry}, we obtain 

\begin{corollary} \label{optimal_result}  Theorem \ref{thm_CHS} holds for all $n \geq 2$.
\end{corollary}

One of the goals of this paper is to present a  purely spinor geometric proof of the following version of Theorem \ref{thm_CHS} in odd dimensions that is independent of geometric flows. 

\begin{theorem} \label{theo:main_odd} Let $M$ be a closed smooth connected oriented manifold of odd dimension $n \geq 3$ admitting a spin structure.
Let $g$ be a Riemannian metric of Sobolev regularity $W^{1,p}$, $p > n+1$, and with distributional scalar curvature $\scal_g \geq n(n-1)$ on $M$.
Let $f \colon M \to \SSS^n$ be a Lipschitz map of non-zero degree which is
area non-increasing almost everywhere. 

Then $f$ is a metric isometry.
\end{theorem}

The reasoning for this result can be outlined as follows. 
By Proposition \ref{P:infinitesimal_isometry}, we can assume that $f$ is $1$-Lipschitz.
To go from odd to even dimensions, we  perform a spherical suspension of the map $f \colon M \to \SSS^n$.
That is, we consider the open Riemannian manifold $\susp M := M \times (0,\pi)$ with the Riemannian metric $dr^2 + \sin(r)^2 g$, where $r$ is the canonical coordinate on $(0, \pi)$.
This construction is functorial for $1$-Lipschitz maps, and we have $\susp \SSS^n = \SSS^{n+1}\setminus\{N,S\}$, the unit $(n+1)$-sphere with north and south poles removed.
While the metric completion of $\susp \SSS^n$ is an even-dimensional smooth manifold, this is not the case for $\susp M$ in general.
Such objects are known as {\em manifolds with conical singularities} in the literature.

This leads us to consider versions of the theorems \ref{Llarull_Classic} and \ref{thm_CHS} for manifolds $M$ with conical singularities, which is the subject of this article.
In the smooth case, there is a well-established index theory for manifolds with conical singularities, going back to the work of Cheeger \cite{Cheeger}, Br\"uning-Seeley \cite{BS88} and Br\"uning \cite{Bruening}, among others.
In our situation, however, we face two fundamental challenges:

\begin{enumerate}[label=\myicon]
\item The cross-sectional twisted Dirac operator acts on Lipschitz sections, so we have to combine this type of singularity with the cone point singularity produced when passing to the spherical suspension of $M$.
\item The metric estimates are very delicate and do not allow for any deformations.
In particular, we cannot assume that our twist bundles are trivialized with a trivial connection near the cone tips, but we have to allow curvature of the twist bundle all the way along the cone region.
\end{enumerate}

In Section \ref{abstr_cone} of this paper we develop an abstract functional analysis setup that allows us to treat these situations in a unified way.
The main objects are so-called {\em abstract cone operators} which we introduce in an axiomatic setting in Definition \ref{abstractcone}.
These are closable unbounded linear operators between Hilbert spaces, modelling conical suspensions of essentially self-adjoint unbounded operators between Hilbert spaces.
In the case of the spherical suspension of a manifold $(M,g)$ as in Theorem \ref{theo:main_odd}, this essentially self-adjoint operator is the twisted Dirac operator studied in \cite{CHS}.
Theorem \ref{thm:RegularSingular} discusses the main properties of abstract cone operators, including their Fredholm property and the determination of their domain.
In spirit, our exposition can be seen as a functional-analytic distillation of some ideas developed in \cite{Bruening}.
We  expect that this may also be helpful in other situations.

In Section \ref{spherical_suspension}, we apply the theory developed in Section \ref{abstr_cone} to spherical suspensions of low-regularity metrics, thereby proving Theorem \ref{theo:main_odd}.
The condition $p > n+1$ in Theorem  \ref{theo:main_odd} is slightly stronger than the expected condition $p > n$.
This is because the dimension increases by one under the spherical suspension.
We are confident that, with a more refined analysis, our approach to Theorem  \ref{theo:main_odd} will also cover the more general case of $p > n$.
Note that according to Corollary \ref{optimal_result}, Theorem \ref{theo:main_odd} does hold  for all $p > n$.

An immediate corollary of our Theorem \ref{theo:main_odd}  is that the rigidity of disks \cite{CHS}*{Theorem B}  generalises to the odd-dimensional case.
We recall here that the doubling procedure in the proof of \cite{CHS}*{Theorem B} produces Lipschitz Riemannian metrics, which are of Sobolev regularity $W^{1,\infty}$, so that our Theorem \ref{theo:main_odd} applies.

In  Section \ref{Lip_Cone}, we present a second, novel application of the material developed in Section \ref{abstr_cone}.
Specifically,  we prove a Llarull comparison theorem for even-dimensional Riemannian manifolds $(N,G)$ with cone-like singularities and for Lipschitz comparison maps $f$ to spheres.
For a positive mass theorem for asymptotically flat spin manifolds with isolated conical singularities, see \cite{DaiSunWang}*{Theorem 1.1}.
In order to keep the exposition transparent, we restrict our attention to smooth Riemannian metrics $G$.

Near the singular points, the metrics are assumed to be of a particular form as described in the next definition.

\begin{definition}\label{generalizedConeMetric} Let $M$ be a closed smooth manifold.
Let $0 < \vartheta \leq 1$ and $g_r\in C^{\infty}((0,\vartheta), C^\infty(M, T^*M\otimes T^*M))$ for $r \in (0, \thet)$ be a smooth family of Riemannian metrics on $M$.
Furthermore, let
\begin{equation*}
	g_0 \in C^2(M, T^*M\otimes T^*M)
\end{equation*}
be a $C^2$-Riemannian metric.

Assume that 
   \begin{equation*}
	 g_r \overset{r\to 0}{\longrightarrow} g_0\qquad\text{in }C^2(M, T^*M\otimes T^*M) 
  \end{equation*}
and that, for $x \in M$, we have
 \[
      \lim_{r \to 0} \abs{r \, \partial_r g_{r} }_{(T_x M, g_r)} = 0 , \qquad  \lim_{r \to 0} \abs{r^2 \, \partial_r^2 g_{r} }_{(T_xM, g_r)}  = 0. 
\]
Then the metric $\conic g_r:= d r^2 + r^2 g_r$ on $(0,\vartheta)\times M$ is called a \emph{generalized cone metric}.
\end{definition}

\begin{definition}\label{MfdConeLikeSing}
A \emph{compact Riemannian manifold with cone-like singularities} consists of the following data. 
\begin{enumerate}[label=\myicon]
 \item a smooth Riemannian manifold $(N, G)$, 
 \item a compact smooth submanifold $\mathcal{K} \subset N$, possibly with boundary,
 \item some $0 < \vartheta \leq 1$ and a smooth Riemannian isometry
\[
      \nu\colon (N\setminus \mathcal{K} , G\vert_{N\setminus \mathcal{K}}) \to ((0,\vartheta) \times \partial \mathcal{K}, {\conic g_r})
 \]
 where ${\conic g_r}$ is a generalized cone metric as in Definition \ref{generalizedConeMetric}.
\end{enumerate}
\end{definition}

Spherical suspensions of closed smooth Riemannian manifolds give examples of manifolds with cone-like singularities.

Let $n \geq 1$, let $(N,G)$ be a compact connected Riemannian manifold of dimension $n+1$ with cone-like singularities.
Let $\bar N$ be the metric completion of $N$ with respect to the path metric induced by $G$.
Note that $\bar N$ is compact.
There are finitely many points $x_1, \ldots, x_k \in \bar N$, called the \emph{cone tips} of $\bar N$, such that 
\begin{equation*}
	\bar N\setminus N = \{x_1, \ldots, x_k\}.
\end{equation*}
The cone tips are in one-to-one correspondence with the connected  components  of $\partial \mathcal{K}$.
The connected component of $\partial \mathcal{K}$ corresponing to the cone tip $x_i$ is called  the {\em link manifold} of $x_i$.
The space $\bar N$ is a topological manifold,  if and only if the link manifold of each $x_i$ is homeomorphic to the $n$-sphere.

Let $f \colon (N, d_G) \to \SSS^{n+1}$ be a Lipschitz continuous map. 
Let  $\Lambda > 0$ be a Lipschitz constant for $f$.
Since $\SSS^{n+1}$ is a complete metric space, the map $f$ extends uniquely to a $\Lambda$-Lipschitz map $\bar f\colon \bar N \to \SSS^{n+1}$.
Choose $0 < \eps < \tfrac{\pi}{2}$ small enough such that the open $\eps$-balls around $\bar f(x_i)$ and $\bar f(x_j)$, $1 \leq i < j \leq k$, are disjoint if $\bar f(x_i) \neq \bar f(x_j)$.
Let $V \subset \SSS^{n+1}$ be the union of these open $\eps$-balls. 
Possibly after passing to a larger $\mathcal{K} \subset N$, we can assume that $f(N \setminus \rm{ int}(\mathcal{K})) \subset V$, in particular, $f( \partial \mathcal{K} ) \subset V$.
Then $f$ induces a map $H_{n+1}(\mathcal{K}, \partial \mathcal{K}; \mathbb{Z}) \to H_{n+1}(\SSS^{n+1}, V; \mathbb{Z}) \cong \mathbb{Z}$. 
If $N$ is oriented, the image of the orientation class of $(\mathcal{K}, \partial \mathcal{K})$ defines the  homological mapping degree of $f$, denoted $\deg (f) \in \mathbb{Z}$.

\begin{theorem}\label{LlarullGeneralizedCone}
Let $n \geq 3$ be odd and let $(N^{n+1},G)$ be a compact connected oriented Riemannian manifold with cone-like singularities.
Assume that $N$ admits a spin structure and that  $\scal_G \geq (n+1)n$.
Let  $f\colon N \to \SSS^{n+1}$ be a Lipschitz continuous map such that $f$ is area non-increasing almost everywhere and $f$ has non-zero degree.
        
Let $x_1, \ldots, x_k$ be the cone tips of $\bar N$ and let $\bar f \colon \bar N \to \SSS^{n+1}$ be the metric extension of $f$.

Then  $f$ is a smooth Riemannian isometry of $(N,G)$ onto $\SSS^{n+1} \setminus \{\bar f(x_1), \ldots, \bar f (x_k)\}$.
In particular,  $N$ is diffeomorphic to a $k$-punctured sphere.
\end{theorem}

This theorem holds without further differential-topological assumptions on the link manifolds.
If the link manifolds are diffeomorphic to standard spheres and $f$ is smooth, then Theorem \ref{LlarullGeneralizedCone} follows from \cite{ChuLeeZhu}*{Theorem 1.4}.

\textit{Acknowledgments:} 
We are grateful to Claude LeBrun for useful remarks.
The authors gratefully acknowledge support by the SPP 2026 ``Geometry at Infinity'' funded by the Deutsche Forschungsgemeinschaft (DFG, German Research Foundation).
S.~C.~was supported by a grant from the Simons Foundation (MPS-TSM-00007902, SC).
B.H.~and T.~S.~were supported by the Hausdorff Research Institute for Mathematics funded by the DFG under Germany's Excellence Strategy – EXC-2047/1 – 390685813.

\section{Abstract cone operators} \label{abstr_cone}

\subsection{Definition, Fredholm property} \label{sec:BrueningSetting}

If $V$ and $W$ are normed vector spaces, we denote by $\boundedops(V,W)$ the space of bounded linear maps $V \to W$ equipped with the operator norm.
For $V = W$ we write  $\boundedops(V)$ instead of $\boundedops(V,V)$.
In the following we will work with complex Hilbert spaces for the sake of clarity.
The case of real Hilbert spaces is analogous.

Let us assume that we are given the following data.
\begin{enumerate}[label=\textup{(\Roman*)}]
  \item \label{uno} separable Hilbert spaces $\EE$ and $\FF$ and a closable, densely defined  linear operator 
 \[  
    \Brue\colon \EE \supset \dom(\Brue) \to \FF ;
\] 
   \item \label{due} orthogonal decompositions 
\[
     \EE = \EE_{\bulk}\oplus \EE_{\cone}, \quad  \FF = \FF_{\bulk}\oplus \FF_{\cone} ;
\]
  \item  \label{tre} a separable Hilbert space  $\LL$, a real number $0 <  \thet \leq 1$, and  Hilbert space isometries
\[
     \Phi_{\EE} \colon \EE_{\cone}  \to L^2\left((0,\thet),\LL \right) , \quad \Phi_{\FF} \colon \FF_{\cone}   \to  L^2\left((0,\thet),\LL \right) ;
  \]
    \item \label{quattro}  an essentially self-adjoint operator, called the \emph{link operator} of $\Brue$,
   \[
      S_0  \colon \LL \supset \dom(S_0)  \to \LL ;
  \]
  \item \label{cinque}  an essentially bounded measurable map,  called the {\em perturbation} of $S_0$,
  \[
    S_1 \colon (0,\thet)\to\mathcal L(\dom(S_0),\LL),
   \]
   where $\dom(S_0)$ is equipped with the (possibly incomplete) graph norm of $S_0$.
 \end{enumerate}

It is not required, and in applications usually not true, that $\dom(\Brue) \subset \EE$ splits as a direct sum along the decomposition $\EE = \EE_{\cone} \oplus \EE_{\bulk}$.
 
From now on, we identify $\EE_{\cone}$ with $L^2((0,\thet), \LL)$ via $\Phi_{\EE}$, and $\FF_{\cone}$ with $L^2((0,\thet), \LL)$  via $\Phi_{\FF}$.
We obtain pointwise multiplication operations on $\EE$  and on $\FF$ as follows:
For  $\varphi\in C^\infty_\cc([0,\thet),\R)$ and  $u = (u_b,u_c)\in \EE_{\bulk}\oplus \EE_{\cone} = \EE_{\bulk} \oplus L^2((0,\thet), \LL)$, we set
  \begin{align*}
    \varphi\cdot u & := (0,  \varphi \cdot u_c) \in \EE , \\
    (1-\varphi)\cdot u & := (u_b, (1-\varphi) \cdot u_c) \in \EE . 
  \end{align*}
The construction for $\FF$ is analogous.

Let  $\bar S_0 \colon \LL \supset \dom(\bar S_0) \to \LL$ denote the (unique) self-adjoint extension of $S_0$ and let  $\sigma(\bar S_0) \subset \R$ be its spectrum.
Moreover, for $r \in (0,\thet)$, let $\bar S_1(r) \colon \dom(\bar S_0) \to \LL$ denote the (unique) bounded extension of $S_1(r)$.
We have  $\| \bar S_1(r)\|_{\boundedops(\dom(\bar S_0), \LL)} = \|S_1(r)\|_{\boundedops(\dom(S_0), \LL)}$ so that the family $\bar S_1 \colon (0,\thet) \to \boundedops ( \dom(\bar S_0), \LL)$ is essentially bounded.
 
 Let  
\[
      \bar{\Brue} \colon \EE \supset \dom(\bar{\Brue}) \to \FF
\] 
be the closure of $\Brue$.
We equip $\dom(\Brue)$ and $\dom(\bar \Brue)$  with the graph norm of $\bar \Brue$, and $\dom(\bar S_0)$ with the graph norm of $\bar S_0$.

For $k \in \NN_0 \cup \{\infty\}$, let $C_\cc^k((0,\thet),\dom(S_0))$ denote the space of compactly supported $C^k$-functions on $(0,\thet)$ with values in the normed vector space $\dom(S_0)$.
Furthermore, for $k \in \NN_0 \cup \{ \infty \}$,  let $C^{k}_{\cc,0}([0,\thet), \R)$ denote the space of compactly supported real valued $C^k$-functions on $[0,\thet)$ that are equal to $1$ near $0$.

 Now consider the following conditions.
 
     \begin{enumerate}[start=0,label=\textup{(AC\arabic*)}]
 \item (Locality of domain)\label{AnSetup3} Let $\psi \in C^\infty_{\cc,0}([0,\thet),\R)$.
Then we have $\psi \cdot \dom (\Brue) \subset \dom (\Brue)$.
    \item (Domain over cone part) \label{AnSetup4} 
  We have
  \[
        C_\cc^\infty((0,\thet),\dom(S_0)) \subset \dom( \Brue).
       \]
Furthermore, for all $\psi \in C^{\infty}_{\cc,0} ([0,\thet), \R)$, the inclusion
 \[
    \psi \cdot C_\cc^\infty((0,\thet),\dom(S_0)) \subset \psi \cdot \dom(\Brue) 
 \]
 is dense with respect to the graph norm of $\dom(\Brue)$.
        \item (Locality of operator) \label{AnSetup2} $\Brue$ is  local on the cone part,
    \begin{equation*}
      \Brue( C_\cc^\infty((0,\thet),\dom(S_0))) \subset \FF_{\cone}.
    \end{equation*}
  Moreover, $\Brue$ is local away from the cone part: If $\varphi , \psi \in C_{\cc,0}^{\infty}([0,\thet), \RR)$ with $\psi|_{\supp \varphi} \equiv 1$, then 
    \[
          \varphi \Brue (1-\psi) = 0.
     \]
      \item (Product structure over cone part) \label{AnSetup5} 
      For all $u\in  C^\infty_\cc((0,\thet),\dom(S_0))$ and all $r \in (0,\thet)$, we have
    \begin{equation*} 
      \Brue u(r) =\partial_ru(r)+\tfrac{1}{r}\left(S_0+S_1(r) \right)u(r) .
    \end{equation*}
     \item (Spectral gap) \label{AnSpGap} We have 
     \[
          \sigma(\bar S_0) \cap [-\tfrac{1}{2}, \tfrac{1}{2}] = \emptyset.
      \]
      In particular, by functional calculus, $\bar S_0$ has a bounded inverse $\bar S_0^{-1} \colon \LL \to \dom(\bar S_0)$.
     \item (Small perturbation I) \label{AnSetup6}  For all $r \in (0,\thet)$, the bounded operator  
  \[
        \bar S_0^{-1}  \bar S_1(r)\colon \dom(\bar S_0) \to \dom(\bar S_0) \to \LL
  \]
  extends to a bounded operator $\LL \to \LL$ and  $\bar S_0^{-1} \bar S_1 \colon (0,\thet) \to \boundedops(\LL)$ is essentially bounded.
     \item (Small perturbation II) \label{boundedperturbtwo}
     We have 
\begin{align*}
 \norm{  \bar S_1\bar  S_0^{-1}}_{L^{\infty}((0, \vartheta), \boundedops(\LL))} & \leq  \inf_{s \in \sigma(\bar S_0)} \left|  \frac{ 2s +1}{4s}\right| , \\
    \norm{  \bar  S_0^{-1} \bar S_1}_{L^{\infty}((0, \vartheta), \boundedops(\LL))} & \leq  \inf_{s \in \sigma(\bar S_0)} \left|  \frac{ 2s - 1}{4s}\right| . 
\end{align*}
  \end{enumerate}

\begin{definition} \label{abstractcone} An {\em abstract cone operator} consists of data \ref{uno} - \ref{cinque} subject to conditions \ref{AnSetup3} - \ref{AnSetup6}.
\end{definition}

Let $\Brue \colon \EE \supset \dom(\Brue) \to \FF$ be an abstract cone operator.
Using  algebraic tensor products of complex vector spaces, we have dense inclusions
\begin{align*}
   C_{\cc}^{0}((0,\thet),\CC) \otimes  \dom(S_0)   &  \subset L^2((0,\thet),\CC) \otimes \LL   \subset  L^2((0,\thet), \LL), \\
   C_{\cc}^{\infty}((0,\thet),\CC) \otimes  \dom(S_0)   &  \subset C_{\cc}^\infty((0,\thet),\dom(S_0))   \subset  L^2((0,\thet), \LL) .
 \end{align*} 
This elementary fact is used at several places in our argument.

By \ref{AnSetup4} and  \ref{AnSetup2}, the restriction of $\Brue$ yields a densely defined closable operator 
 \[
     \Brue_{\cone} \colon L^2((0,\thet), \LL)\supset C_{\cc}^\infty((0,\thet),\dom(S_0))  \to L^2((0,\thet), \LL)
 \]
given by the formula \ref{AnSetup5}.
We consider its closure
\[
    \bar \Brue_{\cone} \colon L^2((0,\thet), \LL)\supset \dom(\bar \Brue_{\cone})  \to L^2((0,\thet), \LL) .
\]

\begin{lemma} \label{maptocone} Let $\varphi \in C_{\cc}^{\infty}([0,\thet), \RR)$.
Then multiplication with $\varphi$ induces a bounded map $\dom(\bar \Brue) \to \dom(\bar \Brue_{\cone})$ and, for all $u \in \dom(\bar \Brue)$, the Leibniz formula holds:
\[
     \bar \Brue_{\cone}(\varphi \cdot u) = \varphi \cdot \bar \Brue u + \varphi' \cdot u . 
 \]
\end{lemma} 

\begin{proof} 
Let  $u \in \dom(\bar \Brue)$ and let $u_n \in \dom(\Brue)$ be a sequence approximating $u$ in $\dom( \bar \Brue)$.
 Then $\varphi \cdot u_n \stackrel{n \to \infty}{\longrightarrow} \varphi \cdot u$ in $\EE$.
Let $\psi \in C_{\cc,0}^{\infty}([0,\thet), \RR)$ such that $\psi|_{\supp \varphi} \equiv 1$.
By  \ref{AnSetup4}, for each $n$, we find a sequence $\tilde u_{n,k} \in C_{\cc}^{\infty}((0,\thet), \dom(S_0))$, $k \in \NN$, such that $\tilde u_{n,k} \stackrel{k \to \infty}{\longrightarrow} \psi u_n$  in   $\dom(\Brue)$.
We have $\varphi \cdot \tilde u_{n,k} \stackrel{k \to \infty}{\longrightarrow} \varphi \cdot \psi u_n$ in $L^2((0,\thet), \LL)$.
Furthermore, for each $k$, \ref{AnSetup5} implies
\[
    \Brue(\varphi \cdot \tilde u_{n,k}) = \varphi \cdot \Brue(\tilde u_{n,k}) + \varphi' \cdot \tilde u_{n,k}.
\]
For $k \to \infty$, the right hand side converges to $\varphi \cdot \Brue(\psi u_n) +  \varphi' \cdot \psi u_n$ in $L^2((0,\thet), \LL)$.
Hence, 
\begin{equation} \label{easyleib}
        \Brue( \varphi \cdot \psi u_n) = \varphi \cdot \Brue( \psi u_n) + \varphi' \cdot (\psi u_n).
 \end{equation}

We compute
\begin{align*}
       \Brue(\varphi \cdot u_n) & = \Brue( \varphi \cdot \psi u_n) \\
                                              & \stackrel{\eqref{easyleib}}{=} \varphi \cdot \Brue( \psi u_n) + \varphi' \cdot (\psi u_n) \\
                                              & = \varphi \cdot \Brue(u_n) - \varphi \cdot \Brue((1-\psi) u_n) +  \varphi' \cdot u_n  \\
                                              & \stackrel{\ref{AnSetup2}}{=} \varphi \cdot \Brue(u_n) + \varphi' \cdot u_n.
 \end{align*}
 Hence, in $\FF$, we get
 \[
      \Brue(   \varphi \cdot u_n)  =  \varphi \cdot \Brue(u_n) + \varphi' \cdot u_n \stackrel{n \to \infty}{\longrightarrow}  \varphi \cdot \bar \Brue (u) + \varphi' \cdot u  .
 \]
This shows $\varphi \cdot u \in \dom (\bar \Brue_{\cone})$ and the Leibniz formula.
The boundedness of multiplication with $\varphi$ is an immediate consequence. 
\end{proof}

 \begin{definition} Let $\Brue \colon \EE \to \FF$ be an abstract cone operator. 
Pick $\varphi,\psi\in C^\infty_{\cc,0}([0,\vartheta);\RR)$ such that $\psi|_{\supp \varphi}  \equiv 1$.
An {\em interior parametrix} of $\bar \Brue$ for $\varphi$ and $\psi$  is a bounded operator $P \colon \FF \to \dom(\bar \Brue)$ together with compact operators $R$ on $\FF$ and  $L$ on $\dom( \bar \Brue)$ such that
    \begin{align*}
        \bar \Brue (1-\varphi) P (1-\psi) & =(1-\psi)+R, \\
        (1-\varphi) P(1-\psi)\bar \Brue&=(1-\psi)+L. 
      \end{align*}
\end{definition}

Consider the densely defined closable operators
\begin{align*}
     \partial_r & \colon L^2((0,\vartheta), \LL) \supset  C_{\cc}^{\infty}( (0,\vartheta), \dom(S_0)) \to L^2((0,\vartheta),\LL),  \\
     \tfrac{1}{r} S_0 & \colon L^2((0,\vartheta), \LL) \supset  C_{\cc}^{\infty}( (0,\vartheta), \dom(S_0)) \to L^2((0,\vartheta),\LL),
 \end{align*}
and let
\begin{align} 
   \label{cute1}   \overline{\partial_r}& \colon L^2((0,\vartheta), \LL) \supset  \dom(\overline{\partial_r} ) \to L^2((0,\vartheta),\LL), \\
    \label{cute2}  \overline{\tfrac{1}{r} S_0} & \colon L^2((0,\vartheta), \LL) \supset  \dom( \overline{\tfrac{1}{r} S_0})  \to L^2((0,\vartheta), \LL) 
 \end{align}
be their closures.

\begin{theorem}\label{thm:RegularSingular}
  Let $\Brue$ be an abstract cone operator.
  Then
  \begin{enumerate}[label=\textup{(\alph*)}]  
  \item \label{b} We have $\dom(\bar \Brue_{\cone}) = \dom(\overline{\partial_r} ) \cap \dom( \overline{\tfrac{1}{r} S_0})$. 
  \item \label{c} The graph norm on $ \dom(\bar \Brue_{\cone})$ is equivalent to the Sobolev $1$-norm
\[
     \norm{u}_{H^1_{\cone}} := \left(  \norm{u}^2_{L^2((0,\thet), \LL)} + \norm{\overline{\partial_r} u}^2_{L^2((0,\thet), \LL)}  + \norm{\overline{\tfrac{1}{r} S_0} u}^2_{L^2((0,\thet), \LL)}\right)^{\tfrac{1}{2}} .
\]
  \end{enumerate}
Furthermore, for all $\varphi\in C_{\cc,0}^\infty([0,\thet), \RR)$,  the following statements hold.
 
  \begin{enumerate}[label=\textup{(\alph*)}]  
  \setcounter{enumi}{2}
      \item \label{d}  The graph norm on $\dom(\bar \Brue)$ is equivalent to the Sobolev $1$-norm
  \[
    \norm{u}_{H^1} := \left(  \norm{u}^2_{\EE}  + \norm{\bar \Brue((1-\varphi)u)}^2_{\FF} + \norm{\overline{\partial_r}(\varphi u)}^2_{L^2((0,\thet),\LL)} + \norm{\overline{\tfrac{1}{r}  S_0} \varphi u}^2_{L^2((0,\thet),\LL)}\right)^{\tfrac{1}{2}}.
  \]
   Here, the right hand side is well defined by Lemma \ref{maptocone} and part \ref{b}.
  \item\label{e} An element  $u \in \EE$ lies  in $\dom(\overline\Brue)$ if and only if we have
  \begin{enumerate}[label=\textup{(\roman*)}] 
  \item (Interior regularity)  $(1-\varphi)u \in \dom(\bar \Brue)$.
  \item (Cone regularity) \label{reg_cone} $\varphi u \in \dom (\bar \Brue_{\cone})$. 
  \end{enumerate}
  \label{item:RegularSingular3}
  \end{enumerate}
Thus, $\dom(\bar \Brue)$ is independent of $S_1$ and  the graph norm on $\dom(\bar \Brue)$ is independent of $S_1$ up to equivalence.
  \begin{enumerate}[label=\textup{(\alph*)}]  
    \setcounter{enumi}{4}
    \item  \label{fred} Assume that, for all  $\varphi,\psi\in C^\infty_{\cc,0}([0,\vartheta);\RR)$ such that $\psi|_{\supp \varphi}  \equiv 1$, there exists an interior parametrix of $\bar \Brue$ for $\varphi$ and $\psi$. 
    Then $\bar \Brue$ is Fredholm.
    \end{enumerate}

\end{theorem}

\begin{remark} \label{absorb} Given an abstract cone operator $\Brue$ and some $0 < \thet' < \thet$, we can absorb arbitrary large portions of the cone part into the bulk part as follows. 
Let $\chi \colon (0,\thet) \to \{ 0,1\}$ be the characteristic map of $(0,\thet')$. 
Then $\chi$ and $1-\chi$ act on $\EE = \EE_{\bulk} \oplus \EE_{\cone} = \EE_{\bulk} \oplus L^2((0,\thet), \LL)$ by $\chi (u_b, u_c) :=  (0 , \chi u_c)$ and $(1-\chi)(u_b, u_c) := (u_b, (1-\chi) u_c)$.
Similarly on $\FF$. 
Then the operator  $\Brue \colon \EE \to \FF$ together with the decompositions 
\begin{align*}
    \EE = \EE'_{\bulk} \oplus \EE'_{\cone}, & \qquad \EE'_{\bulk} := (1-\chi) \EE, \quad \EE'_{\cone} := \chi  \EE, \\
    \FF = \FF'_{\bulk} \oplus \FF'_{\cone}, & \qquad \FF'_{\bulk} :=   (1-\chi) \FF,\quad \FF'_{\cone} :=  \chi  \FF, \\
\end{align*}
is an abstract cone operator.
Thus a version of Theorem \ref{thm:RegularSingular} still holds if \ref{AnSetup6} and \ref{boundedperturbtwo} are only satisfied with $(0,\thet)$ replaced by $(0,\thet')$ for some $0 < \thet' < \thet$.
\end{remark}

\subsection{Cone parametrix} \label{cone_para}  

Let $0 < \thet \leq 1$, let $\LL$ be a separable Hilbert space, let $S_0$ be an essentially self-adjoint unbounded linear operator on $\LL$ and let 
 \[
    S_1 \colon (0,\thet) \to\mathcal L(\dom(S_0),\LL)
\]
be an essentially bounded measurable map.
Let $\bar S_1 \colon (0,\thet) \to\mathcal L(\dom(\bar S_0),\LL)$ be the essentially bounded map induced by $S_1$.
Furthermore, assume that
 \[
          \sigma(\bar S_0) \cap [-\tfrac{1}{2}, \tfrac{1}{2}] = \emptyset
 \]
 and that  $\bar S_0^{-1} \bar S_1 \colon (0, \thet) \to \boundedops(\dom(\bar S_0) )$ induces an essentially bounded map $\bar S_0^{-1} \bar S_1 \colon (0,\thet) \to \boundedops(\LL)$.

Define an unbounded, densely defined, closable linear operator
\[
    \Keg \colon L^2((0,\thet) , \LL) \supset  C^\infty_\cc((0,\thet),\dom(S_0)) \to L^2((0,\thet), \LL) 
\]
by
\begin{equation} \label{prod_form}
     \Keg  u(r) =\partial_ru(r)+\tfrac{1}{r}\left(S_0+S_1(r) \right)u(r) . 
\end{equation}
Let $\bar \Keg \colon  L^2((0,\thet) , \LL) \supset  \dom ( \bar \Keg) \to L^2((0,\thet), \LL)$  be its closure.

The operator $\Keg$ models the operator $\Brue_{\cone}$ for an abstract cone operator $\Brue$.
However, in this subsection, we work without assumption \ref{boundedperturbtwo}.

\begin{lemma}\label{lem:extend_domain}
We have
\[
     C_\cc^1((0,\thet), \dom( \bar S_0)) \subset \dom( \bar \Keg) 
\]
and the inclusion 
\[
     C_\cc^{\infty}((0,\thet), \dom( S_0)) \subset  C_\cc^1((0,\thet), \dom( \bar S_0))
\]
is dense with respect to the graph norm of $\bar \Keg$.

Furthermore, working with $\bar \Keg$, $\bar S_0$ and $\bar S_1$,  the formula \eqref{prod_form} holds for all $u \in C_\cc^1((0,\thet), \dom( \bar S_0))$.
\end{lemma}

\begin{proof}
Let $u \in C_\cc^1((0,\thet), \dom( \bar S_0))$.
We must show that there exists a sequence  $u_k \in C^\infty_{\cc}((0,\thet),\dom(S_0))$ which converges in $L^2((0,\thet), \LL)$  to $u$ and such that $\Keg u_k$ converges in $L^2((0,\thet), \LL)$ to the function $\mathscr{K} u \colon (0,\thet) \to \LL$, $r \mapsto \partial_ru(r)+\frac{1}{r}(\bar{S}_0+\bar S_1(r))u(r) $.

Let  $0 < c < \tfrac{1}{2}$ such that $\supp(u)\subset [c,1-c]$, and choose a piecewise linear compactly supported function $U\colon (0,\thet)\to  \dom(\bar S_0)$ with $\supp(U)\subset [c,1-c]$  that approximates $u' \in C_\cc^0((0,\thet), \dom(\bar S_0))$ in $C^0$-norm and such that $\int_0^1 U(r)\,dr =0$.
 This  is possible as  $\int_0^1u'(r)\,dr =0$.
 By construction, $U$ takes values in a finite dimensional subspace $V\subset\dom(\bar S_0)$.

Then $v_1 \colon (0,\thet) \to \dom ( \bar S_0)$, $v_1(r):=\int_0^r U(t)\,dt$, is a  $C^1$-function with values in a compact subset $K \subset V\subset \dom(\bar S_0)$ such that $\supp(v_1)\subset [c,1-c]$ and such that $v_1$ approximates $u$ in $C^1$-norm.

Since $\dom(S_0)$ is dense in $\dom(\bar S_0)$, for all $\eps > 0$, there is a linear map $\pi \colon V \to \dom(S_0)$ such that $\abs{\pi(x)-x}_{\dom(\bar S_0 )} < \eps$ for all $x\in K$. 
Then $v_2 :=\pi \circ v_1\in C^1( (0,\thet) , \dom(S_0))$ is a  $C^1$-function which approximates $u$ in $C^1$-norm and such that  $\supp(v_2)\subset [c,1-c]$.

The image of $v_2$ is contained in $\pi(V) \subset \dom(S_0)$. 
Since $\pi(V)$ is finite-dimensional it is a (complete) Hilbert space.
By convolution with mollifiers, we find $v_3\in C^\infty_{\cc}((0,\thet),\pi(V))$ which approximates $v_2$, hence $u$, in $C^1$-norm and such that  $\supp(v_3)\subset [c/2,1-c/2]$.

Summarizing, we find a sequence $u_k \in C^\infty_{\cc}((0,\thet),\dom(S_0))$ supported in  $ [c/2,1-c/2]$ and such that $u_k \longrightarrow u$ in $C^1((0,\thet), \dom(\bar S_0))$.
In particular, we get $u_k \longrightarrow u$ in $L^2((0,\thet), \LL)$. 

Furthermore, 
 \begin{align*}
   \| \Keg (u_k) & - \mathscr{K} u\|_{L^2((0,\thet), \LL)}^2  =  \int_0^{\thet} \abs{\partial_r(u_k - u)(r)+\frac{1}{r}(\bar S_0+ \bar S_1(r))(u_k-u)(r)}_{\LL}^2\,dr\\
      & \le \left(\sup_{r \in (0,\thet)}\norm{\partial_r (u_k-u)(r)}_{\LL} + \frac{2}{c}\left(\sup_{r \in (0,\thet)} \norm{(\bar S_0 + \bar S_1(r))(u_k-u)(r)}_{\LL}   \right)\right)^2\\
      & \le \frac{4}{c^2} \left( \| u_k - u \|_{C^1}  + \left( \| \bar S_0\|_{\boundedops(\dom(\bar S_0), \LL)} + \|\bar S_1 \|_{L^{\infty}((0,\thet), \boundedops(\dom(\bar S_0), \LL))}\right) \| u_k - u\|_{C^0} \right)^2\\
       & \le\frac{4}{c^2}   \left( 1 + \| \bar S_0\|_{\boundedops(\dom(\bar S_0), \LL)} + \|\bar S_1 \|_{L^{\infty}((0,\thet), \boundedops(\dom(\bar S_0), \LL))} \right)^2 \| u_k - u\|^2_{C^1} .
      \end{align*}
  Hence $\Keg (u_k) \longrightarrow \mathscr{K} u$ in $L^2((0,\thet), \LL)$.
  This shows our claim.
  \end{proof}

By Lemma \ref{lem:extend_domain}, we can assume, without altering the closed extension $\bar \Keg$,  that $S_0$ is self-adjoint.
In a first step we construct a solution operator for the first order differential operator $\partial_r + \tfrac{1}{r} S_0$, see Proposition \ref{difficult}.
This relies on the spectral decomposition of the self-adjoint operator $S_0$.

For $s > - \frac{1}{2}$ consider the Hilbert-Schmidt integral kernel $K_{0,s} \in L^2((0,1)^2, \R)$,
\begin{equation}
     \label{kernel0}  K_{0,s}(r,y)  =  \begin{cases}  \big( \frac{y}{r}\big)^s,  &  0\le\frac{y}{r}\le 1, \\
                                                   0, & \text{otherwise},
                                                      \end{cases}
\end{equation}
and the corresponding integral operator  $\PP_{0,s} \in \boundedops(L^2((0,1), \CC))$, 
\[
     \PP_{0,s}(\omega)(r) = \int_0^1 K_{0,s}(r,y) \omega (y) \, d y = \int_0^r \left(\frac{y}{r}\right)^s \omega(y) \,d y.
\]
In a similar fashion,  for $s < \tfrac{1}{2}$, consider the integral kernel  $K_{1,s} \in L^2((0,1)^2, \R)$,                                            
\begin{equation} 
    \label{kernel1}  K_{1,s}(r,y)  = \begin{cases}  \big(\frac{y}{r}\big)^s,  &  0\le\frac{r}{y}\le 1, \\
                                                   0, & \text{otherwise},
                                                   \end{cases}
\end{equation} 
and the corresponding integral operator $\PP_{1,s} \in \boundedops( L^2((0,1)^2, \R))$, 
\[
     \PP_{1,s} \omega(r) :=  - \int_0^1 K_{1,s}(r,y) \, d y = \int_{1}^{r} \left( \frac{y}{r}\right)^s \omega(y) \, d y . 
\]

\begin{lemma}\label{elementary} 
  Let $\omega \in C_\cc^{0}((0,1), \CC)$ and $\nu\in C_\cc^1((0,1),\CC)$.
	\begin{enumerate}[start=1,label=\textup{(\roman*)}]
		\item \label{ODE1} 
          For $s > - \tfrac{1}{2}$, we get $\PP_{0,s} \omega \in C^1((0,1), \CC)$ and
		      \begin{equation*}
			      \left(\partial_r + \tfrac{s}{r} \id\right) \PP_{0,s} \omega = \omega.
		      \end{equation*}
          For $s>\tfrac12$ we get
          \begin{equation*}
            \PP_{0,s}\left(\partial_r + \tfrac sr \id\right)\nu = \nu.
          \end{equation*}
		\item \label{ODE2} 
          For  $s  < \tfrac{1}{2}$, we get $\PP_{1,s} \omega \in C^1((0,1), \CC)$ and 
		      \begin{equation*}
			      \left(\partial_r + \tfrac{s}{r} \id\right) \PP_{1,s} \omega = \omega.
		      \end{equation*}
          For $s < - \tfrac12$ we get
          \begin{equation*}
            \PP_{1,s}\left(\partial_r + \tfrac sr \id\right)\nu = \nu.
          \end{equation*}
		\item\label{ODE3}
          Let $K \subset (0,1)$ be a compact subset.
		      Then there is a constant $C = C(K) > 0$ such that, on $K$,
		      \begin{align*}
			       & |\PP_{0,s} \omega| \leq C \textrm{ and } |  (\PP_{0,s} \omega)'| \leq C   \quad \textrm{for} s > - \tfrac{1}{2} , \\
			       & |\PP_{1,s} \omega| \leq C \textrm{ and }  | (\PP_{1,s} \omega)' | \leq C \quad  \textrm{for}  s < \tfrac{1}{2}  .
		      \end{align*}
	\end{enumerate}
\end{lemma}

\begin{proof} The assertions \ref{ODE1} and \ref{ODE2} follow from standard integration theory.

Let $r \in (0,1)$, then we obtain 
\begin{align*}
  |\PP_{1,s} \omega(r)| & \leq  \tfrac{1}{r^s} \|\omega\|_{\infty} \int_r^1 y^s d y   = \|\omega\|_{\infty} \tfrac{1}{r^s} \abs{ \tfrac{1 - r^{s+1}}{s+1} }  , \quad s <  \tfrac{1}{2} , s \neq -1, \\
  |\PP_{1,s} \omega(r)| & \leq   \|\omega\|_{\infty} \int_r^1 y^{-1} d y   = \|\omega\|_{\infty} \abs{\log(r)}   , \quad s = -1,   \\
  |\PP_{0,s} \omega(r)| & \leq   \tfrac{1}{r^s} \| \omega \|_{\infty} \int_0^r y^s d y =  \| \omega \|_{\infty}  \left| \tfrac{1}{s+1} \right|   , \quad s > - \tfrac{1}{2}, \\
  \end{align*}
and 
\begin{align*}
               \abs{\tfrac1r s\PP_{1,s}\omega(r)} & \leq \tfrac{1}{r^{s+1}} \norm{\omega}_{\infty} \abs{\int_r^1 sy^sd y} \leq \norm{\omega}_{\infty}  \abs{ \tfrac{s}{r^{s+1}} }  \abs{\tfrac{1-r^{s+1}}{s+1}}, \quad s < \tfrac{1}{2}, s \neq -1 , \\
              \abs{\tfrac1r s\PP_{1,s}\omega(r)} & \leq \norm{\omega}_\infty \int_{r}^1 y^{-1} d y = \norm{\omega}_\infty \abs{\log(r)},  \quad s = -1, \\
            \abs{\tfrac1r s\PP_{0,s}\omega(r)} & \leq \tfrac{1}{r^{s+1}}\norm{\omega}_\infty \abs{\int_0^r sy^sd y} \leq \norm{\omega}_\infty\abs{\tfrac{s}{s+1}},  \quad  s > - \tfrac{1}{2} . 
\end{align*}
For $r$ contained in a compact subset of $(0,1)$, the right hand sides are uniformly bounded for all $s < \tfrac{1}{2}$, respectively for all $s > - \tfrac{1}{2}$.
In combination with \ref{ODE1} and \ref{ODE2}, we obtain  \ref{ODE3}.
\end{proof} 

In the following, $\tfrac{1}{r}$ denotes multiplication with the function $(0,1) \to \RR$, $r \mapsto \tfrac{1}{r}$, considered as a $\CC$-linear endomorphism on the space of maps $(0,1) \to \CC$.

\begin{lemma}\label{lem:ConeParamEstimates}
Left and right composition of $\PP_{0,s}$ and $\PP_{1,s}$ with $\frac{1}{r}$ define bounded operators on $L^2((0,1),\complexs)$ that satisfy the following operator norm estimates.
    \begin{enumerate}[label=\textup{(\roman*)}]
      \item \label{1} $\norm{ \tfrac1r \PP_{0,s} } \leq \abs{s+\tfrac12 }^{-1} \quad \textrm{for} s> -\tfrac12$ ,
      \item \label{2} $\norm{ \PP_{0,s} \tfrac1r } \leq \abs{s-\tfrac12 }^{-1}  \quad \textrm{for}  s >  \tfrac12$ , 
      \item \label{3} $\norm{ \tfrac1r \PP_{1,s} }  \leq \abs{s+\tfrac12 }^{-1} \quad \textrm{for}   s<  -\tfrac12$ , 
      \item \label{4} $\norm{  \PP_{1,s} \tfrac1r } \leq \abs{s-\tfrac12 }^{-1} \quad \textrm{for}   s< \tfrac12$.
    \end{enumerate}
 \end{lemma}
 
\begin{proof}
  Let  $p(y)= \frac{1}{\sqrt{y}}$ and $q(r) = \frac{1}{\sqrt{r}}$.
  We compute
  \begin{equation*}
     \int_0^1  \frac{1}{r} K_{0,s}(r,y)  p(y)d y  = \frac{1}{r^{s+1}}\int_0^r y^{s-\frac12} d y = \frac{y^{s+\frac12}}{\left( s+\tfrac12 \right)r^{s+1}}\Big|_0^r = \frac{1}{s+\tfrac12} q(r)
  \end{equation*}
  and 
  \[
    \int_0^1 \frac{1}{r} K_{0,s}(r,y) q(r)  d r  = y^s \int_y^1 r^{-s-\frac32} d r  = \frac{y^s}{(-s-\frac12)r^{s+\frac12}}\Big|_y^1 =  -\frac{y^s}{s+\frac12}\left( 1 - \frac{1}{y^{s+\frac12}} \right) \stackrel{s > - \tfrac{1}{2}}{\leq}  \frac{1}{s+\frac12} p(y).
  \]
 From this, assertion \ref{1}  follows from Schur's test  (see \cite{Halmos_Sunder}*{Theorem 5.2}). 
 Assertion \ref{3} is proved similarly, working with the same test functions $p(y) = \frac1{\sqrt{y}}$ and $q(r) = \frac1{\sqrt{r}}$.
 We have $P_{0,s}\tfrac{1}{r} = \tfrac{1}{r}P_{0,s-1}$ for $s > \tfrac{1}{2}$ and $P_{1,s}\tfrac{1}{r} = \tfrac{1}{r} P_{1, s-1}$ for $s< \tfrac{1}{2}$.
 From this, the remaining assertions \ref{2} and \ref{4} follow.
\end{proof}

Consider  the decomposition into positive and negative part of the spectrum
  \[ 
     \spec_-(S_0) := \spec(S_0)\cap (-\infty,-\tfrac{1}{2}) , \qquad \spec_+(S_0):= \spec(S_0)\cap (\tfrac{1}{2},\infty) .
  \]
Given $\omega \in C_{\cc}^0((0,1), \CC)$, we define $\mathcal{\PP} \omega  \colon \sigma(S_0)  \times (0,1) \to \CC$,
\begin{equation} \label{def_PP}
       \PP\omega(s,r) := \begin{cases} \PP_{1,s}\omega (r),  & s \in \sigma_-(S_0) , \\
                                                \PP_{0,s}\omega(r),  & s \in \sigma_+(S_0) . 
                                       \end{cases}
\end{equation}
Since  the integral kernels $K_{0,s}, K_{1,s} \in L^{2} ( (0,1)^2, \R)$ from \eqref{kernel0} and  \eqref{kernel1} depend continuously on $s$, the map $\PP\omega$ is continuous.
Furthermore, by Lemma \ref{elementary} \ref{ODE3}, for each compact subset $K \subset (0,1)$, the map $\PP\omega$  is bounded on $\sigma(S_0) \times K$.
Hence, for each $r \in (0,1)$, functional calculus yields a bounded linear operator $\PP\omega(S_0,r) \colon \LL \to \LL$, and the map $(0,1) \to \boundedops(\LL)$, $r \mapsto \PP\omega(S_0,r)$, is continuous.

\begin{lemma} \label{lengthy}  There exists a unique bounded linear operator $\PPP \colon L^2((0,1), \LL) \to L^2((0,1), \LL)$ such that for all $\omega \in C_{\cc}^0((0,1), \CC)$ and $\xi \in \LL$, we have 
\[
   \PPP (\omega \otimes \xi) (\cdot ) =   \PP\omega(S_0,\cdot) (\xi) , \quad a.e.
\]
\end{lemma} 

\begin{proof} 
We use the spectral theorem in the form of a  direct integral decomposition, see \cite{Hall}*{Theorem 10.9}, writing 
\begin{equation*}
  \LL = \int_{\spec(S_0)} \LL_s\,d\mu(s),  \qquad S_0 = \int_{\spec(S_0)} s \cdot \id_{\LL_s}\,d\mu(s) .
\end{equation*}

As a preparation, for each separable Hilbert space $\HH$ and $s \in \sigma(S_0)$, we obtain a bounded linear operator 
\[
    \mathcal{P}_s \colon L^2((0,1), \HH) \to L^2((0,1), \HH)
\]
as follows.
Let $(e_i)_{i \in \NN}$ be an orthonormal basis of $\HH$. 
Then, for $\omega \in L^2((0,1), \HH)$, we put $\omega_i := \langle \omega, e_i\rangle \in L^2((0,1), \CC)$ and set
\[ 
      \mathcal{P}_s (\omega) (r) := \sum_{i=1}^n \mathcal{P}\omega_i(s,r)  \, e_i . 
\]
This definition is independent of the choice of $(e_i)$. 
Furthermore, by Lemma \ref{lem:ConeParamEstimates} \ref{1} and \ref{3},  the operator norm of $\mathcal{P}_s$ satisfies
\begin{equation} \label{operatornormbound}
    \norm{\mathcal{P}_s}_{L^2} \leq |s + \tfrac{1}{2}|^{-1} .
 \end{equation}

Now, for
\[
    u = u(s)   \in  L^2((0,1), \LL)  = L^2\left( (0,1) , \int_{\sigma(S_0)} \LL_s \, d\mu(s) \right) =  \int_{\spec(S_0)} L^2((0,1), \LL_s)  \, d \mu(s),
\]
we define the measurable family $\PPP(u) \in \big( L^2((0,1), \LL_s)\big)_{s \in \sigma(S_0)}$ by
\[
     \PPP(u)(s)  :=   \mathcal{P}_s (u(s)) .
\]
The $L^2$-norm of $\PPP(u)$ is estimated as 
\[
   \|  \PPP(u) \|_{L^2}^2 = \int_{\spec(S_0)} \left| \mathcal{P}_s(u(s))\right|^2   d\mu(s) \stackrel{\eqref{operatornormbound}}{\leq}     \sup_{s \in \sigma(S_0)} |s + \tfrac{1}{2}|^{-2} \| u\|_{L^2}^2.  
 \]
Thus $\PPP$ defines a bounded linear map $L^2((0,1), \LL) \to L^2((0,1), \LL)$ with $\norm{\PPP}_{L^2} \leq \sup_{s \in \sigma(S_0)} |s + \tfrac{1}{2}|^{-2}$.

For $\omega \in L^2((0,1), \CC)$, $\xi = \xi(s) \in \LL = \int_{\sigma(S_0)} \LL_s \, d\mu(s)$,  and almost all $r \in (0,1)$, we have
\[
   \PPP(u) (r) =  \PPP (\omega \otimes \xi)(r) =   \int_{\sigma(S_0)} \mathcal{P}_s(\omega  \otimes \xi(s)) (r) \, d\mu(s)  =  \PP\omega(S_0, r ) (\xi) .
\]
The uniqueness of $\PPP$ follows from the fact that the algebraic tensor product $C^0_{\cc}((0,1), \CC) \otimes \LL$ is dense in $L^2((0,1), \LL)$.
This completes the proof of  Lemma \ref{lengthy}.
\end{proof}

\begin{proposition} \label{difficult} 
Let $u \in  C_{\cc}^0((0,1), \CC)\otimes \dom(S_0)$. 
Then 
	\begin{enumerate}[label=\textup{(\roman*)}]
    \item \label{oans} $\PPP u \in C^1((0,1), \dom(S_0))$ and $\left(\partial_r + \tfrac1r S_0\right) \PPP  u  = u$.
    \item \label{zwoa} $   \PPP\left(\partial_r + \tfrac1r S_0\right) u = u$. 
  \end{enumerate}
\end{proposition}

\begin{proof}
 By linearity, it suffices to show the claims for simple tensors $u = \omega \otimes \xi \in   C_{\cc}^0((0,1), \CC)\otimes \dom(S_0)$.
 
 We start with the proof of \ref{oans}.
Functional calculus applied to the measurable function $s \PP\omega \colon \sigma(S_0) \times (0,1) \to \CC$, $(s,r) \mapsto s \PP\omega(s,r)$, yields, for each $r \in (0,1)$, an unbounded operator 
\[
   (s\PP\omega)(S_0,r) \colon \LL \supset \dom(S_0)  \to \LL. 
\]
Since $\xi \in \dom(S_0)$, we have $\PPP  u(r) = \PP\omega(S_0, r)\xi  \in \dom(S_0)$ and 
\begin{equation} \label{funny2} 
    S_0 \PPP  u(r) = S_0( \PP\omega(S_0,r) \xi) = (s\PP\omega)(S_0,r)\xi .
\end{equation}
Define $\partial_r \PP\omega \colon \sigma(S_0) \times (0,1) \to \CC$,
\begin{equation} \label{defdel}
       \partial_r \PP\omega (s,r) := \begin{cases} (\PP_{1,s}\omega)' (r) =\omega(r) - \tfrac{s}{r} \PP_{1,s}\omega(r),  & s \in \sigma_-(S_0) , \\
                                                (\PP_{0,s}\omega)'(r) = \omega(r)-\tfrac{s}{r} \PP_{0,s} \omega(r),  & s \in \sigma_+(S_0) . 
                                       \end{cases}
\end{equation}
Since  the integral kernels $K_{0,s}, K_{1,s} \in L^{\infty} ( (0,1)^2, \R)$ from \eqref{kernel0} and  \eqref{kernel1} depend continuously on $s$,  Lemma \ref{elementary} \ref{ODE1} and \ref{ODE2} imply that the function $\partial_r \PP\omega$ is continuous.
Furthermore, by Lemma \ref{elementary} \ref{ODE3}, for each compact subset $K \subset (0,1)$,  the function $\partial_r \PP\omega$  is bounded on  $\sigma(S_0) \times K$.  
Hence, by  the dominated convergence theorem, we obtain $\PPP  u \in C^1((0,1),  \dom(S_0))$ and
\begin{equation} \label{funny3}
    \partial_r \PPP u(r) = \partial_r \PP\omega ( S_0,r) \xi  . 
\end{equation} 
Combining the previous computations, for $r \in (0,1)$, we obtain
\[
   ( \partial_r + \tfrac{1}{r} S_0) \PPP u(r) \stackrel{\eqref{funny2}, \eqref{funny3}}{=} \partial_r \PP\omega(S_0,r)\xi + \tfrac{1}{r} ( s\PP\omega)(S_0,r)\xi   = \big(    \partial_r \PP\omega    + \tfrac{s}{r} \PP\omega \big) (S_0, r)  \stackrel{\eqref{defdel}}{=}      \omega \otimes \xi = u(r) .  
\]
This finishes the proof of \ref{oans}.

For \ref{zwoa}, let $u = \omega \otimes \xi \in C_{\cc}^1((0,1), \CC)\otimes \dom(S_0)$. 
 Define $\PP(\partial_r + \frac{s}{r}) \omega \colon \sigma(S_0) \times (0,1) \to \CC$,
\begin{equation*} 
       \PP (\partial_r + \tfrac{s}{r})\omega (s,\cdot) := \begin{cases} \left( \PP_{1,s} \partial_r \omega + \PP_{1,s} (\tfrac{s}{r} \omega)\right) (\cdot) , & s \in \sigma_-(S_0) , \\
                                                 \left(  \PP_{0,s} \partial_r \omega +  \PP_{0,s} ( \tfrac{s}{r} \omega) \right) (\cdot),  & s \in \sigma_+(S_0) . 
                                       \end{cases}
\end{equation*}
 
By Lemma \ref{elementary}, we have  $\PP (\partial_r + \tfrac{s}{r})\omega (s,\cdot) = \omega(\cdot)$.
Furthermore, by functional calculus,  for $r \in (0,1)$ we get
\[
    \PP (\partial_r + \tfrac{s}{r})\omega (S_0,r) \xi = \partial_r \omega\otimes \xi  + \tfrac1r\omega\otimes S_0\xi =   \left(\partial_r + \tfrac1r S_0\right) u . 
\]
Summarizing, 
  \begin{align*}
    \PPP\left(\partial_r + \tfrac1r S_0\right) u(r) = \left( \PP (\partial_r + \tfrac{s}{r})\omega \right)(S_0,r) \, \xi  = \omega(r) \cdot \xi = (\omega \otimes \xi)(r). 
  \end{align*}
We conclude $\PPP\left(\partial_r + \tfrac1r S_0\right) u = u$, as required.
\end{proof}

\begin{lemma}\label{lemma:BoundOnS0Param}
\begin{enumerate}[start=1,label=\textup{(\roman*)}]
  \item \label{um} There exists  a unique bounded linear operator 
  \[
       \left\{ \tfrac1r S_0\PPP \right\} \colon L^2((0,1), \LL) \to L^2((0,1), \LL)
  \]
  with the following property: 
  Let $u \in C^{0}_{\cc}((0,1), \CC) \otimes \dom(S_0)$. 
  Then the map $\frac{1}{r} S_0 (\PPP u) \colon (0,1) \to \LL$,  which is well defined by Proposition \ref{difficult}, satisfies
  \[
      \tfrac1r S_0 (\PPP u ) =  \left\{ \tfrac1r S_0\PPP \right\} (u)   \quad a.e. 
  \]
Furthermore, we have
\[
   \norm{   \left\{ \tfrac1r S_0\PPP \right\}}_{L^2}  \leq \sup_{ s\in\spec(S_0)} \frac{|s|}{|s + \frac12|}.
\]
 \item \label{dois} There exists  a unique bounded linear operator 
 \[
        \left\{ \PPP \tfrac{1}{r} S_0  \right\} \colon L^2((0,1), \LL) \to L^2((0,1), \LL) 
 \]
 with the following property: 
 Let  $u\in C^{0}_{\cc}((0,1), \CC) \otimes \dom(S_0)$. 
 Then the map  $\tfrac{1}{r} S_0 u \in C^0_{\cc}( (0,1) , \LL)$,  $r \mapsto \tfrac{1}{r}S_0( u(r))$, satisfies
 \[
    \PPP(\tfrac{1}{r} S_0 u) = \left\{ \PPP \tfrac{1}{r} S_0  \right\}(u)      \quad a.e. 
\] 
Furthermore, we have
\[
    \norm{  \left\{ \PPP \tfrac{1}{r} S_0  \right\} }_{L^2} \leq \sup_{s\in \spec(S_0)}  \frac{|s|}{|s - \frac12|} .
\]
\end{enumerate}
\end{lemma}

\begin{proof}
Let 
\[
    u = u(s)   \in  L^2((0,1), \LL)  \in   \int_{\spec(S_0)} L^2((0,1), \LL_s)  \, d \mu(s).
\]

For \ref{um}, we define the measurable family $ \left\{ \tfrac1r S_0\PPP \right\}(u) \in ( L^2((0,1), \LL_s))_{s \in \sigma(S_0)}$ by
\[
    \left\{ \tfrac1r S_0\PPP \right\}(u)(s) :=     \tfrac{1}{r} \,s \,  \mathcal{P}_s  (u(s)) .
\]
Using Lemma \ref{lem:ConeParamEstimates} \ref{1} and \ref{3}, we obtain 
\[
   \| \left\{ \tfrac1r S_0\PPP \right\}(u)  \|_{L^2}^2 = \int_{\spec(S_0)} |s|^2 \left| \tfrac{1}{r}\mathcal{P}_s(u(s))\right|^2   d\mu(s)   \leq \left(\sup_{ s\in\spec(S_0)} \frac{|s|}{|s + \frac12|}\right)^2 \|u\|_{L^2}^2.
\]
Hence  $\left\{ \tfrac1r S_0\PPP \right\}$ defines a bounded linear operator on $L^2((0,1), \LL)$.
The remaining assertions of \ref{um} are immediate.

For \ref{dois}, we define the measurable family $ \left\{ \PPP \tfrac{1}{r} S_0  \right\} (u) \in ( L^2((0,1), \LL_s))_{s \in \sigma(S_0)}$ by
\[
   \left\{ \PPP \tfrac{1}{r} S_0  \right\} (u)(s)   :=      \,s \,  \mathcal{P}_s  (\tfrac{1}{r} u(s)) .
\]
Using Lemma \ref{lem:ConeParamEstimates} \ref{2} and \ref{4}, we obtain 
\[
   \| \left\{ \tfrac1r S_0\PPP \right\}(u)  \|_{L^2}^2 = \int_{\spec(S_0)} |s|^2 \left|\mathcal{P}_s ( \tfrac{1}{r}u(s))\right|^2   d\mu(s)   \leq \left(\sup_{ s\in\spec(S_0)} \frac{|s|}{|s - \frac12|}\right)^2 \|u\|_{L^2}^2.
\]
Hence  $\left\{ \tfrac1r S_0\PPP \right\}$ defines a bounded linear operator on $L^2((0,1), \LL)$.
The remaining assertions of \ref{dois} are immediate.
\end{proof}

Lemma \ref{lemma:BoundOnS0Param} has the following immediate consequence for the operator $S_1$.

\begin{lemma}\label{cor:BoundOnParamS1}
Define non-negative real numbers
\begin{align*}
     C_1 & :=  \norm{S_1 S_0^{-1}}_{L^{\infty}((0,\thet), \boundedops(\LL))}  \cdot \sup_{ s\in\spec(S_0)} \frac{|s|}{|s + \frac12|}, \\
     C_2 & := \norm{S_0^{-1} S_1}_{L^{\infty}((0,\thet), \boundedops(\LL)))}  \cdot \sup_{s\in \spec(S_0)}  \frac{|s|}{|s - \frac12|} .
\end{align*}
\begin{enumerate}[start=1,label=\textup{(\roman*)}]
  \item \label{egy} Consider the bounded operator 
 \[
      \left\{ \tfrac1r S_1\PPP \right\} := S_1 S_0^{-1} \circ  \left\{ \tfrac1r S_0 \PPP \right\} \colon L^2((0,\thet), \LL) \to L^2((0,\thet), \LL).
 \]
 Then  $\norm{\left\{ \tfrac1r S_1\PPP \right\}}_{L^2} \leq C_1$. 
 Furthermore, the following holds: Let $u  \in C^0_{\cc}((0,\thet), \CC) \otimes \dom(S_0)$ and consider the map $ \tfrac{1}{r} S_1 (\PPP u) \colon (0,\thet) \to \LL$ which is  well defined by Proposition \ref{difficult}.
Then 
\[
     \tfrac{1}{r} S_1 (\PPP u) =  \left\{ \tfrac1r S_1\PPP \right\} (u)  \qquad a.e.
 \]
 \item \label{ket} Consider the bounded operator 
 \[
       \left\{ \PPP \tfrac{1}{r} S_1  \right\} := \left\{ \PPP \tfrac{1}{r} S_0  \right\} \circ S_0^{-1} S_1 \colon L^2((0,\thet), \LL) \to L^2((0,\thet), \LL) .
 \]
Then  $\norm{\left\{ \PPP \tfrac{1}{r} S_1  \right\}}_{L^2} \leq C_2$. 
 Furthermore, the following holds: Let $u \in C^0_{\cc}((0,\thet), \CC) \otimes \dom(S_0)$ and consider the map $ \tfrac{1}{r} S_1 u  \in L^2((0,\thet), \LL)$. 
Then 
\[
    \PPP( \tfrac{1}{r} S_1 u) = \left\{ \PPP \tfrac{1}{r} S_1  \right\} (u) \qquad a.e. 
\]
\end{enumerate}
\end{lemma}

\begin{lemma} \label{little_shit} 
The operator 
\[
      \left\{\tfrac{1}{r}\PPP\right\} := S_0^{-1} \circ  \left\{\tfrac{1}{r} S_0 \PPP\right\} \colon L^2((0,1), \LL) \to L^2((0,1), \LL) 
 \]
 is bounded. 
 Furthermore $\PPP = r   \left\{\tfrac{1}{r}\PPP\right\}$.
\end{lemma}

\begin{proof} The first assertion holds since $S_0^{-1} \colon \LL \to \dom(S_0) \to \LL$ is bounded.
The second claim is obvious for  $u \in C_c((0,1), \CC) \otimes \dom(S_0)$ and follows in general since $C_c((0,1), \CC) \otimes \dom(S_0) \subset L^2((0,1), \LL)$ is dense.
\end{proof}

\begin{lemma}\label{lem:DefFunction}
For each $0 < \eps < 1$,  there is a $C^\infty$-function $\tau_{\eps}\in C_{\cc,0}^{\infty} (  [0,\infty), [0,1])$ such that
\[ 
   \supp(\tau_{\eps}) \subset [0,\eps], \qquad  \sup_{ r \in (0,\eps) } \big|r \cdot  \tau'_{\eps}(r)\big| \le \eps .
\]
\end{lemma} 

\begin{proof} By \cite{BaerHanke}*{Lemma 2.8}, for each $0 < \delta < \tfrac{1}{4}$ and $0 < \eps < 1$, there exists a $C^{\infty}$-function $\tau_{\delta, \eps} \colon \R \to \R$ with the following properties:
\begin{enumerate}[label=\textup{(\roman*)}]
\item\label{eq:tauepsi}
$\tau_{\delta,\eps}(r)  =1$ for $r\le \delta \eps$;
\item\label{eq:tauepsii}
$\tau_{\delta , \eps}(r)  =0$ for $r\ge\eps $;
\item
$0\le \tau_{\delta, \eps} \le 1$ everywhere;
\item\label{eq:tauepsiv}
there is a constant $C_1 >0$ such that $\big|\tau_{\delta, \eps}'(r)\big| \leq C_1 \cdot r^{-1} \cdot | \ln \delta|^{-1}$ for all $r > 0$.
\end{enumerate}
Finally, set  $\tau_{\eps} := \tau_{\delta, \eps}$ with $\delta = \min\{ \exp(-C_1/ \eps), \tfrac{1}{8}\}$.
\end{proof}

\begin{proposition}\label{lemma:coneParametrix} 
    Let $\psi\in C_\cc^\infty([0,\thet), \RR)$. 
    Then, for every $u\in L^2((0,\thet),\LL)$, we have $\psi \PPP u \in \dom(\bar \Keg )$ and
    \begin{equation*}
      \bar \Keg \psi \PPP u = \psi \cdot u + \psi \cdot  \left\{ \tfrac{1}{r}S_1 \PPP  \right\} u +  \psi'  \PPP  u.
    \end{equation*}
    In particular, $u \mapsto \psi \PPP u$ defines a bounded operator $L^2((0,\thet), \LL) \to \dom(\bar \Keg)$.
\end{proposition}
  
\begin{proof}
    First, let $u \in C_{\cc}^0((0,\thet), \CC) \otimes \dom(S_0)$.
    For $k \geq 1$, put $\chi_k :=1-\tau_{1/k} \colon C^{\infty}((0,\thet), [0,1])$ where $\tau_{1/k}$ is taken from  Lemma \ref{lem:DefFunction}.
    
    Due to  Proposition \ref{difficult} \ref{oans}, we have $\chi_k \psi \PPP u \in C_{\cc}^1((0,\thet), \dom(S_0))$ and $(\partial_r + \tfrac{1}{r} S_0) \PPP u = u$.
   With Lemma \ref{lem:extend_domain}, we get
    \begin{align} 
        \bar \Keg \chi_k \psi \PPP u &=  \left( \partial_r + \tfrac1r S_0\right) \chi_k \psi \PPP u + \tfrac1r S_1(r) \chi_k \psi \PPP u \\
      &= \chi_k \psi u + \chi'_k \psi \PPP u + \chi_k \psi' \PPP  u + \tfrac1r S_1(r) \chi_k \psi   \PPP u    \nonumber \\
      & = \chi_k \psi u + (r \cdot \chi'_k) \psi \left( \tfrac{1}{r} \PPP \right)u + \chi_k \psi' \PPP  u + \chi_k \psi  (\tfrac1r S_1(r) )  \PPP u \nonumber . 
    \end{align}
Taking the limit $k\to\infty$ and using Lemma \ref{little_shit}, we conclude
that $\psi \PPP u \in \dom(\bar \Keg)$ and
 \begin{equation} \label{firstcase} 
    \bar \Keg \psi \PPP u =  \psi u +  \psi' \PPP  u + \psi  \left\{\tfrac1r S_1(r)   \PPP\right\} u .
\end{equation}
     
    Now let $u\in L^2((0,\thet),\LL)$.
   Pick  a sequence $u_n  \in C_\cc^0((0,\thet), \CC) \otimes \dom(S_0)$ such that $u_n \stackrel{n\to  \infty}{\longrightarrow} u$  in $L^2((0,\thet),\LL)$. 
    Since, by Lemma \ref{lengthy} and \ref{cor:BoundOnParamS1}, $\PPP$ and $\left\{\tfrac1r S_1(r) \PPP\right\}$ are bounded operators on $L^2((0,\thet),\LL)$, the sequence
    \[
     \bar \Keg \psi \PPP u_n =  \psi u_n +  \psi' \PPP  u_n + \psi  \left\{\tfrac1r S_1(r)   \PPP\right\} u_n
    \]
    converges in $L^2((0,\thet), \LL)$ to $\psi u +  \psi' \PPP  u + \psi  \left\{\tfrac1r S_1(r)   \PPP\right\} u$. 
    
    The last assertion of Proposition \ref{lemma:coneParametrix} is immediate.
\end{proof}

\begin{proposition}\label{lem:BoundParamMaxExt}
	Let  $u\in \dom(\bar \Keg) \subset L^2((0,\thet), \LL)$.
	Then
	\begin{equation*}
		\PPP  \bar \Keg u =  u + \left\{ \PPP  \tfrac{1}{r} S_1 \right\}u .
	\end{equation*}
\end{proposition}

\begin{proof}
  First, let $u\in C_\cc^\infty((0,\thet), \CC)\otimes \dom(S_0)$. 
  Then we get with Proposition \ref{difficult} \ref{zwoa}
  \begin{equation*}
    \PPP \bar \Keg u = \PPP(\partial_r + \tfrac1rS_0)u + \PPP\tfrac1r S_1 u = u + \PPP \tfrac1r S_1 u.
  \end{equation*}
  Now let $u\in \dom(\bar \Keg)$ and choose a sequence $u_n$ in $C_\cc^\infty((0,\thet),\CC)\otimes \dom(S_0)$ converging to $u$ with respect to the graph norm on $\dom(\bar\Keg)$. 
  Using that $\PPP$ and $\left\{ \PPP\tfrac1r S_1 \right\}$ are bounded operators on $L^2((0,\thet),\LL)$, we conclude the proof. 
\end{proof}

\subsection{Small perturbations} \label{small_pert}

We work in the setting of Subsection \ref{cone_para}, but now also assume \ref{boundedperturbtwo}, i.e., 
\begin{align}
    \label{small1}     \norm{S_1 S_0^{-1}}_{L^{\infty}((0,\thet), \boundedops(\LL))} &  \leq \inf_{ s\in\spec(S_0)} \left| \frac{2s + 1}{4s} \right|, \\
     \label{small2}     \norm{S_0^{-1} S_1}_{L^{\infty}((0,\thet), \boundedops(\LL))} &  \leq \inf_{s\in \spec(S_0)} \left|  \frac{2s - 1}{4s} \right| .
\end{align}
Note that this includes the case $S_1 = 0$.
Under these assumptions, we can perturb the cone parametrix $\PPP$ constructed in Subsection \ref{cone_para} to improve Proposition \ref{lem:BoundParamMaxExt}, see Proposition \ref{relationmaxmin}.
First we need a seemingly tautological lemma for the operators defined in Lemma \ref{cor:BoundOnParamS1}.

\begin{lemma} \label{crazy}
For all $j \geq 1$  we have 
\[
      \left\{ \PPP \tfrac{1}{r} S_1 \right\}^j  \PPP = \PPP  \left\{ \tfrac{1}{r} S_1 \PPP \right\}^j \in \boundedops( L^2((0, \thet), \LL)).
\]
\end{lemma}

\begin{proof} Using the associativity of composition  in $\boundedops( L^2(0, \thet), \LL)$, it is enough to treat the case $j = 1$.
That is, we are claiming
\begin{equation} \label{jone}
      \left\{ \PPP \tfrac{1}{r} S_1 \right\}  \PPP = \PPP  \left\{ \tfrac{1}{r} S_1 \PPP \right\} \in \boundedops( L^2((0, \thet), \LL) ).
\end{equation}

Let $u \in C_{\cc}^0((0,\thet), \dom(S_0))$. 
Since $S_1$ is essentially bounded, we obtain $\tfrac{1}{r} S_1 u \in L^2((0,\thet), \LL)$.
We claim that 
\begin{equation} \label{seemstrivial}
     \left\{ \PPP \tfrac{1}{r} S_1 \right\} u = \PPP ( \tfrac{1}{r} S_1  u). 
\end{equation}
Let $0 < c < \thet$ be such that $\supp(u) \subset [c,\thet)$.
We find $u_n \in C_{\cc}^0((0,\thet), \CC) \otimes \dom(S_0)$ such that $\supp(u_n) \subset [\tfrac{1}{2}c,\thet)$ and  $u_n \longrightarrow u$ in $L^2((0,\thet), \LL)$.
Since $S_1$ is essentially bounded, we get  $\tfrac{1}{r} S_1 u_n \longrightarrow \tfrac{1}{r} S_1 u$ in $L^2((0,\thet), \LL)$ and since $\PPP$ is bounded, we get
\[
   \PPP ( \tfrac{1}{r} S_1 u_n) \longrightarrow \PPP ( \tfrac{1}{r}S_1 u) \textrm{ in } L^2((0,\thet), \LL).
 \]
 On  the other hand,  by definition of $\left\{  \PPP  \tfrac{1}{r} S_1\right\}$, we have $ \PPP ( \tfrac{1}{r} S_1 u_n)  \longrightarrow \left\{  \PPP  \tfrac{1}{r} S_1\right\} u$ in  $L^2((0,\thet), \LL)$.
This shows  \eqref{seemstrivial}.
 
 Now let $u \in C_{\cc}^0((0,\thet), \CC) \otimes \dom(S_0)$ and  let $\psi_n \in C_{\cc}^\infty((0,\thet), [0,1])$ be a sequence converging to the constant function $1$ in $L^2((0,\thet) , \RR)$.
 By Proposition \ref{difficult}, we get $\psi_n \PPP u \in C_{\cc}^1 ((0,\thet), \dom(S_0))$ for each $n$.
Using \eqref{seemstrivial} and the associativity of composition of maps, we have
\begin{equation} \label{tedious}
    \left\{ \PPP\tfrac{1}{r}S_1\right\}\psi_n\PPP u  = (\PPP\tfrac{1}{r}S_1)\psi_n\PPP u=\PPP(\tfrac{1}{r}S_1\psi_n\PPP)u . 
\end{equation}
Using the fact that $\psi_n \PPP u \longrightarrow \PPP u$ in $L^2((0,\thet), \LL)$, the left hand side of \eqref{tedious} converges in $L^2((0,\thet), \LL)$ to $\left\{\PPP \tfrac{1}{r} S_1\right\} \PPP u$.
On the other hand, in $L^2((0,\thet), \LL)$ we have 
\[
  (\tfrac{1}{r}S_1\psi_n \PPP) u = \psi_n (\tfrac{1}{r}S_1  \PPP) u  = \psi_n  \left\{ \tfrac{1}{r}S_1\PPP \right\} u \stackrel{n \to \infty}{\longrightarrow}  \left\{ \tfrac{1}{r}S_1 \PPP\right\} u .
\]
Hence the right hand side of \eqref{tedious} converges in $L^2((0,\thet), \LL)$ to $ \PPP \left\{ \tfrac{1}{r}S_1\PPP \right\}u$.
We conclude that
$$\left\{ \PPP \tfrac{1}{r} S_1 \right\}  \PPP u = \PPP  \left\{ \tfrac{1}{r} S_1 \PPP \right\}u.$$ 
Since $C_{\cc}^0((0,\thet), \CC) \otimes \dom(S_0) \subset L^2((0,\thet), \LL)$ is dense, \eqref{jone}  follows.
\end{proof}

\begin{proposition} \label{relationmaxmin} There exists $V \in \boundedops(L^2((0,\thet), \LL))$ with the following property: Let $u \in \dom( \bar \Keg)$ with essential support $\supp(u) \subset [0,\thet)$.
Then 
\[
     u = \PPP V \bar \Keg u. 
\]
\end{proposition} 

\begin{proof}
By Lemma \ref{cor:BoundOnParamS1} and \eqref{small1} and \eqref{small2},  the operators 
\begin{align*}
    A  & := \left\{ \PPP \tfrac{1}{r} S_1  \right\} \colon L^2((0,\thet), \LL) \to L^2((0,\thet), \LL), \\
    B & := \left \{ \tfrac{1}{r} S_1 \PPP \right\} \colon L^2((0,\thet), \LL) \to L^2((0,\thet), \LL) 
\end{align*} 
satisfy $\|A\| \leq \tfrac{1}{2}$ and $\|B\| \leq \tfrac{1}{2}$.
In particular,  we can define 
\[
    V : =  \sum_{j=0}^{\infty}  (-1)^j B^j \in \boundedops( L^2((0,\thet), \LL)).
 \]
Let $u \in \dom( \bar \Keg)$ with essential support in $[0, \thet)$.
By Proposition \ref{lem:BoundParamMaxExt}, we get
\[
   u =  \PPP \bar \Keg u -  \left\{ \PPP\tfrac{1}{r} S_1 \right\} u. 
\]
For $k \geq 0$, induction and Lemma \ref{crazy} shows
\[
      u = \PPP\left( \sum_{j=0}^{k}  (-1)^j B^j \right)\bar \Keg u   + (-1)^{k+1} A^{k+1} u.
\]
Using  $ \| A^{k+1}u\| \leq \left( \tfrac{1}{2} \right)^{k+1} \|u \| \stackrel{k \to \infty}{\longrightarrow} 0$, we obtain $u = \PPP V  \bar \Keg(u)$. 
\end{proof}

\begin{proposition} \label{exactdom} With the operators defined in \eqref{cute1} and \eqref{cute2}, the following holds. 
\begin{enumerate}[start=1,label=\textup{(\roman*)}]
\item \label{erstens} $\dom(\bar \Keg) = \dom(\overline{\partial_r} ) \cap \dom( \overline{\tfrac{1}{r} S_0})$.
\item \label{zweitens} The graph norm on $ \dom(\bar \Keg)$ is equivalent to the $H^1$-norm
\[
     \norm{u}^2_{H^1} := \norm{u}^2_{L^2} + \norm{\overline{\partial_r} u}^2_{L^2}  + \norm{\overline{\tfrac{1}{r} S_0} u}^2_{L^2} .
\]

\end{enumerate}
\end{proposition}

\begin{proof} 
Let $u \in C_{\cc}^{\infty}((0,\thet), \dom(S_0))$.
Applying Lemma \ref{relationmaxmin}, we obtain
 \begin{align*}
     \norm{ \tfrac1r  S_0 u }_{L^2}  & = \norm{\tfrac1r S_0 \PPP V \Keg u }_{L^2} \leq \norm{ \left\{ \tfrac1r S_0 \PPP \right\}} \cdot \norm{V} \cdot \norm{ \Keg u }_{L^2} , \\
     \norm{ \tfrac1r  S_1 u }_{L^2}  & = \norm{\tfrac1r S_1 \PPP V \Keg u }_{L^2} \leq \norm{ \left\{ \tfrac1r S_1 \PPP \right\}} \cdot \norm{V} \cdot \norm{ \Keg u }_{L^2} .
\end{align*}
Using that $\Keg u  = \partial_r u  + \frac{1}{r}(S_0+S_1(r)) u $, we furthermore have
\[
        \norm{ \partial_r  u }_{L^2}  \le \norm{\bar \Keg u}_{L^2} + \norm{ \tfrac{1}{r}S_0 u}_{L^2} +  \norm{\tfrac{1}{r}S_1 u}_{L^2}. 
\]
Together with the preceding estimates, this implies $\dom(\bar \Keg) \subset  \dom(\overline{\partial_r} ) \cap \dom( \overline{\tfrac{1}{r} S_0})$.

On the other hand, if $u \in L^2((0, \thet), \LL)$ and $u_n \in C_{\cc}^{\infty}( (0,\thet), \dom(S_0))$ is a sequence such that $u_n \longrightarrow u$, $\partial_r u_n \longrightarrow v$ and $\tfrac{1}{r} S_0 u_n \longrightarrow w$ in $L^2((0,\thet), \LL)$, then, since $S_1 S_0^{-1} \colon (0,\thet) \to \boundedops(\LL)$ is essentially bounded,  $\tfrac{1}{r} S_1 u_n$ is a Cauchy sequence in  $L^2((0,\thet), \LL)$ and  there exists  $z \in L^2((0,\thet), \LL)$ such that $\tfrac{1}{r} S_1 u_n \longrightarrow z$ in $L^2((0,\thet), \LL)$. 
We conclude $\Keg u_n \longrightarrow v + w + z$ in $L^2((0,\thet), \LL)$, hence $u \in \dom(\bar \Keg)$. 
Thus,
\[
    \dom(\overline{\partial_r} ) \cap \dom( \overline{\tfrac{1}{r} S_0}) \subset \dom(\bar \Keg).
\]
This completes the proof of \ref{erstens}.

Next, let  $u \in C_{\cc}^{\infty}( (0,\thet), \dom(S_0))$.
By the triangle inequality, we obtain
\[
  \norm{ \Keg u}_{L^2} \le  \norm{\partial_r u}_{L^2} + \norm{\tfrac{1}{r} S_0 u}_{L^2} + \norm{S_1 S_0^{-1}}_{L^{\infty}( (0,\thet), \boundedops(\LL))}  \norm{\tfrac{1}{r} S_0 u}_{L^2}  . 
\]
This implies that the Sobolev $1$-norm dominates the graph norm of $\bar \Keg$.

On the other hand, the previous considerations show  that there exists a constant $C >0$, which is independent of $u$, such that 
\[
    \norm{ \partial_r  u }_{L^2}  +  \norm{ \tfrac1r  S_0 u }_{L^2}  \leq C \norm{\Keg u}_{L^2} .
\]
Hence the two norms are in fact equivalent, completing the proof of \ref{zweitens}.
\end{proof}

 We now show parts \ref{b} to \ref{e} of Theorem \ref{thm:RegularSingular}.  
 For this we apply the previous results to $\Keg := \Brue_{\cone}$.
  Part \ref{b} and \ref{c} follow from Proposition \ref{exactdom}.
  For part \ref{d}, observe that for $\varphi \in C_{\cc, 0}^{\infty} ([0,\thet), \RR)$, the graph norm of $\bar \Brue$ is equivalent to the norm
\[
   \norm{u} := \norm{\varphi u}_{\dom(\bar \Keg)} + \norm{(1-\varphi) u}_{\dom(\bar \Brue)}.
\]
Hence the assertion follows from Lemma \ref{exactdom}.

  For assertion \ref{e} we first note that by \ref{d}, every $u\in \dom(\bar \Brue)$ satisfies the interior and  cone regularity conditions.
   Conversely, assume that $u\in \EE_{\bulk}\oplus L^2((0,1),\LL)$ satisfies the interior  and the cone regularity conditions.
   Then, by Proposition \ref{exactdom}, we obtain $\varphi u \in \dom(\bar \Keg ) \subset \dom(\bar \Brue)$. 
   Since $(1-\varphi) u \in \dom(\bar \Brue)$, we conclude $u \in \dom(\bar \Brue)$.

\subsection{Proof of Fredholm property}

Here we prove part \ref{fred} of Theorem \ref{thm:RegularSingular}.  
We construct a right and a left parametrix for  $\bar \Brue$. 
Let $\eps > 0$. 
Let  $\varphi, \psi, \chi \in C_{\cc,0}^\infty ([0,\thet),[0,1])$ such that $\psi|_{\supp(\varphi)} \equiv 1$ and $\chi|_{\supp(\psi)} \equiv 1$.
With Lemma \ref{lem:DefFunction}, we can assume that  $|r  \chi'| \leq \eps$ and $|r \psi'| \leq \eps$.

Let 
  \[
     \Keg = \Brue_{\cone} \colon L^2((0,\thet), \LL) \supset C_{\cc}^{\infty} ((0,\thet), \dom(S_0)) \to L^2((0,\thet), \LL)
  \] 
  be defined as in \eqref{prod_form} and let $\PPP \colon L^2((0,1), \LL) \to L^2((0,1), \LL)$ be defined as in Lemma \ref{lengthy}.
   Let  $P \colon \FF \to \dom(\bar \Brue)$ be an interior parametrix of $\Brue$ for $\varphi$ and $\psi$.
   
  Define $\mathscr{P}_R \colon \FF \to \EE$ by 
  \begin{equation*}
    \mathscr{P}_R:= (1-\varphi)P (1-\psi) + \chi \PPP \psi  .
  \end{equation*}
By  Proposition \ref{lemma:coneParametrix}, $\PPP_R$ defines a bounded operator $\FF \to \dom (\bar \Brue)$.
In addition, 
  \begin{equation*}\label{eq:PB_eq}
    \bar \Brue \mathscr{P}_R = (1-\psi) + R  + \psi + \chi \left\{  \tfrac1r S_1 \PPP \right\}  \psi +  \chi' \PPP \psi , \\
   \end{equation*}
where $R$ is a compact operator on $\FF$.

Consider the bounded operator $X:= \chi \left\{  \tfrac1r S_1 \PPP \right\}  \psi +  \chi' \PPP \psi$ on $\FF = \FF_{\bulk} \oplus L^2((0,\thet), \LL)$.
By Lemma \ref{cor:BoundOnParamS1},  by \ref{boundedperturbtwo} and Lemma \ref{little_shit}, we get an operator norm estimate
\[
    \norm{X} \leq \tfrac{1}{2} + \sup_{r \in (0,\thet)} | r  \chi'| \cdot \norm{\left\{ \tfrac{1}{r}  \PPP \right\}}.
 \]
For $\eps = \tfrac{1}{4} \norm{\left\{ \tfrac{1}{r}\PPP \right\}}^{-1} $, we hence obtain $\norm{X} \leq \tfrac{3}{4}$.
With this choice, $(\id + X)^{-1}$ is a bounded operator on $\FF$    and $R  (\id + X)^{-1}$ is a compact operator on $\FF$. 
The operator $\PPP_R (\id + X)^{-1} \colon \FF \to \dom(\bar \Brue)$ is bounded and satisfies 
\[
      \bar \Brue \PPP_R (\id + X)^{-1} = \id + R (\id + X)^{-1} .
\]
Therefore, $\PPP_R (\id + X)^{-1}$ is a right parametrix for $\bar \Brue$.

  Now, with $V \in \boundedops( L^2((0,\thet) , \LL))$ from Proposition \ref{relationmaxmin}, define $\PPP_L \colon \FF \to \EE$ by 
   \begin{equation*}
    \mathscr{P}_L := (1-\varphi)P (1-\psi) + \chi  \PPP V \psi . 
  \end{equation*}
  By  Proposition \ref{lemma:coneParametrix}, $\PPP_L$ is a bounded operator $\FF \to \dom (\bar \Keg)$.
Using Lemma \ref{maptocone}, Proposition \ref{relationmaxmin}, we obtain on $\dom(\bar \Brue)$
  \begin{equation*}
    \mathscr{P}_L \bar \Brue  = (1-\psi) + L + \psi - \chi \PPP V \psi'  . 
  \end{equation*}
 
   We want to show that the norm of the (by Proposition \ref{lemma:coneParametrix}) bounded operator 
    \begin{equation*}
        Y =  \chi \PPP V  \psi' \colon \dom(\bar \Brue) \to \dom(\bar \Brue)
    \end{equation*}
 becomes arbitrarily small, choosing $\eps$ small.
    
   As a preparation, we show that the operator norm of the multiplication operator $\psi' \colon \dom(\bar \Keg ) \to L^2((0,\thet), \LL)$ becomes arbitrarily small, choosing $\eps$ small.
   This follows from the choice of $\psi$ and Proposition \ref{relationmaxmin}, since, for $u \in \dom(\bar \Keg)$, we get  by Proposition \ref{relationmaxmin} and Lemma \ref{little_shit},
    \[
        \norm{\psi' u}_{L^2} =  \norm{ r \psi'  \left\{ \tfrac{1}{r} \PPP \right\} V \bar \Keg u}_{L^2}  \leq  \sup_{r \in (0,\thet)} | r \psi' | \norm{ \left\{ \tfrac{1}{r} \PPP \right\} V} \norm{u}_{\dom(\bar \Keg)}.  
    \]
      Using Proposition \ref{lemma:coneParametrix}, we now compute, for $u \in \dom(\bar \Keg)$,
    \begin{align*}
        \bar\Keg Y u & =  \Big( \chi V \psi'  +   \chi \left \{ \tfrac1r S_1 \PPP \right\}V \psi' -  \chi'\PPP V \psi'  \Big) u \\
                             & =  \Big( \chi V \psi'  +   \chi \left \{ \tfrac1r S_1 \PPP \right\}V \psi' - r \chi' \left\{ \tfrac{1}{r} \PPP \right\}V \psi'  \Big) u \\
                              & = \Big( \chi V  +   \chi \left \{ \tfrac1r S_1 \PPP \right\}V  - (r \chi') \left\{ \tfrac{1}{r} \PPP \right \}V  \Big) \psi' u. 
    \end{align*}
     Since by choosing $\eps$ small,  the norm of $\psi'  \colon \dom(\bar \Keg) \to \LL$ becomes arbitrarily small and since  $|r \chi'| \leq \eps$, the operator norm of $Y \colon \dom(\bar \Brue) \to \dom(\bar \Brue)$ becomes arbitrarily small, as required.
     
     Choose $\eps > 0$ so small that $\| Y \|_{\boundedops(\dom(\bar \Brue))} < 1$.
    Then $(\id + Y)^{-1}$ is a bounded operator on  $\dom(\bar \Brue)$ and $(\id + Y)^{-1} L$ is a compact operator on $\dom(\bar \Brue)$.
     We conclude that  $(\id + Y)^{-1} \mathscr{P}_L$ gives the required left parametrix for $\bar \Brue$.

\subsection{Deformations of abstract cone operators}

\begin{proposition}\label{lem:deformationLemma}
We fix the following data:
\begin{enumerate}[label=\myicon]
   \item  separable Hilbert spaces $\EE$ and  $\FF$ together with orthogonal decompositions $\EE = \EE_{\bulk} \oplus \EE_{\cone}$ and $\FF = \FF_{\bulk} \oplus \FF_{\cone}$;
    \item a separable Hilbert space $\LL$, some $0 < \thet \leq 1$ and Hilbert space isometries $ \EE_{\cone} \cong L^2((0,\thet), \LL)$ and $\FF_{\cone} \cong L^2((0,\thet) ,\LL)$.
  \end{enumerate}
Let $(\Brue_t)_{t \in   [0,1]} \colon  \EE \to \FF$ be a family of abstract cone operators with these data and with link operators $S_{0,t} \colon \LL \supset \dom(S_{0,t}) \to \LL$ and perturbations $S_{1,t}  \colon (0,\thet) \to \boundedops( \dom(S_{0,t}), \LL)$.

Furthermore, assume the following.
 \begin{enumerate}[label=\textup{(\roman*)}]
    \item\label{item:S0_cont}
     $\dom(S_{0,t}) \subset \LL$ is independent of $t$ and the graph norm of $S_{0,t}$  is independent of $t$, up to equivalence.
     Set  $\mathfrak{D}_{\link} := \dom(S_{0,t})$ with the graph norm of $S_{0,0}$.
    \item The map \label{item:S_0_cont}
      \begin{equation*}
        [0,1]\to \boundedops(\mathfrak{D}_{\link} , \LL), \quad t\mapsto S_{0,t} , 
      \end{equation*}
      is continuous.
    \item \label{item:S_1_cont}
       The map
        \begin{equation*} [0,1] \to L^{\infty}((0,\thet), \boundedops(\mathfrak{D}_{\link} , \LL)), \quad t \mapsto S_{1,t} , 
        \end{equation*}
        is continuous.
    \item\label{item:bulk_cont} $\dom(\Brue_t) \subset \EE$ is independent of $t$.
    Let $\varphi\in C_{\cc,0}^\infty([0,\thet), \RR)$ and set $\mathfrak{D}_{\bulk} := (1-\varphi) \dom(\Brue_0) \subset \EE$.
    Then the graph norm of $\Brue_t$ on $\mathfrak{D}_{\bulk}$  is independent of $t$, up to equivalence.
    Equip $\mathfrak{D}_{\bulk}$ with the graph norm of $\Brue_0$.
    \item \label{bulkcont} The map
      \begin{equation*}
        [0,1]\to \boundedops(\mathfrak{D}_{\bulk},\FF), \quad t\mapsto \Brue_t ,
      \end{equation*}
      is continuous.
  \end{enumerate}
 Then $\mathfrak{D} : = \dom(\bar \Brue_t)  \subset \EE$ is independent of $t$, and the graph norm of $\bar \Brue_t$ on $\mathfrak{D}$  is independent of $t$, up to equivalence.
Furthermore,  with the graph norm of $\bar \Brue_0$ on $\mathfrak{D}$, the assignment
      \begin{equation*}
        [0,1]\to \boundedops(\mathfrak{D}, \FF), \quad t \mapsto \overline\Brue_t ,
      \end{equation*}
is continuous.
\end{proposition}

\begin{proof}
   Let $t \in [0,1]$.
  Then, by  assumption \ref{item:S0_cont} the graph norms of $S_{0,0}$ and $S_{0,t}$ on $\mathfrak{D}_{\link}$ are equivalent. 
  By part \ref{b} of Theorem \ref{thm:RegularSingular}, we get  $\dom(\bar \Brue_{t, \cone}) = \dom( \bar \Brue_{0,\cone})$ and the graph norms of $\bar \Brue_{t, \cone}$ and $\bar \Brue_{0,\cone}$ are equivalent. 
Using assumption \ref{item:bulk_cont} and  part \ref{e} of Theorem \ref{thm:RegularSingular}, the graph norms of $\Brue_0$ and $\Brue_t$ on $\dom(\Brue_0) = \dom(\Brue_t)$ are equivalent.
This shows that $\mathfrak{D} : = \dom(\bar \Brue_t) \subset \EE$ is independent of $t$ and the graph norms of $\bar \Brue_t$ are pairwise equivalent.

    It remains to prove the continuity of $t \mapsto \bar \Brue_t$.
    By the spectral gap assumption \ref{AnSpGap}, we have $\norm{S_0}_{\boundedops(D_{\link}, \LL)} > 0$.
  Let $\eps > 0$.
  By assumptions \ref{item:S_0_cont} and \ref{item:S_1_cont}, there exists $\delta>0$ such that, for $\abs{t-s}<\delta$, we have
  \begin{align}
   \label{contdef1}  \norm{S_0^t  -S_0^s}_{\boundedops(D_{\link}, \LL)}&  \le\eps \norm{S_0}_{\boundedops(D_{\link}, \LL)} \\
    \label{contdef2} \norm{S_1^t -S^s_1 }_{L^{\infty}((0,\thet), \boundedops(D_{\link}, \LL))} & \le \eps \norm{S_0}_{\boundedops(D_{\link}, \LL)}  .
  \end{align}

Now for $\abs{t-s}<\delta$ and $u\in \dom(\overline\Brue)$ we get
\begin{equation*}\label{eq:ContinuityEstimateBrueOp}
  \norm{(\overline\Brue_t - \overline\Brue_s)u}_{\FF}\leq \norm{(\overline\Brue_t - \overline\Brue_s)\varphi u}_{L^2((0,\thet) ,\LL)} + \norm{(\overline\Brue_t - \overline\Brue_s)(1-\varphi)u}_{\FF}.
\end{equation*}
The second summand on the right hand side is controlled by assumption \ref{item:bulk_cont}. 
For the first summand on the right hand side, we argue as follows.
By Theorem \ref{thm:RegularSingular} \ref{d}, there exists $C > 0$, independent of $u$ and $\varphi$, such that, with \eqref{contdef1} and \eqref{contdef2},
\begin{align*}
  \norm{(\overline\Brue_t - \overline\Brue_s)\varphi u}_{L^2((0,\thet),\LL)} & \leq \norm{\overline{\tfrac1r (S_0^t-S_0^s)} \varphi u}_{L^2((0,\thet),\LL)} + \norm{\overline{\tfrac1r(S_1^t(r)-S_1^s(r))} \varphi u}_{L^2((0,\thet),\LL)} \\
  \leq& 2\eps \norm{\overline{\tfrac1r S_0} \varphi u}_{L^2((0,\thet),\LL)}\\
  \leq& 2\eps C\norm{u }_{\boundedops(\mathfrak{D}, \LL)}.
\end{align*}
This finishes the proof of Proposition \ref{lem:deformationLemma}.

\end{proof}

\section{Lipschitz rigidity for spherical suspensions of low regularity metrics} \label{spherical_suspension}

\subsection{Twisted Dirac operators on spherical suspensions as abstract cone operators}\label{sec:Dirac_as_regsingop}
 
Let $M$ be a smooth manifold equipped with a continuous Riemannian metric $g$.
We denote by $d_g$ the induced path metric on $M$, compare \cite{CHS}*{Reminder 1.2}. 

\begin{definition}\label{def:spherical_suspension}
  The \emph{spherical suspension} of $(M,g)$ is the smooth manifold
  $\susp M:=(0,\pi)\times M$ equipped with the continuous Riemannian metric
  $\susp g:=dr^2+\sin^2(r)g$ where $r$ is the standard coordinate on
  $(0,\pi)$.  The metric $\susp g$ is incomplete.  If $N$ is another smooth
  manifold equipped with a continuous Riemannian metric $h$ and
  $f\colon M\to N$ is a continuous map, then we define the spherical
  suspension of $f$ as
\end{definition}
\[
 \susp f  \colon (\susp M , \susp g) \to (\susp N, \susp h), \quad (r,x) \mapsto (r,f(x)) .
\]

We will work in the following setup.

\begin{setup} \label{setup} Let
\begin{enumerate}[label=\myicon]
  \item $(M,\gamma)$ be a closed smooth Riemannian spin manifold of odd dimension $n \geq 3$. 
 \item $g$ be a continuous Riemannian metric of Sobolev regularity $W^{1,p}$, $p > n+1$, on $M$.
   \item  $f\colon (M,d_g) \to \SSS^{n}$ be a $\Lambda$-Lipschitz map for some $\Lambda > 0$.
   \item  $(E_0,\nabla^{E_0})$ be a smooth Hermitian vector bundle over $\SSS^{n+1}$ with smooth metric connection $\nabla^{E_0}$.
   \item  $(E,\nabla^E):=(\susp f)^*(E_0,\nabla^{E_0}) \to \susp M$  be the pull back of $E_0$ under the spherical suspension of $f$. 
   This  is a Lipschitz bundle with induced Hermitian metric and metric Lipschitz connection $\nabla^E$.
 \end{enumerate}

\end{setup}

Let $\Cl(M)  \to (M, \gamma)$ be the Clifford algebra bundle and let $\vol_{\CC}(M) \in C^{\infty}(M, \Cl(M) \otimes \mathbb{C})$ be the complex volume form.
It is locally defined by 
\[
     \rm{vol}_{\CC}(M) = i^{\frac{n+1}{2}} e_1 \cdots e_n,
\]
where $(e_1, \ldots, e_n)$ is a smooth positively oriented $\gamma$-orthonormal local frame of $(TM, \gamma)$.
Let $S_M \to (M, \gamma)$  be the smooth spinor bundle associated with the irreducible unitary $\Spin(n)$-module for which  the complex volume element $\vol_{\CC}(M)$ acts as $+\id$. 
Let 
\[
   D_g \colon L_g^2(M, S_M) \supset \Lip(M, S_M) \to L_g^2(M, S_M) 
 \]
 be the Dirac operator associated with the metric $g$ acting on sections of $S_M$, see \cite{CHS}*{Definition 4.5}. 
 
 Let $F$ be a finite dimensional Hermitian vector space.
 We consider the induced Dirac operator 
 \[
   D_{g,F} = D_g \otimes \id  \colon L_g^2(M, S_M\otimes F) \supset \Lip(M, S_M \otimes F) \to L_g^2(M, S_M\otimes F) .
 \]
 Recall that, by \cite{CHS}*{Proposition 4.7},  the operator $D_{g,F}$  is essentially self-adjoint with $\dom(\bar D_{g,F}) = H_g^1(M, S_M \otimes F)$.

\begin{lemma} \label{scal_nearby} Assume that there exists  $\eta > 0$ such that $\scal_g \geq 1 + \eta$ in the distributional sense, see \cite{CHS}*{Definition 5.3}. 
That is, for all $u \in C^{\infty}(M, \R)$, $u \geq 0$, we have $\llangle
\scal_g , u \rrangle \geq (1+\eta) \int_M u\;d\mu_g$. 
 
 Then there exists $\eps > 0$ such that for each $W^{1,p}$-regular Riemannian metric $h$ on $M$  with  $\| h - g \|_{W_\gamma^{1,p}(M)} < \eps$ and for all Hermitian vector spaces $F$, the spectrum of $\bar D_{h,F}$ has empty intersection with $[-\tfrac{1}{2}, + \tfrac{1}{2} ]$.
 \end{lemma}
 
 \begin{proof} There exists $\eps > 0$ such that for each $W^{1,p}$-regular Riemannian metric $h$ on $M$  with $\| h - g \|_{W_\gamma^{1,p}(M)} < \eps$, we have $\scal_h \geq 1 + \tfrac{1}{2} \eta$.
 For this $\eps$, the claim holds by the integral Schr\"odinger-Lichnerowicz formula \cite{CHS}*{Theorem 5.1} (see also \cite{LL15}*{Proposition 3.2})  for the trivial twist bundle $M \times F \to M$.
 \end{proof}
 
The smooth Riemannian manifold $(\susp M, \susp\gamma)$ carries an induced spin structure.
The induced orientation of $\susp M$ is such that for a smooth positively oriented  frame $(e_1, \ldots, e_n)$ of $TM$, the  frame $((0,e_1), \ldots, (0,e_{n}), \partial_r)$ on $\susp M$ is positively oriented.
 
 Let $\Cl(\susp M) \to (\susp M, \susp \gamma)$ be the Clifford algebra bundle and let $ \rm{vol}_{\CC}(\susp M) = i^{\frac{n+1}{2}} e_1 \cdots  e_{n+1}$ be the complex volume element on $\susp M$, where $(e_1, \ldots, e_{n+1})$ is a smooth positively oriented $\susp \gamma$-orthonormal local frame of $T\susp M$.

Let $W$ be the unique irreducible unitary $\Spin(n+1)$-module (recall that $n+1$ is even) and let
 \begin{equation} \label{defspinor}
     \spinor = P_{\Spin}(\susp M, \susp \gamma) \times_{\Spin(n+1)} W \to (\susp M, \susp \gamma),
 \end{equation}
 be the smooth spinor bundle over $(\susp M, \susp \gamma)$.
We obtain the chiral decomposition  $\spinor = \spinor^+ \oplus \spinor^-$  into the $\pm 1$-eigenspaces of   $\vol_{\CC}(\susp M)$. 
The bundle $\spinor$ comes with a Hermitian structure so that the Clifford multiplication with unit vectors in $(T \susp M, \susp \gamma)$ is  isometric.
Let 
\[
    \nabla^{\spinor}_{\susp g} \colon \Lip_{\loc} ( \susp M, \spinor) \to L_{\susp g}^2(\susp M, T \susp M \otimes \spinor) 
\]
be the spinor connection on  $\spinor \to \susp M$, compare \cite{CHS}*{Definition 4.3}.

 Note that $E$ is a Hermitian Lipschitz bundle which carries a metric Lipschitz connection $\nabla^E$ induced by $\nabla^{E_0}$.
 Let  
 \[
    \nabla^{\spinor \tensor E} = \nabla^{\spinor}_{\susp g} \otimes 1 + 1 \otimes \nabla^E
 \]
 be the tensor product connection on $\spinor\tensor E$,

Let $b \colon W^{1,p}_{\loc} (\susp M, S^2_+(T^* \susp M)) \to W^{1,p}_{\loc}(\susp M, \End(T\susp M))$, $G \mapsto b_G$, be the continuous map defined in \cite{CHS}*{Lemma 4.1} with respect to the smooth background metric $\susp \gamma$ on $\susp M$.
Note that $p > n+1 = \dim \susp M$.
Recall that each $b_G$ is an $\susp \gamma$-self-adjoint automorphism of $T\susp M$ and maps local smooth $\susp \gamma$-orthonormal frames to $W^{1,p}$-regular $G$-orthonormal frames.
For  a $W^{1,p}$-regular metric $G$ on $\susp M$ and $v \in T\susp M$, we set
\begin{equation}\label{bIsometry}
  v^{G} := b_{G}(v).
\end{equation}
In particular, since the metric $\susp g$  on $\susp M$ is $W^{1,p}$-regular, for each smooth $\susp \gamma$-orthonormal frame $(e_1, \ldots, e_{n+1})$ of $(T\susp M, \susp \gamma)$, the $W^{1,p}$-regular frame $(e_1^{\susp g}, \ldots, e_{n+1}^{\susp g})$ is $\susp g$-orthonormal.

Let $L^2(\susp M,\spinor\tensor E; d\mu^{\susp g})$ be the Hilbert space of square-integrable sections of $\spinor\tensor E \to (\susp M, \susp g)$.
Define the twisted Dirac operator on $(\susp M, \susp g)$ as 
\begin{eqnarray} \label{defDir}
       \label{defdir}    \Dirac_E\colon L_{\susp g}^2(\susp M,\spinor  \tensor E) \supset \Lip_\cc(\susp M,\spinor  \tensor E) & \to&  L_{\susp g}^2(\susp M,\spinor  \tensor E) , \\
        \nonumber   \Dirac_E(\psi) &=  & \sum_{i=1}^{n+1} e_i \cdot \nabla^{\spinor \tensor E}_{e_i^{\susp g}} \psi , 
\end{eqnarray}
where $(e_1, \ldots, e_{n+1})$ is a local $\susp \gamma$-orthonormal frame of $T \susp M$, compare \cite{CHS}*{Definition 4.5}.

Since integration by parts holds for compactly supported Lipschitz sections, this is a symmetric linear differential operator of first order. 
In particular, it is closable.
  
With respect to the chirality decomposition $\spinor = \spinor^+ \oplus \spinor^-$, the Dirac operator $\Dirac_E$ is in block diagonal form
\[
   \Dirac^{\pm}_E\colon L_{\susp g}^2(\susp M,\spinor^{\pm}  \tensor E) \supset \Lip_\cc(\susp M,\spinor^{\pm} \tensor E)  \to   L_{\susp g}^2(\susp M,\spinor^{\mp}  \tensor E).
\]

\begin{proposition} \label{coneexample} 
Assume we are in the situation of Setup \ref{setup} and that there exists $\eta > 0$ such that $\scal_g \geq 1 + \eta$.

Then the twisted Dirac operators  $\Dirac^{\pm}_E$ on $\susp M$, defined in \eqref{defDir}, are abstract cone operators in the sense of Definition \ref{abstractcone}.
Furthermore, their closures 
\[
   \bar \Dirac\vphantom{\Dirac}^{\pm}_E \colon \dom (\bar \Dirac\vphantom{\Dirac}^{\pm}_E) \to L_{\susp g}^2(\susp , \spinor^{\mp} \otimes E)
\]
are Fredholm operators.
\end{proposition}

We prove this only for $\Dirac^{+}_E$.
The proof for $\Dirac^{-}_E$ is analogous.
Let 

\begin{enumerate}[label=\myicon]
  \item  $\EE := L_{\susp g}^2(\susp M,\spinor^+ \tensor E)$, 
  \item $\FF := L_{\susp g}^2(\susp M,\spinor^- \tensor E)$,
  \item $\EE_{\bulk} := L_{\susp g}^2((1,\pi-1)\times M, \spinor^+ \otimes E)$, 
  \item $\EE_{\cone} :=  L_{\susp g}^2((0,1) \times M, \spinor^+ \otimes E) \oplus L_{\susp g}^2((\pi-1,\pi) \times M, \spinor^+ \otimes E)$,
  \item $\FF_{\bulk} := L_{\susp g}^2((1,\pi-1)\times M, \spinor^- \otimes E)$, 
  \item $\FF_{\cone} :=  L_{\susp g}^2((0,1) \times M, \spinor^- \otimes E) \oplus L_{\susp g}^2((\pi-1,\pi) \times M, \spinor^- \otimes E)$.
\end{enumerate}

We carry out the analysis  for the cone part $L_{\susp g}^2((0,1) \times M, \spinor^+ \otimes E)$. 
The arguments for the other cone part  are analogous.

Using the isometric identification $TM \cong T(M \times \{\tfrac{1}{2}\}) \subset T\susp M$, $\xi \mapsto \sin(\tfrac12)^{-1} \xi$, we consider  $\spinor^{\pm}|_M \to M$ as $\Cl(M,\gamma)$-module bundles, by letting $\xi \in TM$ act by Clifford multiplication with $\pm \sin(\tfrac12)^{-1} \xi  \cdot \partial_r$. 
With this definition, the volume element $\rm{vol}_{\CC}(M)$ acts by $+ \id$ on $\spinor^{\pm}|_M$ and the isomorphism 
\begin{equation*}
  c(\partial_r) \colon \spinor^{+} \cong \spinor^- , 
\end{equation*}
where $c(-)$ denotes Clifford multiplication, is $\Cl(M, \gamma)$-linear. 
Altogether, we obtain unitary $\Cl(M,\gamma)$-module bundle isomorphisms $\alpha \colon \spinor^+ |_M \cong S_M$ and $\beta \colon\spinor^-|_M \cong S_M$ satisfying $\beta \circ c(\partial_r) = \alpha$.

Set $U := (0,1) \times M \subset \susp M$.
Using parallel translation of sections in $C^{\infty}(M, S_M)$ along geodesic lines $(0,1) \times \{x\}\subset \susp M$, the isomorphisms $\alpha$ and $\beta$ induce isometries 
\begin{align*} 
   \Phi_\alpha  \colon L_{\susp g}^2(U, \spinor^+|_U) & \to L^2( (0,1), L^2(M, S_M; \sin(r)^n d\mu^g))  \\ 
   \Phi_\beta  \colon L_{\susp g}^2(U, \spinor^-|_U) & \to L^2( (0,1),  L^2(M, S_M; \sin(r)^n d\mu^g)) .
\end{align*} 
A straightforward calculation shows
\begin{equation*} 
  \Phi_\beta \circ \Dirac^+ \circ\, \Phi_\alpha^{-1} = \partial_r + \frac{n}{2} \cot(r) + \frac{1}{\sin(r)} D_M  .
\end{equation*}
This expression can be further simplified as follows.
We have isometries
\begin{align} 
 \label{important1}  \widetilde  \Phi_\alpha \colon L_{\susp g}^2(U, \spinor^+|_U) & \to L^2 ( (0,1),L_g^2(M, S_M))  , \\
 \label{important2}  \widetilde  \Phi_\beta \colon L_{\susp g}^2(U, \spinor^-|_U ) & \to L^2((0,1), L_g^2(M, S_M) )
\end{align} 
given by 
\begin{align*}
  \widetilde \Phi_\alpha(u)(r) &= \sin(r)^{\frac{n}{2}}  \cdot \Phi_{\alpha}(u)(r)  , \\
  \widetilde \Phi_\beta(u)(r)  &= \sin(r)^{\frac{n}{2}}  \cdot \Phi_{\beta}(u)(r)  ,
\end{align*}
and one calculates
\begin{equation}
 \label{Diracnontwist}  \widetilde \Phi_\beta \circ \Dirac^+ \circ \, \widetilde \Phi_\alpha^{-1} = \partial_r + \frac{1}{\sin(r)} D_M . 
\end{equation}

Next, we consider the twisted Dirac operator $\Dirac_E$ from \eqref{defDir} acting on sections of $\spinor \otimes E$. 
Let 
\begin{equation*}
  V = B_1((0,\ldots, -1)) \subset \SSS^{n+1} 
\end{equation*}
be the open ball of radius $1$ around the south pole of $\SSS^{n+1}$, let $F = E_0|_{(0,\ldots, 0, -1)}$ be the fibre of $E_0$ over the south pole of $\SSS^{n+1}$ and let  
\begin{equation} \label{trivialsphere}
     E_0|_{V} \cong V \times F 
\end{equation}
be the smooth unitary trivialization obtained by parallel transport along radial lines emanating from $(0,\ldots, 0, -1) \in \SSS^{n+1}$.
The trivialization \eqref{trivialsphere} and the map $\susp f$  induce a unitary trivialization $E|_U \cong U \times F$.

Let $\omega \in \Omega^1(V, \End(F))$ be the connection $1$-form of $\nabla^{E_0}$ with respect to the trivialization \eqref{trivialsphere}.
It defines a bounded section  of the Hermitian bundle  $T^*V \otimes \End(F) \to V$.

Consider the separable Hilbert space
\[
    \LL := L_g^2(M, S_M \otimes F) . 
\]
The unitary trivialization $E|_U \cong U \times F$ and the isometries $\widetilde \Phi_{\alpha}$ and $\widetilde \Phi_{\beta}$ induce isometries
\begin{align} 
  \label{L2Isom1}  \Phi_{\EE} \colon L_{\susp g}^2(U, \spinor^+\otimes\ E) & \to L^2 ( (0,1) , \LL ), \\
  \label{L2Isom2}   \Phi_{\FF} \colon L_{\susp g}^2(U, \spinor^- \otimes\ E) & \to L^2((0,1),  \LL ) . 
\end{align} 
This is the datum required in \ref{tre} in Definition \ref{abstractcone}.
Furthermore, if $(e_1, \ldots,  e_{n})$ is a local $\gamma$-orthonormal frame on $TM$, we have on $C_{\cc}^{\infty}((0,1), \Lip(M, S_M \otimes F))$, using \eqref{defDir} and $(\susp f)^* \omega(\partial_r) = 0$, we obtain
\[
  \widetilde \Phi^- \circ \Dirac_E^+ \circ ( \widetilde \Phi^{+})^{-1} = \partial_r + \frac{1}{\sin(r)} ( D_M \otimes \id_F) +\sum_{i=1}^{n}  c \left( e_i \right) \otimes ((\susp f)^* \omega)   \left( \big(0,  \tfrac{1}{\sin(r)} e_i^{g}\big) \right) . 
\]
Here, $\big(0,  \tfrac{1}{\sin(r)} e_i^{g}\big)  \in T( (0,1) \times M))|_{\{r\} \times M}$ is a unit vector with respect to the metric $\susp g$.

Consider the dense subspace $\dom(S_0) := \Lip(M, S_M \otimes F)$ of  $ \LL$ and set
\begin{equation*}
  S_0    =  D_M  \otimes \id_{F} \colon \LL  \supset \dom(S_0)  \to \LL.
\end{equation*}
The operator $S_0$  is a link operator as required in \ref{quattro}.

For $r \in (0,1)$, we define $S_1(r) \in \boundedops( \dom(S_0), \LL)$, 
\begin{equation*}
  S_1 (r) = r \cdot \Bigg( \left( \frac{1}{\sin(r)} - \frac{1}{r} \right) D_M \otimes \id_{F}  + \sum_{i=1}^{n} c \left( e_i \right) \otimes ((\susp f)^* \omega)   \left( \big(0,  \tfrac{1}{\sin(r)} e_i^{g}\big) \right) \Bigg) .
\end{equation*}
The map $S_1 \colon (0,1) \to \boundedops( \dom(S_0), \LL)$ is continuous, hence measurable. 
Furthermore, since $\susp f$ is $\max \{1,\Lambda \}$-Lipschitz continuous and the form $\omega$ is bounded, the map $S_1$ is bounded.
Thus $S_1$ is a perturbation of the form \ref{cinque}.

As usual, the isometries $\Phi_{\EE}$ and $\Phi_{\FF}$ will be suppressed from the notation and we write the cone part of the Dirac operator as 
\begin{equation} \label{productyes}
   \Dirac_E^+ = \partial_r + \tfrac1r (S_0 + S_1(r) ) .
\end{equation}
This is the product structure as required in \ref{AnSetup5}, and we now check that the conditions on the operators are satisfied.

Concerning \ref{AnSetup4}, we observe that we have a canonical isomorphism
\[
    \Lip_c(U, \spinor^+\otimes\ E; \susp g) \cong \Lip_c ( (0,1) , \Lip(S_M \otimes F; g)).
\]
Hence, for every $\varphi\in C_{\cc,0}^\infty([0,1), \R)$ and $u \in  \Lip_\cc(\susp M, \spinor\otimes E)$, we have 
\[
    \varphi u \in \varphi \big( \Lip_c ( (0,1) , \Lip(S_M \otimes F; g)) \big). 
\]
Folding this element with a mollifier in $r$ direction and using the product structure \eqref{productyes}, we see that the inclusion
\begin{equation*}
  \varphi\cdot C_\cc^\infty((0,1), \Lip(M, S_M \otimes F)) \subset \varphi \cdot   \dom(\Dirac_E^+)
\end{equation*}
is dense with respect to the graph norm on $\dom(\Dirac_E^+)$.
Property  \ref{AnSetup2} follows from the fact that $\Dirac_E$ is a differential operator.
Property \ref{AnSpGap} follows from Lemma \ref{scal_nearby}.

For $r \in (0,1)$, the operator $\frac1r S_1(r)$ is the product of $S_0$ by the bounded scalar $\frac{1}{\sin(r)}-\frac{1}{r}$ plus an operator of order $0$ whose $L^2$-operator norm is  bounded in $r$.
This implies that the  continuous family $\bar S_0^{-1} \bar S_1 \colon (0,1) \to \boundedops( \dom(S_0))$ induces a bounded family $(0,1) \to \boundedops( \LL)$, as required in \ref{AnSetup6}.
Furthermore, it implies that
\[
 \lim_{\thet \to 0} \left (   \norm{  \bar S_1\bar  S_0^{-1}}_{L^{\infty}((0, \vartheta), \boundedops(\LL))} +   \norm{  \bar  S_0^{-1} \bar S_1}_{L^{\infty}((0, \vartheta), \boundedops(\LL))} \right) = 0 . 
\]
This shows that \ref{boundedperturbtwo} holds over some $(0,\thet)$, $0 < \thet \leq 1$.

We have thus shown that $\Dirac^+_E$ is an abstract cone operator. 

The Fredholm property follows from Theorem \ref{thm:RegularSingular} \ref{fred} and  the next proposition.

\begin{proposition} \label{constr_int_parametrix} For all $\varphi, \psi \in C_{\cc,0}([0,\thet), \RR)$ with $\psi|_{\supp \varphi} \equiv 1$, there exists an interior parametrix of $\Dirac_E$ for $\varphi$ and $\psi$.
\end{proposition}

\begin{proof}
There exists a closed interval $[a,b] \subset (0,\pi)$ with $\supp(1-\varphi)\subset (a,b) \times M$.
Set  
\[
     \Omega := [a,b]\times M \subset \susp M .
\]

The normal exponential map along $\partial \Omega \subset \Omega$ for the restriction of the smooth Riemannian metric $\susp \gamma$ to $\Omega$ induces a collar diffeomorphism $[0,1) \times \partial \Omega \stackrel{\approx}{\rightarrow} U$ where $U$ is an open neighborhood of $\partial \Omega$ in $\Omega$.
Hence we obtain a smooth structure on the double $\mathsf{D}{\Omega} = \Omega \cup_{\partial \Omega} \Omega$.
There exists a continuous metric $\widetilde{g}$ of regularity $W^{1,p}$ on $\mathsf{D}{\Omega}$ with $\widetilde g|_{\Omega} = \susp g|_\Omega$.

Similarly, let $\mathsf{D} \Sigma$ be the smooth double of $\Sigma := [a,b]\times \SSS^{n-1} \subset \SSS^n$ with respect to the metric induced by $\SSS^n$ and let $\widetilde h$ be a smooth metric on $\mathsf{D} \Sigma$ which restricts to the round metric on $[a,b]\times \SSS^{n-1} \subset \SSS^n$.

The restriction of the smooth Hermitian vector bundle $E_0$ with metric connection to $\Sigma$ extends to a smooth Hermitian vector bundle $\tilde E_0 \to \mathsf{D} \Sigma$ with metric connection.

The closed Riemannian manifold $(\mathsf{D} \Omega, \tilde g)$ carries a spin structure restricting to the given spin structure on $(\Omega,g)$.
Let $\widetilde{\spinor} \to (\mathsf{D} \Omega, \tilde g)$ be the spinor bundle.
The restriction of $\widetilde{\spinor}$ to $\Omega$ can be identified with the restriction of $\spinor$ to $\Omega$.

The map $\susp f$ extends to a Lipschitz map
\begin{equation*}
  \widetilde{F}\colon \mathsf{D}\Omega \to \mathsf{D} \Sigma  .
\end{equation*}
Let $\widetilde E$ be the pull back bundle $\widetilde{F}^*(\widetilde E_0) \to \mathsf{D} \Omega$ along $\widetilde F$ together with the induced metric connection.
Then $\widetilde E$ is a Hermitian Lipschitz bundle with compatible metric connection.

We hence obtain  a twisted Dirac operator 
\[
   \mathscr{D} \colon L^2( \mathsf{D} \Omega, \widetilde \spinor\vphantom{\spinor}^+ \otimes \widetilde E) \supset \Lip (  \mathsf{D} \Omega, \widetilde \spinor\vphantom{\spinor}^+ \otimes \widetilde E) \to  L^2( \mathsf{D} \Omega, \widetilde \spinor\vphantom{\spinor}^- \otimes \widetilde E) .
 \]
The operator $\mathscr{D}$ restricts to the original Dirac operator $\Dirac^+_E$ on $\Lip_c (\Omega, \spinor^+ \otimes E)$. 

Let $q = \dim \widetilde \spinor\vphantom\spinor^{\pm}$ and $r= \dim E$.
Working in smooth trivializations of $\widetilde \spinor\vphantom{\spinor}^{\pm}$ and Lipschitz trivializations of $\widetilde E$ over open subsets $U \subset \mathsf{D} \Omega$, the operator $\mathscr{D} $ takes the form 
\[
      \mathscr{D}(u) = a^j \partial_j u + bu
\]
where $a^j \in W^{1,p}_{\loc}(U, \End(\CC^{q+r}))$ and $b \in L^p_{\loc}(\susp U, \End(\CC^{q+r}))$.
Furthermore, each $a^j$ is continuous. 
In particular, since $p \geq \dim \susp M \geq 4$, the regularity conditions in \cite{BartnickChrusciel}*{Equation (3.4)} and the ellipticity condition \cite{BartnickChrusciel}*{Equation (3.5)} are satisfied.
By \cite{BartnickChrusciel}*{Theorem 3.7} and \cite{BartnickChrusciel}*{Corollary 4.5}, the domain of the closure of $\mathscr{D}$  is equal to $H^1(  \mathsf{D} \Omega, \widetilde \spinor\vphantom{\spinor}^+ \otimes \widetilde E)$ and the operator $\bar{\mathscr{D}} \colon H^1(  \mathsf{D} \Omega, \widetilde \spinor\vphantom{\spinor}^+ \otimes \widetilde E) \to L^2(  \mathsf{D} \Omega, \widetilde \spinor\vphantom{\spinor}^- \otimes \widetilde E)$ is Fredholm.

Hence we obtain a parametrix for $\bar{\mathscr{D}}$, i.e., a bounded operator
\begin{equation*}
  \widetilde{P}\colon L^2(\mathsf{D} \Omega , \widetilde{\spinor}\vphantom{\spinor}^- \otimes  \widetilde{E}) \to H^1(\mathsf{D}\Omega, \widetilde{\spinor}\vphantom{\spinor}^+ \otimes  \widetilde{E})
\end{equation*}
and compact operators $\widetilde R$ on $L^2(\mathsf{D} \Omega, \widetilde{\spinor}\vphantom{\spinor}^-\otimes \widetilde{E})$ and $\widetilde L$ on $H^1(\mathsf{D} \Omega, \widetilde{\spinor}\vphantom{\spinor}^+ \otimes  \widetilde{E})$ satisfying
\[
  \bar{\mathscr{D}} \widetilde{P}   = \id  + \widetilde R , \qquad \widetilde{P} \bar{\mathscr{D}}  = \id  + \widetilde L  .
\]
Let $\chi \in C_{\cc}^{\infty}((a,b) \times M, [0,1])$ such that $\chi|_{\supp (1-\varphi)} \equiv 1$ and consider the bounded operator
\[
     P := \chi \widetilde{P} \chi \colon L^2 ( \susp M, \spinor^- \otimes E) \to H^1(\susp M, \spinor^+ \otimes E).
\]
We claim that $P$ is an interior parametrix of $\bar \Dirac^+_E$ for $\varphi$ and $\psi$.

For this aim, we compute 
\begin{align*}
     \bar \Dirac\vphantom{\Dirac}^+_E (1-\varphi) P (1-\psi) & = \bar{\mathscr{D}} (1-\varphi) \widetilde P (1-\psi) \\
    & = (1-\varphi) \bar{\mathscr{D}} \widetilde P (1-\psi) + [\varphi , \bar{\mathscr{D}}] \widetilde P (1-\psi) \\
    & = (1-\psi) + (1-\varphi) \widetilde{R} (1-\psi) + [\varphi , \bar{\mathscr{D}}] \widetilde P (1-\psi) .
\end{align*}
Since the commutator $[\varphi, \mathscr{D}]$ is a differential operator of order $0$ and $\widetilde P$ is a bounded operator $L^2 \to H^1$, the composition $[\varphi , \bar{\mathscr{D}}] \widetilde P (1-\psi)$ defines a compact operator on $L^2(\susp M, \spinor^- \otimes E)$ by the Rellich-Kondrachov embedding theorem. 
Hence, 
\[
    R := (1-\varphi) \widetilde{R} (1-\psi) + [\varphi , \bar{\mathscr{D}}] \widetilde P (1-\psi)
\]
defines a compact operator on $L^2 ( \susp M, \spinor^- \otimes E)$. 

Similarly, 
\begin{align*}
     (1-\varphi) P (1-\psi)  \bar \Dirac\vphantom{\Dirac}^+_E & =(1-\varphi) \widetilde P (1-\psi)  \bar{\mathscr{D}}  \\
    & = (1-\varphi) \widetilde P \bar{\mathscr{D}} (1-\psi) + (1-\varphi) \widetilde P [\bar{\mathscr{D}}, \psi] \\
    & = (1-\psi) + (1-\varphi) \widetilde{L} (1-\psi) +(1-\varphi) \widetilde P [\bar{\mathscr{D}}, \psi] .
\end{align*}
Again, the commutator $[\mathscr{D}, \psi]$ is a differential operator of order zero, hence $ (1-\varphi) \widetilde P [\bar{\mathscr{D}}, \psi]$ defines a compact operator on $H^1(\susp M, \spinor^+ \otimes E)$ and 
\[
   L:= (1-\varphi) \widetilde{L} (1-\psi) +(1-\varphi) \widetilde P [\bar{\mathscr{D}}, \psi]
\]
is a compact operator on $H^1(\susp M, \spinor^+ \otimes E)$. 

Hence $P$ satisfies  all of the required properties.
\end{proof}

\subsection{Integral Schr\"odinger-Lichnerowicz formula for spherical suspensions}

We work in the setting of Setup \ref{setup}.
Furthermore, we assume $\scal_g \geq n(n-1)$ in the distributional sense.
In Theorem \ref{thm:SL_cone}, we will prove an integral Schr\"odinger-Lichnerowicz  formula which applies to sections in $\dom (\bar  \Dirac\vphantom{\Dirac}^{\pm}_E)$ on the spherical suspension $(\susp M, \susp g)$.

  \begin{lemma} \label{scalinsuspension}
  Define the distribution $\rho^{-2}\scal_g \colon C^\infty_\cc(\susp M ) \to \R$ by
  \begin{equation*}
   \llangle \rho^{-2}\scal_g,u \rrangle := \int_a^b \rho(t)^{-2} \llangle \scal_g, u(t,\cdot) \big\rrangle \,dt .
  \end{equation*}
 Then
  \begin{equation*}
    \scal_{\susp g} = \rho^{-2}\scal_g - n\frac{(n-1)(\rho')^2 + 2\rho  \rho''}{\rho^2}.
  \end{equation*}
In particular,  if $\scal_g \geq n(n-1)$, then $\scal_{\susp g} \ge (n+1)n$.
\end{lemma}

\begin{proof}
  This follows from the definition of the scalar curvature distribution \cite{LL15}*{Definition 2.1} together with the standard computation of the scalar curvature of a warped product as in \cite{ONeill}*{Chapter 7}.
\end{proof}

From now on we work in Setup \ref{setup}.

\begin{lemma} \label{curvatureterm}
Let $R^{E_0} \in \Omega^2(\SSS^{n+1}, \End(E_0))$ be the curvature form of  $(E_0,\nabla_0) \to \SSS^{n+1}$.
We consider the almost everywhere defined measurable $2$-form on $M$ with values in $\End(E)$ given by 
\[
     R^E := (\susp f)^*(R^{E_0}).
\]

Then  $R^E  \in L^{\infty}\left(\Omega^2(\susp M, \End(E))\right)$. 
\end{lemma}

\begin{proof}
   Since $\SSS^{n+1}$ is compact, $R^{E_0}$ is uniformly bounded.
   By definition, we have
  \[
      (\susp f)^*R^{E_0}_{v,w}= R^{E_0}_{d_x\susp f(v),d_x\susp f(w)}
  \]
    for almost all $x \in \susp M$ and all $v,w\in T_x\susp M$.
   Since $f$ is $\Lambda$-Lipschitz, the map  $\susp f$ is $\max\{1,\Lambda\}$-Lipschitz.
   From this, the assertion of Lemma \ref{curvatureterm} follows.
\end{proof}

For $u = \sigma \otimes \eta \in (\spinor \otimes E)_x$, we set 
   \begin{equation*}
    \mathcal{R}_{\susp g}^E u:= \frac{1}{2}\sum_{i,j} e_i\cdot e_j \sigma \otimes R^E_{e_i^{\susp g},e_j^{\susp g}} \eta \in (\spinor \otimes E)_x ,
    \end{equation*}
  where $(e_i, \ldots e_{n+1})$ is some $\susp \gamma$-orthonormal basis of $(T_{x} \susp M, (\susp \gamma)_x)$. 
 It follows from Lemma \ref{curvatureterm} that $\mathcal{R}_{\susp g}^E$ defines a bounded operator  $L_{\susp g}^2(\susp M, \spinor \otimes E) \to L_{\susp g}^2(\susp M, \spinor \otimes E)$.

Due to Proposition \ref{coneexample}, the twisted Dirac operator $\Dirac_E$ has a unique closed extension $\bar \Dirac_E$.
We furthermore consider the operator
\[
   \nabla_E := \nabla^{\spinor \otimes E} \colon L_{\susp g}^2(\susp M , \spinor \otimes E) \supset \Lip_{\cc}(\susp M , \spinor \otimes E) \to L_{\susp g}^2(\susp M, T^* \susp M \otimes \spinor \otimes E).
\]
As a linear differential operator it is closable.
Let
\[
   \bar \nabla_E \colon  L_{\susp g}^2(\susp M , \spinor \otimes E) \supset \dom(\bar \nabla_E) \to  L_{\susp g}^2(\susp E, T^* \susp M \otimes \spinor \otimes E)
\]
be its closure.
We now obtain the following integral Schr\"odinger-Lichnerowicz formula for spherical suspensions.

\begin{theorem}\label{thm:SL_cone}
Assume we are working in Setup \ref{setup} and that $\scal_g \geq n(n-1)$ in the distributional sense.
Then the following holds. 
\begin{enumerate}[label=\textup{(\roman*)}]
\item \label{un}  The sesquilinear form $\Lip_{\cc}(\susp M, \spinor \otimes E) \times \Lip_{\cc}(\susp M, \spinor \otimes E) \to \CC$, $(u,v) \mapsto \bigl(  \nabla_E u, \nabla_E v \bigr)_{L^2}$ extends to a continuous sesquilinear functional
\[
  \bigl( \nabla_E -  , \nabla_E -  \bigr)_{L^2} \colon  \dom( \bar \Dirac_E) \times \dom( \bar \Dirac_E) \to \CC .
\]
\item \label{deux} The sesquilinear form $\Lip_{\cc}(\susp M, \spinor \otimes E) \times \Lip_{\cc}(\susp M, \spinor \otimes E) \to \CC$, $(u,v) \mapsto \llangle \scal_g , \langle u,v \rangle \rrangle$ extends to a continuous sesquilinear functional
\[
  \llangle \scal_{\susp g}, \langle - , - \rangle \rrangle  \colon \dom( \bar \Dirac_E) \times \dom( \bar \Dirac_E) \to \CC .
\]
\item \label{trois} For $u, v \in \dom(\bar \Dirac_E)$, the integral Schr\"odinger-Lichnerowicz formula holds: 
 \begin{equation*}\label{eq:SL_cone}
    \bigl(\bar \Dirac_E u, \bar \Dirac_E v \bigr)_{L^2}=\bigl( \nabla_E u,  \nabla_E v \bigr)_{L^2} + \tfrac14\llangle \scal_{\susp g}, \langle u,v\rangle \rrangle + \bigl(\mathcal R^E u, v \bigr)_{L^2}
  \end{equation*}
(Note that the last summand is defined by Proposition \ref{curvatureterm}).
\end{enumerate}
\end{theorem}

\begin{proof}
If $u$ and $v$ are compactly supported Lipschitz sections, the proof of \cite{CHS}*{Theorem 5.1} carries over to show \ref{trois} (recall that $p > n+1$ by assumption).
Let $u \in \dom(\bar \Dirac_E)$ and let the sequence $u_i \in  \Lip_\cc(\susp M,\spinor\tensor E)$ converge in  $\dom( \bar \Dirac_E)$ to $u$.

We  apply the Schr\"odinger-Lichnerowicz formula \ref{trois} to the compactly supported Lipschitz sections $u_i - u_j$ and get
  \begin{align*}
    \|\Dirac_E u_i - \Dirac_E u_j \|_{L^2}^2 =&\ \| \Dirac_E(u_i-u_j) \|^2_{L^2} \\
    =&\ \|\nabla_E(u_i-u_j) \|^2_{L^2} + \tfrac{1}{4} \llangle \scal_{\susp g},\abs{u_i-u_j}^2\rrangle + \bigl(\mathcal R^E(u_i-u_j),u_i-u_j \bigr)_{L^2}  .
  \end{align*}
  The operator $\mathcal R^E$ is uniformly bounded on $L^2$ and the distribution $\scal_{\susp g}$ satisfies $\scal_{\susp g} \ge (n+1)n$ by Lemma \ref{scalinsuspension}. 
When we pass to the limit $i,j\to \infty$, the left-hand side goes to $0$ by assumption. 
By the boundedness of $\mathcal R^E$, the same applies to the right-most summand, while the other two summands are non-negative.

  It follows that all terms converge to $0$.
  In particular
  \begin{equation*}
    \lim_{i,j\to\infty} \| \nabla_E u_i- \nabla_E u_j \|_{L^2} =0 .
  \end{equation*}
  We conclude that $\dom( \bar \Dirac_E) \subset \dom( \bar \nabla_E)$.
  Furthermore, by polarization, we conclude \ref{un}.
  
  We  now apply the Schr\"odinger-Lichnerowicz formula \ref{trois} to the compactly supported Lipschitz sections $u_i $ and get, for all $i$,
    \[
 \| \Dirac_E u_i \|^2_{L^2} = \| \nabla_E u_i  \|^2_{L^2} + \tfrac{1}{4} \llangle \scal_{\susp g},\abs{u_i}^2\rrangle + \bigl(\mathcal R^E u_i ,u_i\bigr)_{L^2}  .
  \]
  Sending $i$ to infinity, the left hand side converges and the first and third summand of the right hand side converge. 
  Hence, $\llangle \scal_{\susp g},\abs{u_i}^2\rrangle$ converges, as well.
  A similar argument shows that the limit is independent of the choice of $u_i$ so that we can define
  \[
   \llangle \scal_{\susp g},\abs{u}^2\rrangle := \lim_{i \to \infty}  \llangle \scal_{\susp g},\abs{u_i}^2\rrangle \in \R. 
  \]
  Now let $u, v \in \dom(\bar \Dirac_E)$ and let $u_i$ and $v_i$ be sequences of compactly supported Lipschitz sections of $\spinor \otimes E$ converging to $u$ and $v$ in $\dom ( \bar \Dirac_E)$.
  By polarization and the previous argument, we see that the limit 
 \[
   \llangle \scal_{\susp g},\langle u, v \rangle\rrangle := \lim_{i \to \infty}  \llangle \scal_{\susp g},\langle u_i, v_i \rangle \rrangle \in \CC
  \]
  exists and is independent of the choice of $u_i$ and $v_i$.
  A similar argument shows that the sesquilinear form $\dom(\bar \Dirac_E) \times \dom(\bar \Dirac_E) \to \CC$, $(u,v) \mapsto \llangle \scal_{\susp g}, \langle u, v \rangle \rrangle$ is continuous. 
  This finishes the proof of \ref{deux}.
  
  For assertion \ref{trois}, we write down the Schr\"odinger-Lichnerowicz formula for $u_i$ and $v_i$, send $i$ to infinity and observe that each summand of the Schr\"odinger-Lichnerowicz formula converges.

\end{proof}

\subsection{Index formula on the spherical suspension}

Assume that we are in the situation of Setup \ref{setup} where $E_0$ is  the canonical spinor bundle $\Sigma \to \SSS^{n+1}$.
This bundle comes with a chirality decomposition $E_0 = E_0^+ \oplus E_0^-$ coming from the $\pm 1$-eigenbundle decomposition of the volume element over the even dimensional sphere $\SSS^{n+1}$.
We obtain a corresponding decomposition $E = E^+ \oplus E^-$, where $E^{\pm} = (\susp f)^*(E_0^{\pm})$.

Let $\gamma$ be a smooth Riemannian metric on $M$, choose a spin structure of $(M,\gamma)$ and let $\spinor = \spinor^+ \oplus \spinor^- \to (\susp M, \susp \gamma)$ be the spinor bundle from \eqref{defspinor}.

With respect to the direct sum decomposition
\begin{equation*}
  \spinor \otimes E = (\spinor \otimes E)^+ \oplus (\spinor\otimes E)^-,
\end{equation*}
where
\begin{align} 
  \label{eq:decomp_of_plustwistbundle} (\spinor \otimes E)^+ & = \big( \spinor^+ \otimes\ E^+ \big) \oplus \big( \spinor^- \otimes\ E^- \big),  \\
  \label{eq:decomp_of_minusplustwistbundle} (\spinor\otimes E)^-  & = \big( \spinor^- \otimes\ E^+ \big)  \oplus \big( \spinor^+ \otimes\ E^- \big) ,
\end{align} 
the twisted Dirac operator $\Dirac_E$ introduced in \eqref{defDir} acting on sections of $\spinor \otimes E$ is of block diagonal form, that is $ \Dirac_E = \Dirac_E^+ \oplus \Dirac_E^-$,
\begin{equation*}
      \Dirac_E^{\pm} \colon L^2_{\susp g}(\susp M, \spinor \otimes E) \supset \Lip_c( \susp M, (\spinor \otimes E)^{\pm}) \to L^2_{\susp g}( \susp M,  (\spinor \otimes E)^{\mp}).
\end{equation*}
It follows from Proposition \ref{coneexample} that $\Dirac_E^{\pm}$ are abstract cone operators and their closures $\bar \Dirac\vphantom{\Dirac}^{\pm}_E$ are Fredholm.
The aim of this section is to prove the following index formula.

\begin{proposition}\label{prop:IndexFormula} We have
  \begin{equation*}
    \ind( \bar \Dirac\vphantom{\Dirac}^+_E) =  (-1)^{\tfrac{n+1}{2}} \deg(f) \cdot \chi(\SSS^{n+1}).
  \end{equation*}
\end{proposition}

Let  $S_M \to (M, \gamma)$ be the spinor bundle.
Let $\delta$ be a smooth Riemannian metric on $M$ and consider the continuous family of $W^{1,p}$-metrics on $M$
\begin{equation}\label{eq:familyOfMetrics}
  g_t := (1-t)g + t\delta, \qquad t\in [0,1] .
\end{equation}
For $t \in [0,1]$, let 
\begin{equation} \label{Dirac_oncemore}
   D_t := L_{g_t}^2(M, S_M) \supset H^1_{g_t}(M, S_M) \to L_{g_t}^2(M, S_M) 
\end{equation}
be the associated family of Dirac operators.
By Lemma \ref{scal_nearby} and since smooth Riemannian metrics are dense in the space of $W^{1,p}$-regular Riemannian metrics on $M$, we can choose $\delta$ such that for each Hermitian vector space $F$, the spectrum of the operator 
\[
   D_{F,t} = D_{t} \otimes \id_F \colon L_{g_t}^2(M, S_M \otimes F) \supset H^1_{g_t} (M, S_M \otimes F) \to L_{g_t}^2(M, S_M \otimes F)
\]
has empty intersection with $ [-\tfrac{1}{2}, + \tfrac{1}{2}]$. 
Since these spectra are independent of the choice of the smooth background metric $\gamma$, we can henceforth assume without loss of generality that this property holds for the choice $\delta := \gamma$.

\begin{proof}[Proof of Proposition \ref{prop:IndexFormula}] We construct a continuous deformation of $\bar \Dirac\vphantom{\Dirac}^+_E$ through Fredholm operators to a Dirac operator on a smooth Riemannian manifold with conical singularities whose index can be computed by the results  in \cite{Chou}.

Let $\varphi \in C^{\infty}((0,\pi), \RR)$ be a smooth positive function such that $\varphi(r) = r$ for $r \in (0,1)$ and $\varphi(r) = \pi - r$ for $r \in (\pi-1,1)$.
    For $t \in [0,1]$, set $\sigma_t := t\varphi(r)  + (1-t) \sin(r)$ and consider the warped product metric on $\susp M$ given by 
    \begin{equation*}
      \Gamma_t := d r^2 + \sigma_t(r)^2 g_t(r). 
    \end{equation*}
Note, that the metric $\Gamma_1$ is smooth and strictly conical, i.e., equal to $d r^2 + r^2 \gamma$, on $(0,1) \times M$, and equal to $d r^2 + (\pi-r)^2 \gamma$ on $(\pi-1 ,\pi) \times M$.

Take a local trivialization of the pullback bundle $E^{\pm} = (\susp f)^*(E_0^{\pm})$ over $\mathscr{C} = \left( ( 0,1) \cup (\pi-1,\pi) \right) \times M$, using local trivializations of $E_0^{\pm}$ near the north and south poles of $\SSS^{n+1}$  as in \eqref{trivialsphere}.
 By Proposition \ref{lem:smoothStructureOnE}, we find a smooth structure on $E^{\pm}$ that is compatible with the Lipschitz structure and with  this trivialization of $E$ on $\mathscr{C}$.
Choose a smooth Hermitian bundle metric on $E^{\pm}$ and a smooth metric connection $\widetilde \nabla^{E^{\pm}}$ on $E^{\pm}$ that has a vanishing local connection 1-form for this trivialization over $\mathscr{C}$.
Set
    \begin{equation*}
      \nabla^{E^{\pm}}_t := t \widetilde\nabla^{E^{\pm}} + (1-t)\nabla^{E^{\pm}}.
    \end{equation*}
Note that, for $0 \leq t <1$, this connection is not necessarily metric.
 Let 
 \[
     \bar\Dirac\vphantom{\Dirac}^+_{E,t} \colon L_{\Gamma_t}^2(\susp M, (\spinor \otimes E)^+) \supset \dom(\bar\Dirac\vphantom{\Dirac}^+_{E,t}) \to L_{\Gamma_t}^2(\susp M, (\spinor \otimes E)^-) 
 \]
be the Dirac operator for the metric $\Gamma_t$ on $\susp M$ twisted with $(E,\nabla_t^E)$. As in the proof of Proposition \ref{coneexample}, it follows that each $\Dirac_{E,t}^+$ is an abstract cone operator.

To apply Proposition \ref{lem:deformationLemma}, we  conjugate the $\Dirac_{E,t}^+$ with a family of isometries to abstract cone operators $L^2_{\Gamma_1}(\susp M, (\spinor\otimes E)^+) \to L^2_{\Gamma_1}(\susp M, (\spinor\otimes E)^-)$ between fixed Hilbert spaces.
For this, let 
\begin{equation*}
  \rho^{\Gamma_t} := \tfrac{d\mu^{\Gamma_t}}{d\mu^{\Gamma_1}} \in W^{1,p}_{\loc}(\susp M)
\end{equation*}
  be the volume density of $\Gamma_t$ with respect to $\Gamma_1$.
  It is non-zero almost everywhere.
  Multiplication with $\left(\rho^{\Gamma_t}\right)^{\tfrac12}$ induces isometries
  \begin{equation*}
    \Psi_t^\pm\colon L^2_{\Gamma_t}(\susp M, (\spinor \otimes E)^{\pm}) \to L^2_{\Gamma_1}(\susp M, (\spinor \otimes E)^{\pm}).
  \end{equation*}

  We claim that $(\Psi_t^+)^{-1}$, which is multiplication with $\left(\rho^{\Gamma_t}\right)^{-\tfrac12}$, restricts to a map
  \begin{equation}\label{RestrictionLipcToDomDirac}
    \Lip_\cc(\susp M, (\spinor \otimes E)^+) \subset H^1_{\cc,\Gamma_1}(\susp M, (\spinor \otimes E)^+) \to H^1_{\cc, \Gamma_t}(\susp M, (\spinor \otimes E)^+)\subset  \dom(\bar\Dirac\vphantom{\Dirac}^+_{E,t}),
  \end{equation}
  where $H^1_{\cc,\Gamma_t}$ denotes compactly supported $H^1$-sections.

First, we show that for the image we have
\begin{equation*}
  (\Psi_t^+)^{-1} \left(H^1_{\cc, \Gamma_1}(\susp M, (\spinor \otimes E)^{\pm})\right) \subset H^1_{\cc, \Gamma_t}(\susp M, (\spinor \otimes E)^{\pm}).
\end{equation*}
On compact subsets $\Omega= [a,\pi-a]\times M\subset \susp M$ the Sobolev norm is independent of the choice of the metric and the connection up to equivalence. 
Therefore, it is enough to show that multiplication with any real valued function $\eta\in W^{1,p}_\loc(\susp M)$ induces a map $H^1(\Omega, (\spinor\otimes E)^+)\to H^1(\Omega, (\spinor\otimes E)^+)$.
Let $u\in H^1(\Omega, (\spinor\otimes E)^+)$. 
We apply the Leibniz rule and H\"older's inequality to compute
\begin{equation*}
  \norm{\nabla(\eta u)}_{L^2} \leq \norm{d\eta\otimes u}_{L^2} + \norm{\eta \nabla u}_{L^2} \leq \norm{d\eta}_{L^p}\norm{u}_{L^q} + \norm{\eta}_{L^2}\norm{\nabla u}_{L^2},
\end{equation*}
where $q = \frac{2p}{p-2}$ and all norms are considered over $\Omega$. 
By the Sobolev embedding theorem, we have $H^1\subset L^q$,  and $\eta \in L^2$, as $p>n+1$.
Hence, the right-hand side is finite and $\eta u \in H^1(\Omega, (\spinor\otimes E)^+)$.

  It is left to show that the compactly supported $H^1$ sections are contained in $\dom(\bar\Dirac\vphantom{\Dirac}^+_{E,t})$.
  For $u\in H^1_\cc(\susp M, (\spinor \otimes E)^{\pm})$ we find $a>0$ such that $\supp(u)\subset (a,\pi-a)\times M$.
  Consider the smooth double $\mathsf{D}\Omega$ of $\Omega$ with a $W^{1,p}$-metric $\widetilde g$ that restricts to $g$ on $\Omega$ and extensions $\widetilde \spinor \vphantom{\spinor}^\pm\otimes \widetilde E$ of the twist bundles to the double as in the proof of Proposition \ref{constr_int_parametrix}. 
  The minimal closed extension of the twisted Dirac operator $\mathscr{D}$ associated with $\widetilde g$ is $H^1(\mathsf D \Omega, \widetilde \spinor\vphantom\spinor^+\otimes \widetilde E)$ and contains the section $u$ extended by zero to $\mathsf D\Omega$.
  Therefore, we can approximate $u$ by Lipschitz sections, with support contained in $(a,\pi-a)\times M$, in the graph norm of $\mathscr{D}$.
  Since the restriction of $\mathscr{D}$ to $(a,\pi-a)\times M$ coincides with $\Dirac_{E,t}^+$, we obtain a sequence of Lipschitz sections $(u_i)$ with support in $(a,\pi-a)\times M$ that converges to $u$ in the graph norm of $\bar\Dirac\vphantom{\Dirac}^+_{E,t}$.

  As a consequence of \eqref{RestrictionLipcToDomDirac}, we obtain a well-defined family of operators
  \begin{equation*}
    \Dirac_t:= \Psi_t^-\circ \bar\Dirac\vphantom{\Dirac}^+_{E,t}\circ (\Psi_t^+)^{-1}\colon L^2_{\Gamma_1}(\susp M, (\spinor \otimes E)^+) \supset \Lip_\cc(\susp M, (\spinor\otimes E)^+) \to L^2_{\Gamma_1}(\susp M, (\spinor \otimes E)^-).
  \end{equation*}

  We will varify that each $\Dirac_t$ is an abstract cone operator. As in Section \ref{sec:Dirac_as_regsingop}, we demonstrate the argument only for the cone tip at $r=0$. 
  The argument for the cone tip at $r=\pi$ is analogous.

  Let
  \begin{enumerate}[label=\myicon]
  \item  $\EE := L_{\Gamma_1}^2(\susp M,\spinor^+ \tensor E)$, 
  \item $\FF := L_{\Gamma_1}^2(\susp M,\spinor^- \tensor E)$,
  \item $\LL:= L^2_\gamma(M, S_M \otimes F)$,
  \end{enumerate}
  where $F$ is the fibre of $E_0$ over the south pole of $\SSS^{n+1}$, see \eqref{trivialsphere}.
  On $U = (0,1) \times M \subset \susp M$, we have
\begin{equation*}
  \rho^{\Gamma_t} = \left(\tfrac{\sigma_t(r)}r\right)^n\tfrac{d\mu^{g_t}}{d\mu^{\gamma}}.
\end{equation*}
Set
\begin{equation*}
  \rho^{g_t} := \tfrac{d\mu^{g_t}}{d\mu^{\gamma}}.
\end{equation*}
Similar to \eqref{L2Isom1}  and \eqref{L2Isom2}, we obtain isometries
\begin{align*} 
  \Phi_{\EE} \colon L_{\Gamma_1}^2(U, \spinor^+\otimes\ E) & \to L^2 ( (0,1) , \LL ), \\
  \Phi_{\FF} \colon L_{\Gamma_1}^2(U, \spinor^- \otimes\ E) & \to L^2((0,1),  \LL ) 
\end{align*}
such that 
  \begin{equation*}
    \Phi_{\FF}\circ\Dirac_t \circ (\Phi_{\EE})^{-1} =\partial_r + \tfrac1{\sigma_t(r)} \left(\rho^{g_t}\right)^{\tfrac12}\cdot \left( D_t\otimes \id_F\right)\cdot \left(\rho^{g_t}\right)^{-\tfrac12} + (1-t)\sum_{i=1}^{n} c \left( e_i \right) \otimes ((\susp f)^* \omega)   \left( \big(0,  \tfrac{1}{\sin(r)} e_i^{g_t}\big) \right)
  \end{equation*}
  for $(e_1, \ldots, e_{n})$ a local $\gamma$-orthonormal frame of $TM$.

  We define the link operator by
  \begin{equation*}
    S_{0,t} = \left(\rho^{g_t}\right)^{\tfrac12}\left( D_t\otimes \id_F\right) \left(\rho^{g_t}\right)^{-\tfrac12} \colon \LL \supset \Lip(M, S_M \otimes F) \to \LL,
  \end{equation*}
  which is well-defined as multiplication by $\left(\rho^{g_t}\right)^{-\tfrac12}$ maps Lipschitz sections into the domain of $D_t$.
  Note that this operator is essentially self-adjoint. 
  Furthermore, we define the perturbation operator by
  \[
    S_{1,t}(r) =  r \cdot \Bigg( \left(\frac{1}{\sigma_t(r)} - \frac{1}{r}\right) S_{0,t} + (1-t)\sum_{i=1}^{n} c \left( e_i \right) \otimes ((\susp f)^* \omega)   \left( \big(0,  \tfrac{1}{\sin(r)} e_i^{g_t}\big) \right) \Bigg) .
  \]
  
  The domain $\mathfrak{D}_{\link} = \Lip(M, S_M \otimes F)$ of the link operators is independent of $t$ and the graph norms on $\mathfrak{D}_{\link}$ induced by $S_{0,t}$ are pairwise equivalent. 
  Furthermore, the map 
    \begin{equation*}
      [0,1]\to \boundedops(\mathfrak{D}_{\link},  \LL), \quad t \mapsto S_{0,t},
    \end{equation*}
    is continuous.
    The spectrum of $\bar S_{0,t}$ has empty intersection with $[-\tfrac{1}{2}, + \tfrac{1}{2}]$ by our choice of $\gamma$.
   Note also that 
        \begin{equation*} [0,1] \to L^{\infty}((0,\thet), \boundedops(\mathfrak{D}_{\link} , \LL)), \quad t \mapsto S_{1,t} 
        \end{equation*}
   is continuous.

The constant $0 < \thet \leq 1$, such that \ref{boundedperturbtwo} holds over some $(0,\thet)$, can be chosen independent of $t$.
Each $\Dirac_t$ has an interior parametrix so that their minimal closed extensions are Fredholm.

Proposition \ref{lem:deformationLemma} shows that $\mathfrak{D} := \dom(\bar
\Dirac_t)$ is independent of $t$, the graph norms of $\bar \Dirac_t$ on
$\mathfrak{D}$ are independent of $t$, up to equivalence with constant which
can be chosen independently of $t$, and the map 
\[
    [0,1] \to \boundedops(\mathfrak{D}, \FF), \qquad t \mapsto \bar \Dirac_t
\]
is continuous.
In particular, the Fredholm indices of $\bar \Dirac_0 = \Psi_0^-\circ \bar \Dirac\vphantom{\Dirac}^+_{E}\circ (\Psi_0^+)^{-1}$ and of $\bar \Dirac_1 = \bar \Dirac\vphantom{\Dirac}^+_{E,1}$ are equal. 
As $\Psi_0^\pm$ is a unitary isomorphism, we have
\begin{equation*}
  \ind(\bar \Dirac\vphantom{\Dirac}^+_{E})  = \ind(\bar \Dirac_0) = \ind(\bar \Dirac\vphantom{\Dirac}^+_{E,1}).
\end{equation*}

The twisted Dirac operator  $\Dirac_{E,1}^+$ is associated with a smooth Riemannian metric on $\susp M$ which is strictly conical on $\mathscr{C}$  and with a smooth Hermitian twist bundle $E$ equipped with a metric connection so that $E$ is trivialized and  flat over $\mathscr{C}$.

By \cite{Chou}*{Theorem (3.2) and remarks at the end of page 5}, the associated twisted Dirac operator for these data,
\begin{equation} \label{DiracChou}
    \mathcal{D}_{E,1} \colon L_{\Gamma_1}^2(\susp M , \spinor \otimes E) \supset \Lip(\susp M, \spinor \otimes E) \to L_{\Gamma_1}^2( \susp M, \spinor \otimes E),
 \end{equation}
 is essentially self-adjoint.
 To compute the index of $\bar \Dirac\vphantom{\Dirac}^+_{E,1}$, we observe that, due to \eqref{eq:decomp_of_plustwistbundle} and \eqref{eq:decomp_of_minusplustwistbundle}, the operator 
  \begin{equation*}
    \Dirac_{E,1}^+ \colon C_\cc^{\infty}(\susp M, (\spinor \otimes E)^+) \to C_\cc^{\infty}( \susp M, (\spinor \otimes E)^-)
  \end{equation*}
  decomposes as
  \begin{equation*}
    \Dirac_{E,1}^+ = \Dirac^{++} \oplus \Dirac^{--}.
  \end{equation*}
Here
  \begin{align*}
    \Dirac^{++} \colon C_\cc^{\infty}( \susp M , \spinor^+ \otimes\ E^+ ) &\to C_\cc^{\infty}(\susp M, \spinor^- \otimes\ E^+), \\
    \Dirac^{--} \colon C_\cc^{\infty}( \susp M , \spinor^- \otimes\ E^- ) &\to C_\cc^{\infty}(\susp M, \spinor^+ \otimes\ E^-),
  \end{align*}
  are restrictions of the operator $\mathcal{D}_{E,1}$ from \eqref{DiracChou}.
  Since $\mathcal{D}_{E,1}$ is essentially self adjoint, we obtain
    \begin{equation*}
    \ind ( \bar \Dirac\vphantom{\Dirac}^+_{E,1}) = \ind( \bar \Dirac\vphantom{\Dirac}^{++} ) + \ind (\bar \Dirac\vphantom{\Dirac}^{--}).
  \end{equation*}

  Let $\omega_{E^\pm}$ denote the index forms of the twisted Dirac operators induced by $\Dirac_{E,1}$,
  \begin{equation*}
       C_\cc^{\infty}( \susp M, \spinor^+ \otimes\ E^{\pm} ) \to C_\cc^{\infty} (\susp M, \spinor^-  \otimes\ E^{\pm}).
  \end{equation*} 
  Since $\Dirac^{++}$ and $\Dirac^{--}$ correspond to smooth conical Riemannian metrics on $\susp M$ and a connection on $E$ that is flat near the tips, we can apply \cite{Chou}*{Remark 5.25} to obtain
  \begin{align*} 
   \ind ( \bar \Dirac\vphantom{\Dirac}^{++} ) & = +  \int_{\susp M} \omega_{E^+} - \frac{\eta_{S_0}}{2}  , \\
   \ind ( \bar \Dirac\vphantom{\Dirac}^{--}  ) & = - \int_{\susp M} \omega_{E^-} + \frac{\eta_{S_0}}{2} . 
  \end{align*} 
  Here $\eta_{S_0}$ denotes the $\eta$-invariant of $S_0$.
  Adding the indices, the $\eta$-terms cancel, and 
  \[
    \ind( \bar \Dirac\vphantom{\Dirac}^+_E) = \int_{\susp M} \omega_{E^+} - \omega_{E^-} =  \int_{\susp M} f^*( \ch(E_0^+) - \ch(E_0^-)) = (-1)^{\tfrac{n+1}{2}}\deg(f) \chi(\SSS^{n+1}).
  \]
\end{proof}

\subsection{From area contracting to length contracting comparison maps}
\label{sec:infinitesimal_isometry}

\begin{proposition}\label{P:infinitesimal_isometry}
Let $M$ be a closed smooth connected  oriented manifold of odd dimension $n\geq 3$ which admits a spin structure, let $g$ be a $W^{1,p}$-regular  Riemannian metric for some $p > n$ with (distributional) scalar curvature $\scal_g \geq n(n-1)$, and let $f \colon (M,d_g) \to \SSS^n$ be a  Lipschitz continuous map of non-zero degree such that $d_x f$ is area non-increasing almost everywhere.

Then  $d_xf$ is an isometry almost everywhere. 
Furthermore, $f$ is $1$-Lipschitz.
\end{proposition}

We follow the argument in \cite{Baer_2024} and \cite{LiSuWang} with some modifications to adapt it to the lower regularity setting. 

Choose a smooth background metric $\gamma$ on $M$ and let $S_M\to M$ denote the associated spinor bundle with spin connection $\nabla$.
Let 
\begin{equation*}
  D_g \colon H^1_g(M,S_M) \to L^2_g(M,S_M)
\end{equation*}
be the Dirac operator associated with the metric $g$.
Due to Friedrich's inequality, we have the following spectral gap
\begin{equation}\label{DiracSpectralGap}
  \spec(D_g) \cap \left(-\tfrac n2,\tfrac n2\right) = \emptyset.
\end{equation}
Furthermore, let $E_0 \to \SSS^n$ be the spinor bundle with spin connection $\nabla^{E_0}$ over the round sphere $\SSS^n$.
We consider a family of metric connections on $E_0$
\begin{equation*}
  \nabla_X^{E_0,s} u = \nabla_X^{E_0} u + sX\cdot u ,
\end{equation*}
where $s\in \RR$, $X\in TM$ and $X\cdot$ denotes Clifford multiplication by $X$.
For $s=\frac12$ and $s=-\frac12$ there are $\nabla^{E_0,s}$-parallel spinors on $\SSS^n$, called \emph{Killing spinors}, that form a global orthonormal frame of $E_0$ and hence, trivialize $E_0$.

We consider  
\begin{equation*}
  (E, \nabla^{E,s}):= (f^*E_0, f^*\nabla^{E_0,s}) \to M,
\end{equation*}
which is a Hermitian Lipschitz bundle equipped with a family of metric Lipschitz connections $\nabla^{E,s}$.
The Killing spinors on $\SSS^n$, for $s = \frac12$ and $s=-\frac12$, induce unitary Lipschitz trivializations
\begin{equation*}
  U_+ \colon E \to  \underline\CC^{2^k} \quad \text{and} \quad U_- \colon E \to \underline\CC^{2^k},
\end{equation*}
where $\underline \CC^{2^k} = M\times \CC^{2^k}$ denotes the trivial bundle with the canonical Hermitian structure and $n = 2k+1$. 
Thus, we obtain unitary vector bundle isomorphisms
\begin{align*}
  V_+ := \id \otimes\, U_+ \colon S_M\otimes E &\to S_M\otimes \underline \CC^{2^k}, \\
  V_- := \id \otimes\, U_- \colon S_M\otimes E &\to S_M\otimes \underline \CC^{2^k}.
\end{align*}

Let 
\begin{equation*}
  \nabla^s := \nabla \otimes 1 + 1\otimes \nabla^{E,s}
\end{equation*}
be the twist connection on $S_M\otimes E$ and let
\begin{equation*}
  D_{g,s} \colon H^1_g(M, S_M\otimes E)\to L^2_g(M, S_M\otimes E)
\end{equation*}
be the twisted Dirac operator.

There are $L^\infty$-endomorphism fields $Z_+$ and $Z_-$ on $S_M\otimes \underline\CC^{2^k}$ such that (see \cite{Baer_2024}*{Lemma 1})
\begin{equation*}
  V_+ \circ D_{g,s}\circ V_+^{-1} = D_g + \left(s- \tfrac12\right)Z_+
\end{equation*}
and 
\begin{equation*}\label{DiracFlowUnitaryTransform}
  V_- \circ D_{g,s}\circ V_-^{-1} = D_g + \left(s + \tfrac12\right)Z_-.
\end{equation*}
This implies for the spectra
\begin{equation}\label{CoincidingSpectra}
  \spec(D_{g,-\frac12}) = \spec(D_{g,\frac12}) = \spec(D_g) .
\end{equation}

We denote the spectral flow of the family $D_{g,s}$, for $s\in \left[-\frac12, \frac12\right]$, with $\specflow(D_{g,s})$.

\begin{lemma}\label{SpekFlowNonTrivial}
  We have
  \begin{equation*}
    \abs{\specflow(D_{g,s})} = \abs{\deg(f)}.
  \end{equation*}
  
  In particular, there is at least one $s\in \left[-\frac12,\frac12\right]$ such that $\ker(D_{g,s})\neq \{0\}$.
\end{lemma}

\begin{proof}
  Let $\delta$ be a smooth metric and consider the family of metrics
  \begin{equation*}
    g_t := (1-t)g + t\delta
  \end{equation*}
  for $t\in [0,1]$.
  Since the smooth metrics are dense in the $W^{1,p}$-metrics and by \eqref{DiracSpectralGap}, we can arrange that 
  \begin{equation}\label{DiracDeformSpectralGap}
    \ker(D_{g_t})= \{0\}\qquad \forall t\in [0,1].
  \end{equation}
  Without loss of generality we can assume that $\delta$ coincides with the background metric $\gamma$.
  Let
  \begin{equation*}
    \rho^{g_t} := \tfrac{d\mu^{g_t}}{d\mu^{\gamma}} \in W^{1,p}(M)
  \end{equation*}
  be the volume density of the metric $g_t$ with respect to the background metric $\gamma$. 
  Multiplication with $(\rho^{g_t})^{\frac12}$ induces a unitary transformation 
  \begin{equation*}
    \varphi_t \colon L^2_{g_t}(M, S_M\otimes E) \to L^2_\gamma(M, S_M\otimes E).
  \end{equation*}
  Due to H\"older's inequality and the Sobolev embedding theorem, the map $\varphi_t$ maps $H^1_{g_t}(M, S_M\otimes E)$ to $H^1_\gamma(M, S_M\otimes E)$. We obtain a map 
  \begin{equation*}
    H\colon \left[-\tfrac12, \tfrac12\right]\times [0,1] \to \boundedops\big(H^1_\gamma(M, S_M\otimes E), L^2_\gamma(M, S_M\otimes E)\big), \qquad (s,t) \mapsto \varphi_t \circ D_{g_t,s} \circ \varphi_t^{-1},
  \end{equation*}
  which has image in the set of unbounded self-adjoint Fredholm operators on $L^2_\gamma(M, S_M\otimes E)$. 
  Since the coefficients of $D_{g_t,s}$ depend continuously on $t$ and $s$, the map is continuous with respect to the gap topology on the space of self-adjoint Fredholm operators (see \cite{BoosLeschPhillips}).
  Due to \eqref{CoincidingSpectra} and \eqref{DiracDeformSpectralGap}, the operators $H\left(-\tfrac12,t\right)$ and $H\left(\tfrac12,t\right)$ are invertible for all $t\in [0,1]$. 
  Thus, it follows from \cite{BoosLeschPhillips}*{Proposition 2.3} that
  \begin{equation*}
    \specflow(D_{g,s})= \specflow(D_{\gamma,s}).
  \end{equation*}

  As in \cite{Baer_2024} we deform $f$ into a smooth map $f_1\colon M\to \SSS^n$ such that $\deg(f) = \deg(f_1)$.
  This induces a homotopy of self-adjoint Fredholm operators on $L^2_\gamma(M, S_M\otimes E)$ and with the same argument as in \cite{Baer_2024}*{Section 5.2} the claim follows. 
\end{proof}

Let $R^{s}\in \Omega^2(\SSS^n, \End(E_0))$ be the curvature form of $(E_0, \nabla^{E_0, s})\to \SSS^n$ and $(e_1,\ldots ,e_n)$ a $\gamma$-orthonormal frame of $TM$.
We define the curvature operator on simple tensors $u = \sigma \otimes \eta \in (S_M\otimes E)_x$ by
\begin{equation*}
  \mathcal R^{E,s}_g u := \tfrac12 \sum_{i,j=1}^n e_i\cdot e_j \cdot \sigma \otimes (f^*R^{s})_{e_i^g, e_j^g}\eta
\end{equation*}
for every $x\in M$ where $f$ is differentiable. 
Here we adopted the notation $e_i^g = b_g(e_i)$ from \eqref{bIsometry}.

By the singular value decomposition of $d_x f\circ b_g$, we find a $\gamma$-orthonormal basis $(e_1,\ldots, e_n)$ of $T_xM$, a $g_{\SSS^n}$-orthonormal basis $(\varepsilon_1, \ldots, \varepsilon_n)$ of $T_{f(x)}\SSS^n$ and real numbers $\mu_i \geq 0$ such that $d_x f(e_i^g) = \mu_i \varepsilon_i$.
Moreover, a straightforward computation shows that the family of connections $\nabla^{E_0,s}$, for $s\in \left[-\frac12, \frac12\right]$ corresponds to $\nabla_s = d + \left(s+\frac12\right)\omega$ in \cite{LiSuWang}. 
Therefore, \cite{LiSuWang}*{Equation 2.7} yields
\begin{equation}\label{CurvatureFamily}
  \mathcal R_g^{E,s} u = \left(s+\tfrac12\right)\left(s-\tfrac12\right) \sum_{i\neq j}\mu_i\mu_j \big(e_i\cdot e_j\otimes \varepsilon_i\cdot \varepsilon_j\big) \cdot u \qquad \mathrm{ a.e.}
\end{equation} 

\begin{proof}[Proof of Proposition \ref{P:infinitesimal_isometry}]
  Let $s\in \left[-\frac12, \frac12\right]$ such that $\ker(D_{g,s})\neq \{0\}$, which exists by Lemma \ref{SpekFlowNonTrivial}, and choose $u \in \ker(D_{g,s})\setminus\{0\}$.

  At every point $x\in M$, where $f$ is differentiable, we deduce from \eqref{CurvatureFamily} and $\abs{\Lambda^2 d_xf}\leq 1$ that
  \begin{equation}\label{CurvatureInequality}
    \innerprod{\mathcal R_g^{E,s}u, u} \geq -\tfrac14 n(n-1)\abs{u}^2 
  \end{equation}
  with equality if and only if $d_xf\colon (T_xM, g_x)\to T_{f(x)}\SSS^n$ is an isometry.

  We insert $u$ into the integrated Schr\"odinger-Lichnerowicz formula (see \cite{CHS}*{Theorem 5.1}) and use $\scal_g \geq n(n-1)$ in the distributional sense to get
  \begin{equation*}
    0 = \norm{D_{g,s}u}^2_{L^2} = \norm{\nabla^s u}_{L^2}^2 + \tfrac14 \llangle  \scal_g , |u|^2  \rrangle + \left(\mathcal R_g^{E,s}u ,u\right)_{L^2} \geq 0.
  \end{equation*}
  It follows that we are in the equality case of \eqref{CurvatureInequality} and hence, $d_xf$ is an isometry almost everywhere.

The fact that $f$ is $1$-Lipschitz follows from the same  argument used in the last part of the proof of \cite{CHS}*{Proposition 2.14}.
This completes the proof of Proposition \ref{P:infinitesimal_isometry}.

\end{proof}

\subsection{Proof of Theorem \ref{theo:main_odd}}

After establishing an index formula and a Schr\"odinger-Lichnerowicz formula for the Dirac operator on the spherical suspension, the remainder of the proof of Theorem \ref{theo:main_odd} follows the approach in \cite{CHS}.

We work in Setup \ref{setup} where $E_0 = E_0^+ \oplus E_0^- \to \SSS^{n+1}$ is the spinor bundle over the even dimensional sphere $\SSS^{n+1}$.
Let 
\begin{equation*}
  \bar\Dirac_E \colon L^2_{\susp g}( \susp M, \spinor \otimes E) \supset \dom(\bar\Dirac_E) \to L^2_{\susp g}(\susp M, \spinor\otimes E).
\end{equation*}
be the Dirac operator on the spherical suspension $(\susp M, \susp g)$.
We computed the index of the positive part of this operator in Proposition \ref{prop:IndexFormula}:
\begin{equation*}
  \ind \left(\bar\Dirac\vphantom{\Dirac}^+_E\right) = (-1)^{\tfrac{n+1}{2}} \deg (f)\chi (\SSS^{n+1}) .
\end{equation*}
By assumption, the degree of $f$ is non-zero and the Euler characteristic of the even dimensional sphere is 2. 
Hence, possibly after changing the orientation of $M$,  the index of  $ \bar\Dirac\vphantom{\Dirac}^+_E$ is positive.
We thus obtain a non-zero harmonic spinor field  $\psi \in \dom \left(\bar \Dirac\vphantom{\Dirac}^+_E\right)\subset L^2_{\susp g}(\susp M, (\spinor\otimes E)^+)$.

By assumption, $f$ is area non-increasing.
Hence, by Proposition \ref{P:infinitesimal_isometry}, $df$ is an isometry almost everywhere. 
This implies that $d( \susp f)$ is an isometry almost everywhere. 

With  \cite{CHS}*{Proposition 6.1} we obtain a pointwise estimate
\begin{equation}\label{eq:curv_lowerbound}
  \langle \mathcal{R}^E \omega, \omega \rangle\geq -\tfrac14 (n+1)n |\omega|^2 \qquad \forall \omega \in (\spinor\otimes E)_x\  
\end{equation}
at every point $x\in \susp M$ where $\susp f$ is differentiable. 

Plugging the harmonic spinor into the integral  Lichnerowicz formula, by Theorem \ref{thm:SL_cone} and  \eqref{eq:curv_lowerbound}, we get
\begin{align*}
  0 = \| \bar \Dirac_E\psi \|_{L^2}^2 &= \| \nabla_E \psi \|_{L^2}^2 + \tfrac14 \llangle  \scal_{\susp g} , |\psi|^2  \rrangle + \bigl( \mathcal{R}^E \psi, \psi \bigr)_{L^2} \\
  & \geq \tfrac14\llangle  \scal_{\susp g} , |\psi|^2  \rrangle -\tfrac14 (n+1)n \|\psi \|_{L^2}^2 \\
  &\geq 0 .
\end{align*}
The last inequality follows from the assumption $\scal_g \geq n(n-1)$, which implies $\scal_{\susp g} \geq (n+1)n$ by Lemma \ref{scalinsuspension}.
Hence, we have
\begin{align}
  &\|  \nabla_E \psi \|_{L^2} = 0 \label{eq:parallelharmspinor},\\
  &\langle \mathcal{R}^E \psi, \psi \rangle = -\tfrac14 (n+1) n |\psi |^2 \quad \mathrm{a.e.}  \label{eq:curv_equality},\\
  &\scal_{\susp g} =  (n+1)n \quad \text{in the distributional sense}. \label{eq:equality_scalcurv}  
\end{align}
From \eqref{eq:parallelharmspinor} it follows that for $|\psi|\in H^1(\susp M)$ we have
\begin{equation*}
  d | \psi |^2 = 2 \langle \nabla_E\psi,\psi  \rangle = 0\qquad \text{a.e.}
\end{equation*}
and therefore, that $|\psi| $ is constant almost everywhere. 
Consequently, $\psi$ only vanishes on a set of measure zero.

Equation \eqref{eq:curv_equality} says that we are in the equality case of \cite{CHS}*{Proposition 6.1} which implies that for almost all  $x\in \susp M$  where $\susp f$ is differentiable, and for all $\gamma$-orthonormal vectors $v,w \in T_x\susp M$ we have
\begin{equation}\label{eq:doubleCM_invariance}
  \bigl( v\cdot w \otimes d_x\susp f(v^{\susp g})\cdot d_x\susp f (w^{\susp g}) \bigr) \cdot \psi = \psi .
\end{equation}

\begin{proposition}\label{prop:differential_orientpres}
  At all points $x \in \susp M$, where $\susp f$ is differentiable, the differential $d_x \susp f$ is an orientation preserving isometry.
\end{proposition}

\begin{proof}
  Let $(e_1, \ldots , e_{n+1} )$ be an $\susp \gamma$-orthonormal frame of $T\susp M$ around $x\in \susp M$ and $( \varepsilon_1, \ldots ,\varepsilon_{n+1} )$ a $g_0$-orthonormal frame of $TS^{n+1}$ around $\susp f(x)$. The volume elements of the respective Clifford algebra bundles are defined by
  \begin{equation*}
    \mathrm{vol}_{\susp \gamma} = e_1\cdot \ldots \cdot e_{n+1} \in \mathrm{Cl}(\susp M)
  \end{equation*}
  and
  \begin{equation*}
    \mathrm{vol}_{g_0} = \varepsilon_1  \cdot \ldots \cdot \varepsilon_{n+1} \in \mathrm{Cl}(S^{n+1}) . 
  \end{equation*}

  The element $\mathrm{vol}_{\susp \gamma} \otimes \mathrm{vol}_{g_0}$ acts on $(\spinor\otimes E)^+$ as identity. 

  We already showed that at almost all points $x \in \susp M$, the differential $d_x \susp f$ is an isometry.
  For such $x$ we have
  \begin{equation} \label{eq:differential_and_volelement}
    d_x \susp f (e_1^{\susp g}) \cdot \ldots \cdot d_x \susp f (e_{n+1}^{\susp g})  = \pm \mathrm{vol}_{g_0} ,
  \end{equation}
  where the sign depends on whether $d_x \susp f$ is orientation preserving or reversing. 

  By assumption, $n+1$ is even and hence we can group the above basis vectors in pairs $(v,w) = (e_1, e_2), (e_3, e_4) ,\ldots, (e_n, e_{n+1})$.
 An iterative application of  \eqref{eq:doubleCM_invariance} shows
  \begin{equation*}
    \bigl( e_1 \cdot \ldots \cdot e_{n+1} \otimes d_x \susp f (e_1^{\susp g}) \cdot \ldots \cdot d_x \susp f (e_{n+1}^{\susp g}) \bigr)\cdot \psi = \psi .
  \end{equation*}
  Combining this with \eqref{eq:differential_and_volelement} and the fact that $\psi$ has only a non-trivial contribution from $(\spinor\otimes E)^+$ and only vanishes on a set of measure zero, we conclude that $d_x\susp f$ has to be orientation preserving. 
\end{proof}

Combining Proposition \ref{prop:differential_orientpres} with \cite{CHS}*{Theorem 2.4},  the proof of Theorem \ref{theo:main_odd} is complete.

\section{Lipschitz rigidity for manifolds with cone-like singularities} \label{Lip_Cone}

\subsection{Scalar curvature of generalized cone metrics}

Let $\conic g_r:= d r^2 + r^2 g_r$ be a generalized cone metric on $(0, \vartheta) \times M$ in the sense of Definition \ref{generalizedConeMetric}.

\begin{proposition}\label{LinkScalCurvBehaviour}
 For $x \in M$, we have
 \begin{equation*}
    \lim_{r\to 0}r^2\scal_{\conic g_r}(r,x) = \scal_{g_0}(x) - n(n-1).
 \end{equation*}

In particular, if $\scal_{\conic g_r} \geq 0$, then we have $\scal_{g_0}\geq n(n-1)$. 
\end{proposition}

\begin{proof}
 Let $W_r = - \nabla^{\conic g_r} \partial_r$ be the Weingarten map.
  Then we have (see \cite{BaerGauduchonMoroianu}*{Equation 4.8})
   \begin{equation*}
       \scal_{\conic g_r} = \scal_{r^2g_r} + 3 \tr (W_r^2) - \tr (W_r)^2 - \tr_{r^2 g_r}(\partial_r^2 (r^2g_r)).
    \end{equation*}

Recall, e.g.~from  \cite{BaerGauduchonMoroianu}*{Section 4}, that the Weingarten map satisfies
       \begin{enumerate}[label=\myicon]
	 \item $W_r(TM)\subset TM$,
	 \item $W_r(\partial_r) = 0$,
	 \item $W_r$ is symmetric with respect to ${\conic g_r}$,
	 \item $g_r(W_r(X),Y) = -\tfrac12 (\partial_r(r^2g_r))(X,Y)$ for all $X,Y\in TM$.
       \end{enumerate}

We have $\scal_{r^2g_r} = \tfrac1{r^2}\scal_{g_r}$.
Moreover, we compute, using the shorthand $\dot g_r := \partial_r g_r$ and $\ddot g_r := \partial_r^2 g_r$,
     \[
        \partial_r (r^2g_r)= 2r g_r + r^2 \dot g_r , \qquad \partial_r^2(r^2g_r) = 2g_r + 4r \dot g_r + r^2 \ddot g_r.
    \]

Let $x \in M$ and let $r \in (0,\thet)$. 
Since $\dot g_r$ is symmetric and bilinear, there is a local $g_r$-orthonormal frame $(e_1, \ldots, e_n)$ of $TM$ that is orthogonal with respect to $\dot g_r$.
Then $(\tfrac{1}{r} e_1, \ldots, \tfrac{1}{r} e_n)$ is a local  $r^2 g_{r}$-orthonormal frame on $TM$.

For the terms containing the Weingarten map we now compute 
  \begin{align*}
     \tr (W_r^2) = & \ \sum_{i = 1}^n  g_{r}  \Big(W_r^2 \left(\tfrac{1}{r}  e_i \right), \tfrac{1}{r}  e_i \Big)  \\
          =           &  \frac1{r^4} \sum_{i,j = 1}^n    g_r  \Big(W_r ( e_i),  e_j \Big) \cdot  g_r \Big( e_j , W_r(  e_i)\Big)\\
          =           &   \frac1{4r^4}\sum_{i,j=1}^n \big( (\partial_r (r^2g_r))( e_i,  e_j)\big)^2   \\
	 =            &   \frac1{4r^4}\sum_{i,j=1}^n  \Big( 4r^2 g_r(  e_i ,   e_j )^2 + 4r^3 g_r(  e_i ,   e_j )\cdot \dot g_r(  e_i,   e_j) + r^4  \, \dot g_r (  e_i,  e_j)^2  \Big)\\
	=             &  \frac1{r^2} \left( n + \sum_{i,j = 1}^n \Big( g_r(  e_i ,   e_j )\cdot r \dot g_r (  e_i ,   e_j ) + \frac14  r^2 \dot g_r (  e_i   ,e_j )^2\Big) \right) 
\end{align*}
   and
\begin{align*}
         \tr(W_r) = & -\frac12 \sum_{i = 1}^n  \partial_r(r^2g_r) \left(\tfrac1r   e_i , \tfrac1r   e_i \right)                 \\
            =                      & -\frac1{2r^2}\sum_{i= 1}^n \Big( 2rg_r(  e_i ,   e_i ) + r^2 \dot g_r(  e_i ,   e_i)\Big)  \\
	   =                      & - \frac 1r \Big( n +  \frac{1}{2} \sum_{i = 1}^n r \dot g_r(  e_i ,   e_i )\Big).
\end{align*}
For the last term we obtain
    \begin{align*}
           \tr_{r^2g_r} (\partial_r^2(r^2g_r)) = &  \frac1{r^2}\sum_{i=1}^n (\partial_r^2 (r^2g_r))(  e_i ,   e_i )                                                                                \\
	    =      &  \frac1{r^2}\sum_{i= 1}^n \Big( 2g_r(  e_i ,   e_i ) + 4r \dot g_r(  e_i ,   e_i ) + r^2 \ddot g_r(  e_i ,   e_i )\Big) \\
            =      &   \frac1{r^2}\left(2n + \sum_{i=1}^n \Big( 4r \, \dot g_r(  e_i ,   e_i ) + r^2 \ddot g_r(  e_i ,   e_i )\Big) \right).
\end{align*}
Summarizing, we have
\begin{align*}
 r^2 \scal_{\conic g_r} - \scal_{g_r} =&\  r^2 \Big( 3\tr(W_r^2) - \tr(W_r)^2 - \tr_{r^2g_r}\left(\partial_r^2 (r^2g_r)\right)\Big) \\
=&\ 3\Big( n + \sum_{i = 1}^n  r \dot g_r(e_i, e_i) + \frac14 r^2 \dot g_r (e_i, e_i)^2\Big) \\
&- \left( n^2 + \frac{n}{4}  \sum_{i=1}^n r \dot g_r(e_i, e_i) + 4\left( \sum_{i = 1}^n r \dot g_r (e_i, e_i)\right)^2\right) \\
&- \left(2n + \sum_{i = 1}^n 4r \dot g_r (e_i, e_i) + r^2 \ddot g_r (e_i, e_i)\right) \\
=&\  - n(n-1) - \sum_{i = 1}^n \Big( \left( \frac{n}{4} + 1 \right) r \dot g_r(e_i, e_i) + r^2 \ddot g_r (e_i, e_i)\Big) \\
&+ \frac34\sum_{i = 1}^n r^2 \dot g_r (e_i, e_i)^2 - 4\left( \sum_{i = 1}^n r \dot g_r(e_i, e_i)\right)^2 \\
\end{align*}
Letting $r$ tend to $0$ we obtain, using the properties stated in Definition \ref{generalizedConeMetric},
\begin{equation*}
    \lim_{r\to 0}  r^2  \scal_{\conic g_r}(r,x) - \scal_{g_0}(x)  =  - n(n-1).
 \end{equation*}
\end{proof}

\begin{remark} This computation shows that the implication ``$\scal_{\conic g_r} \geq 0 \Longrightarrow \scal_{g_0}\geq n(n-1)$'' can be drawn under slightly weaker conditions on the family $g_r$ than those stated in Definition \ref{generalizedConeMetric}.
The details are left to the reader.
\end{remark}

\subsection{Twisted Dirac operator on manifolds with cone-like singularities}\label{sectionDiracGenCone}

\begin{setup} \label{setupGeneralizedCone}
Let
	\begin{enumerate}[label=\myicon]
		\item $(N^{n+1},G)$ be a compact connected Riemannian spin manifold with cone-like singularities, where $n \geq 3$.
		\item  $f\colon (N,d_G) \to \SSS^{n+1}$ be a $\Lambda$-Lipschitz map for some $\Lambda > 0$.
		\item $(E_0,\nabla^{E_0})$ be a smooth Hermitian vector bundle over $\SSS^{n+1}$ with smooth metric connection $\nabla^{E_0}$.
		\item  $  (E,\nabla^E):=f^*(E_0,\nabla^{E_0}) \to (N,G)$ be the pull back of $E_0$ under $f$.
		      Then $E \to N$ is a Lipschitz bundle with induced Hermitian metric and metric Lipschitz connection $\nabla^E$.
	\end{enumerate}
\end{setup}

In the following we use the notation from Definition \ref{MfdConeLikeSing}.
Set
\begin{equation*}
 M := \partial \mathcal{K}, \qquad \conic := N\setminus \mathcal{K},
\end{equation*}
and identify $(\conic, G\vert_{\conic})$ with $((0,\vartheta)\times M, {\conic g_r})$ via $\nu$ where  $0<\vartheta \leq 1$.

In the remainder of the section we fix a smooth background Riemannian metric $\gamma$ on $M$.
For $v \in TM$ and a smooth Riemannian metric $g$ on $M$, we set $v^{g} = b_{g}(v)$.
Recall that $b_g \colon TM \to TM$ is a $\gamma$-self adjoint positive isomorphism that sends $\gamma$-orthonormal frames $(e_1, \ldots, e_n)$ of $M$, to $g$-orthonormal frames $(e_1^{g}, \ldots, e_n^{g})$ of $M$. 
Furthermore, the map $b_g$ depends continuously on $g$.
For further details, see \cite{CHS}*{Section 4}.

We equip $N$ with a smooth background Riemannian metric $\Gamma$ that satisfies
\begin{equation*}
	\Gamma\vert_{\conic} = \nu^* (\conic \gamma) = \nu^*(d r^2 + r^2 \gamma) .
\end{equation*}
This can be obtained by gluing $\conic \gamma$ with a metric on the bulk part and making $\vartheta$ smaller if necessary.

We denote the spinor bundle associated with $\Gamma$ by $\spinor \to N$ and the spinor bundle associated with $\gamma$ by $S_M\to M$.
Let
\begin{equation}\label{twistedDiracGenCone}
	\Dirac_E \colon L^2_G(N, \spinor\otimes E)\supset \Lip_\cc(N, \spinor\otimes E) \to L^2_G(N, \spinor\otimes E)
\end{equation}
be the Dirac operator with respect to $G$ twisted by the Lipschitz bundle $E$. This is a symmetric operator and hence closable.

\begin{proposition}\label{TwistedDiracGenConeAbstractConeOp}
	Assume we are in the situation of Setup \ref{setupGeneralizedCone} with $\scal_{g_0}> 1$.

	Then the twisted Dirac operators  $\Dirac^{\pm}_E$ on $N$ defined in \eqref{twistedDiracGenCone} are abstract cone operators in the sense of Definition \ref{abstractcone}.
	Furthermore, their closures
	\[
		\bar \Dirac\vphantom{\Dirac}^{\pm}_E \colon \dom (\bar \Dirac\vphantom{\Dirac}^{\pm}_E) \to L^2_G(\susp , \spinor^{\mp} \otimes E)
	\]
	are Fredholm operators.
\end{proposition}

Let

\begin{enumerate}[label=\myicon]
	\item  $\EE := L^2_G(N,\spinor^+ \tensor E)$,
	\item $\FF := L^2_G(N,\spinor^- \tensor E)$,
	\item $\EE_{\bulk} := L^2_G(\mathcal{K}, \spinor^+ \otimes E)$,
	\item $\EE_{\cone} :=  L^2_G(\conic, \spinor^+ \otimes E)$,
	\item $\FF_{\bulk} := L^2_G(\mathcal{K}, \spinor^- \otimes E)$,
	\item $\FF_{\cone} :=  L^2_G(\conic, \spinor^- \otimes E) $.
\end{enumerate}

We first consider the untwisted Dirac operator
\begin{equation*}
	\Dirac\colon C_\cc^\infty(N, \spinor) \to C_\cc^\infty(N, \spinor)
\end{equation*}
associated with the metric $G$ on $N$.
For $r \in [0, \vartheta)$, let $D_{M, g_r}  \colon L^2_{g_r}(M, S_M) \supset C^\infty(M, S_M) \to L^2_{g_r}(M, S_M)$ be the Dirac operator associated with $g_r$ on $M$. 
Let $r \in (0,\vartheta)$.
Denoting by $D_{M, r^2 g_r}$ the Dirac operator associated with $r^2 g_r$ on $M$, we have
\begin{equation*}
	D_{M, r^2g_r} = \tfrac1r D_{M, g_r} \colon L^2_{g_r}(M, S_M) \supset C^\infty(M, S_M) \to L^2_{g_r}(M, S_M) .
\end{equation*}

Let
\begin{equation*}
	\rho_r:=  \frac{d \mu_{g_r}}{d \mu_{\gamma}}
\end{equation*}
be the volume density with respect to $\gamma$.
Multiplication with $\sqrt{\rho_r}$ induces an isometry
\begin{equation}\label{LinkIsometry}
  L^2_{g_r}(M, S_M) \to L^2_{\gamma}(M, S_M).
\end{equation}

Since $g_r$ converges to $g_0$ in $C^2$, we get the following lemma (see \cite{CHS}*{Proposition 4.10}).
\begin{lemma}\label{DiracOpContFamily}
	The family of Dirac operators
	\begin{equation*}
		[0,\vartheta) \to \boundedops\left(H^1_{\gamma}(M, S_M), L^2_{\gamma}(M, S_M)\right),\qquad r \mapsto  D_{M,r}:= (\rho_r)^{\tfrac12} \cdot \bar D_{M, g_r}\cdot  (\rho_r)^{-\tfrac12}
	\end{equation*}
	is continuous.
\end{lemma}

Using the isometric identification $TM \cong T( \{\tfrac{\vartheta}{2}\} \times M) \subset T\conic$, $\xi \mapsto  \tfrac{2}{\vartheta} \, \xi$, we consider  $\spinor^{\pm}|_M \to M$ as $\Cl(M,\gamma)$-module bundles, by letting $\xi \in TM$ act by Clifford multiplication with $\pm \tfrac{2}{\vartheta} \, \xi \cdot \partial_r$.
In an analogous fashion as in the beginning of the proof of Proposition \ref{coneexample}, we obtain unitary $\Cl(M,\gamma)$-module bundle isomorphisms $\alpha \colon \spinor^+ |_M \cong S_M$ and $\beta \colon\spinor^-|_M \cong S_M$ satisfying $\beta \circ c(\partial_r) = \alpha$.

Using parallel translation of sections in $C^{\infty}(M, S_M)$ along geodesic lines $(0,\vartheta) \times \{x\}\subset \conic$, the isomorphisms $\alpha$ and $\beta$ induce isometries
\begin{align*}
	\Phi_\alpha  \colon L^2_G(\conic, \spinor^+|_{\conic}) & \to L^2( (0,\vartheta), L^2_{r^2g_r}(M, S_M)),    \\
	\Phi_\beta  \colon L^2_G(\conic, \spinor^-|_{\conic})  & \to L^2( (0,\vartheta),  L^2_{r^2g_r}(M, S_M)) .
\end{align*}

With \cite{BaerGauduchonMoroianu}*{Equation (3.6)} we obtain

\begin{equation*}
	\Phi_\beta \circ \Dirac^+ \circ\, \Phi_\alpha^{-1} = \partial_r - \tfrac{n}{2} H_r + \tfrac1r D_{M,g_r} ,
\end{equation*}
where
\begin{equation*}
	H_r= -\tfrac1{2n}\text{tr}_{r^2g_r}(\partial_r (r^2g_r))
\end{equation*}
is the mean curvature.

Recall that for the volume forms we have $d\mu^{r^2g_r} = r^nd\mu^{g_r}$.
We post-compose the isometries $\Phi_\alpha$ and $\Phi_\beta$ with the isometry in \eqref{LinkIsometry} and multiplication with $r^{\tfrac n2}$ to obtain isometries
\begin{align*}
	\Phi_\alpha' \colon L^2_G(\conic, \spinor^+\vert_{\conic}) \to & \ L^2((0,\vartheta), L^2_{\gamma}(M, S_M)), \\
	\Phi_\beta' \colon L^2_G(\conic, \spinor^-\vert_{\conic}) \to  & \ L^2((0,\vartheta), L^2_{\gamma}(M, S_M)).
\end{align*}

Using
\begin{equation*}
	\partial_r (r^n\rho_r) = -n H_r\cdot  r^n\rho_r,
\end{equation*}
we compute
\begin{equation*}
	\Phi_\beta' \circ \Dirac^+ \circ\, (\Phi_\alpha')^{-1} =  \partial_r + \tfrac1r D_{M,r}.
\end{equation*}
Here $e_1,\ldots, e_n$ denotes a $\gamma$-orthonormal frame on $M$.

We now consider the Dirac operator twisted by $E$. 
Let $\{x_1, \ldots x_k\}= \bar N\setminus N$ be the set of cone tips of $\bar N$ and let $\eps > 0$.
Let
\begin{equation*}
	V_i:=B_{\eps}(\bar f(x_i)) \subset  \SSS^{n+1}
\end{equation*}
be the open $\eps$-ball centered at $\bar f(x_i)$.
By choosing $\eps$ small enough we can assume $V_i\cap V_j = \emptyset$ if $\bar f(x_i)\neq \bar f(x_j)$.
Let
\begin{equation*}
	V:= \bigcup_{i= 1}^k V_i \subset \SSS^{n+1}.
\end{equation*}
For every point $\bar f(x_i)\in \SSS^{n+1}$ consider the fibre
\begin{equation*}
	F_i := E_0\vert_{\bar f(x_i)}.
\end{equation*}
Using that all fibres are isometric, we identify each $F_i$ with a fixed fibre $F$.
Using parallel translation along geodesic lines emanating from $\bar f(x_i) \in \SSS^{n+1}$, we obtain a unitary trivialization 
\begin{equation}\label{trivialSigma}
	E_0\vert_{V} \cong V\times F.
\end{equation}
This induces a unitary trivialization
\begin{equation}\label{TrivialzationTwistBundle}
	E\vert_{f^{-1}(V)} \cong f^{-1}(V)\times F.
\end{equation}
By passing to a larger $\mathcal{K} \subset N$ if necessary, we can assume that
 \begin{equation*}
	\conic \subset f^{-1}(V).
\end{equation*}

Let $\omega \in \Omega^1( V , \End(F))$ be the smooth connection $1$-form of the connection $\nabla^{E_0}$ on $E_0|_V\to \SSS^{n+1}$ with respect to \eqref{trivialSigma}.

Consider the Hilbert space
\begin{equation*}
	\LL:= L^2_{\gamma}(M, S_M\otimes F).
\end{equation*}

Let $(e_1, \ldots, e_n)$ be a local orthonormal frame of $(M ,\gamma)$.
With the trivialization \ref{trivialSigma} and the isometries $\Phi_\alpha', \Phi_\beta'$ we obtain isometries
\begin{align*}
	\Phi_{\EE} \colon L^2_G(\conic, \spinor^+\otimes\ E) \to & \ L^2((0,\vartheta), \LL), \\
	\Phi_{\FF} \colon L^2_G(\conic, \spinor^+\otimes\ E) \to & \ L^2((0,\vartheta), \LL).
\end{align*}
Put  
\begin{equation*}
  S_0:= D_{M,0}\otimes \id_F
\end{equation*} 
and
\begin{equation*}
     S_1(r):=  \left(D_{M,r} -  D_{M,0}\right)\otimes \id_F + r \cdot \Bigg(  \sum_{i = 1}^n c\left(  e_i\right)\otimes (f^*\omega)\left(  \tfrac1r \bar e_i^{g_r}\right) +  c\left(\partial_r \right)\otimes (f^*\omega)\left(  \partial_r \right) \Bigg).
\end{equation*}
Suppressing $\Phi_{\EE}$ and $\Phi_{\FF}$ from the notation, we then have  $\Dirac^+_E = \partial_r + \tfrac1r(S_0 + S_1(r))$.
Note that $S_0$ satisfies the spectral gap condition \ref{AnSpGap} by the assumption $\scal_{g_0}>1$.

The map $S_1 \colon (0,\vartheta) \to \boundedops( \dom(S_0), \LL)$ is continuous, hence measurable. 
Furthermore, it is bounded since $\omega$ is bounded and $f$ is $\Lambda$-Lipschitz.
Due to Lemma \ref{DiracOpContFamily}, we obtain
\begin{align*}
	\abs{S_1(r)S_0^{-1}}_{L^\infty((0,\vartheta), \boundedops(\LL))} \overset{r\to 0}{\longrightarrow}&\ 0, \\
	\abs{S_0^{-1}S_1(r)}_{L^\infty((0,\vartheta), \boundedops(\LL))} \overset{r\to 0}{\longrightarrow}&\ 0.
\end{align*}

The axioms \ref{AnSetup4} - \ref{boundedperturbtwo} are now checked analogously as in the case of the spherical suspension.

\begin{proposition} For all $\varphi, \psi \in C_{\cc,0}([0,\thet), \RR)$ with $\psi|_{\supp \varphi} \equiv 1$, there exists an interior parametrix of $\Dirac_E$ for $\varphi$ and $\psi$.
\end{proposition}

\begin{proof}
	Let $0<a<\inf_{r\in (0,\vartheta)}\supp (1-\varphi)$ and 
	\begin{equation*}
		\Omega := \left([a,\vartheta)\times M\right)\cup \mathcal{K}\subset N,
	\end{equation*} 
	which is a compact manifold with boundary.
	We use now the same doubling procedure as in the proof of Proposition \ref{constr_int_parametrix} 
\end{proof}

\subsection{Integral Schrödinger-Lichnerowicz formula for manifolds with cone-like singularities}

We work in the setting of Setup \ref{setupGeneralizedCone}.

\begin{lemma} \label{curvaturetermtwo}
Let $R^{E_0} \in \Omega^2(\SSS^{n+1}, \End(E_0))$ be the curvature form of  $(E_0,\nabla_0) \to \SSS^{n+1}$.
We consider the almost everywhere defined measurable $2$-form on $N$ with values in $\End(E)$ given by 
\[
     R^E := f^*(R^{E_0}).
\]
Then  $R^E  \in L^{\infty}(\Omega^2(N, \End(E)))$. 
\end{lemma}

\begin{proof}  The proof is analogous to the proof of Lemma \ref{curvatureterm}.
\end{proof}

For $u = \sigma \otimes \eta \in (\spinor \otimes E)_x$, we set 
\begin{equation*}
    \mathcal{R}_{G}^E u:= \frac{1}{2}\sum_{i,j} e_i\cdot e_j \sigma \otimes R^E_{ e_i^G, e_j^G} \eta \in (\spinor \otimes E)_x ,
\end{equation*}
where $(e_1, \ldots, e_{n+1})$ is some $\Gamma$-orthonormal basis of $(T_x N, G_x)$.
It follows from Lemma \ref{curvaturetermtwo} that $\mathcal{R}_{G}^E$ defines a bounded operator  $L^2_G(\susp M, \spinor \otimes E) \to L^2_G(\susp M, \spinor \otimes E)$.

Due to Proposition \ref{TwistedDiracGenConeAbstractConeOp} the twisted Dirac operator $\Dirac_E$ has a unique closed extension $\bar \Dirac_E$.
We furthermore consider the operator
\[
   \nabla_E = \nabla^{\spinor \otimes E} \colon L^2_G(N , \spinor \otimes E) \supset \Lip_{\cc}(N , \spinor \otimes E) \to L^2_G(N, T^* N \otimes \spinor \otimes E).
\]
As a linear differential operator it is closable.
Let
\[
   \bar \nabla_E \colon  L^2_G(N , \spinor \otimes E) \supset \dom(\bar \nabla_E) \to L^2_G(N, T^* N \otimes \spinor \otimes E)
\]
be its closure.
We now obtain the following integral Schr\"odinger-Lichnerowicz formula for manifolds with cone-like singularities. 

\begin{theorem}\label{thm:SL_GeneralizedCone}
	Assume we are working in Setup \ref{setupGeneralizedCone} with $\scal_G \geq 0$.
	\begin{enumerate}[label=\textup{(\roman*)}]
	\item  The sesquilinear form $\Lip_\cc(N, \spinor \otimes E) \times \Lip_{\cc}(N, \spinor \otimes E) \to \CC$, $(u,v) \mapsto \bigl(  \nabla_E u, \nabla_E v \bigr)_{L^2}$ extends to a continuous sesquilinear functional
	\[
	  \bigl( \nabla_E -  , \nabla_E -  \bigr)_{L^2} \colon  \dom( \bar \Dirac_E) \times \dom( \bar \Dirac_E) \to \CC .
	\]
	\item  The sesquilinear form $\Lip_{\cc}(N, \spinor \otimes E) \times \Lip_{\cc}(N, \spinor \otimes E) \to \CC$, $(u,v) \mapsto \bigl( \scal_G u,v \bigr)_{L^2}$, extends to a continuous sesquilinear functional
	\[
	   \bigl(\scal_G - , - \bigr)_{L^2}   \colon \dom( \bar \Dirac_E) \times \dom( \bar \Dirac_E) \to \CC .
	\]
	\item  For $u, v \in \dom(\bar \Dirac_E)$, the integral Schr\"odinger-Lichnerowicz formula holds: 
	 \begin{equation*}
		\bigl(\bar \Dirac_E u, \bar \Dirac_E v \bigr)_{L^2}=\bigl( \nabla_E u,  \nabla_E v \bigr)_{L^2} +  \tfrac14 \bigl(\scal_G u,v\bigr)_{L^2}  + \bigl(\mathcal R^E u, v \bigr)_{L^2}.
	  \end{equation*}
	\end{enumerate}
	\end{theorem}

\begin{proof}
	The proof of Theorem \ref{thm:SL_cone} carries over to this situation. 
\end{proof}

\subsection{Index formula on manifolds with cone-like singularities}

Assume that we are in the situation of Setup \ref{setupGeneralizedCone}, where $E_0$ is  the canonical spinor bundle $\Sigma \to \SSS^{n+1}$ and $\scal_G \geq 0$. We use the notation of Section \ref{sectionDiracGenCone}.

We obtain a corresponding decomposition $E = E^+ \oplus E^-$, where $E^{\pm}= f^*(E_0^{\pm})$.
With respect to the direct sum decomposition
\begin{equation*}
  \spinor \otimes E = (\spinor \otimes E)^+ \oplus (\spinor\otimes E)^-,
\end{equation*}
where
\begin{align} 
(\spinor \otimes E)^+ & = \big( \spinor^+ \otimes\ E^+ \big) \oplus \big( \spinor^- \otimes\ E^- \big),  \\
 (\spinor\otimes E)^-  & = \big( \spinor^- \otimes\ E^+ \big)  \oplus \big( \spinor^+ \otimes\ E^- \big) .
\end{align} 
The twisted Dirac operator $\Dirac_E$ acting on sections of $\spinor \otimes E$ is of block diagonal form,  $\Dirac_E = \Dirac_E^+ \oplus \Dirac_E^-$, where
\begin{equation*}
  \Dirac^{\pm}_E\colon L_G^2(N,\spinor^{\pm}  \tensor E) \supset \Lip_\cc(N,\spinor^{\pm} \tensor E)  \to   L^2_G(N,\spinor^{\mp}  \tensor E) .
\end{equation*}
It follows from Proposition \ref{TwistedDiracGenConeAbstractConeOp}, that $\Dirac_E^{\pm}$ are abstract cone operators and their closures $\bar \Dirac\vphantom{\Dirac}^{\pm}_E$ are Fredholm.
The aim of this section is to prove the following index formula.
\begin{proposition}\label{prop:IndexFormulaGenCone} We have
	\begin{equation*}
	  \ind( \bar \Dirac\vphantom{\Dirac}^+_E) = (-1)^{\tfrac{n+1}2} \deg(f) \cdot \chi(\SSS^{n+1}).
	\end{equation*}
  \end{proposition}

\begin{proof}
We will use a deformation argument similar to the one used in the proof of Proposition \ref{prop:IndexFormula}.
Specifically, we deform the metric on N into one with straight cones.
For this, let $\varphi\in C_\cc^\infty([0,\vartheta))$ be a cutoff function with $\varphi \equiv 1$ on $(0,\vartheta')$, where $0<\vartheta' < \vartheta$.
Furthermore, let $\gamma$ be a smooth Riemannian metric on $M$.
We set
\begin{equation*}
	g_{r,t}:= \varphi(r)\big( t\gamma + (1-t)g_r\big) + \big(1-\varphi(r)\big)g_r
\end{equation*}
and
\begin{equation*}
	{\conic g_{r,t}}:= d r^2 + r^2g_{r,t}.
\end{equation*}
As in the case of the spherical suspension, we can choose the background metric $\gamma$ on $M$ in such a way that $\scal_{g_{r,t}}>1$, which ensures that the associated family of Dirac operators on the link have empty intersection with $[-\tfrac{1}{2}, + \tfrac{1}{2}]$.

Then, $\conic g_{r,t}$ is a generalized cone metric for every $t\in [0,1]$ with limit metric
\begin{equation*}
	g_{0,t}:= t\gamma + (1-t)g_0
\end{equation*}
and
\begin{equation*}
	{\conic g_{r,1}} = d r^2 + r^2 \big(\varphi(r)\gamma + (1-\varphi(r))g_r\big)
\end{equation*}
is a straight cone metric in a neighborhood of the cone tips.

With the assignment
\begin{equation*}
	G_t\vert_{\conic}:= \nu^* ({\conic g_{r,t}})
\end{equation*}
and
\begin{equation*}
	G_t\vert_\mathcal{K} := G\vert_\mathcal{K}
\end{equation*}
we obtain a family of metrics on $N$ such that $(N, G_1)$ is a manifold with conical singularities in the sense of \cite{Chou}.

As in \eqref{TrivialzationTwistBundle} we find a trivialization of $E^\pm$ over an open neighborhood $\conic\subset N$ of the cone tips. 
Due to Proposition \ref{lem:smoothStructureOnE}, we can equip $E^\pm$ with a smooth structure that is compatible with the Lipschitz structure and with the given trivialization over $\conic$.
Choose a smooth Hermitian bundle metric on $E^{\pm}$ and a smooth metric connection $\widetilde \nabla^{E^{\pm}}$ on $E^{\pm}$ that has a vanishing local connection 1-form for the given trivialization over $\conic$.
We set
\begin{equation*}
	\nabla^{E^{\pm}}_t := t \widetilde\nabla^{E^{\pm}} + (1-t)\nabla^{E^{\pm}}
  \end{equation*}
and denote its connection 1-form, with respect to the trivialization over $\conic$, with $\omega_t$.

Let
\begin{equation*}
  \rho^{G_t} = \tfrac{d\mu^{G_t}}{d\mu^{G_1}}.
\end{equation*}
be the volume density with respect to $G_1$. 
Multiplication with $\sqrt{\rho^{G_t}}$ induces isometries
\begin{align*}
  \Psi_t^\pm\colon L^2_{G_t}(N, \spinor^\pm \otimes E) & \to L^2_{G_1}(N, \spinor^\pm \otimes E).
\end{align*}
We consider the family of operators
  \begin{equation*}
	  \Dirac_t := \Psi_t^- \circ \Dirac_{E,t}^+\circ \left(\Psi_t^+\right)^{-1} \colon L^2_{G_1}(N, (\spinor \otimes E)^+) \supset \Lip(N, (\spinor \otimes E)^+) \to L^2_{G_1}(N, (\spinor \otimes E)^-),
  \end{equation*}
  where $\Dirac_{E,t}^+$ denotes the Dirac operator for the metric $G_t$ on $N$ twisted with $(E,\nabla_t^E)$. Due to Proposition \ref{TwistedDiracGenConeAbstractConeOp}, the $\Dirac_{E,t}^+$ form a family of abstract cone operators.

On $(0,\vartheta')\times M \subset \conic$ we have
\begin{equation*}
  \rho_{r,t}:=\rho^{G_t} = \tfrac {d\mu^{g_{r,t}}}{d\mu^{\gamma}}
\end{equation*}
Let
\begin{equation*}
  D_{M,r}^t:= \left(\rho_{r,t}\right)^{\tfrac12}\cdot D_{M, g_{r,t}} \cdot \left(\rho_{r,t}\right)^{-\tfrac12}.
\end{equation*}
Then we have
\begin{equation*}
	\Dirac_t = \partial_r + \tfrac1r (S_{0,t} + S_{1,t}),
\end{equation*}
where
\begin{equation*}
	S_{0,t} = \left(D_{M,0}^t\otimes \id_F\right)
\end{equation*}
and
\begin{equation*}
	S_{1,t}(r) = \ \left(D_{M,r}^t - D_{M,0}^t\right)\otimes \id_F  + r \cdot \Bigg( \sum_{i = 1}^{n} c\left(  e_i\right)\otimes \omega_t\left(\tfrac1r \bar e_i^{g_{r,t}}\right) + c\left(\partial_r \right)\otimes \omega_t \left(  \partial_r \right) \Bigg).
\end{equation*}

Put $\LL:= L^2_\gamma(M, S_M \otimes F)$.
Then the domain of the link operator $S_{0,t}$ is equal to $\mathfrak{D}_{\link} := \Lip(M, S_M \otimes F)$ which is independent of $t$, and the graph norms on $\mathfrak{D}_{\link}$ induced by $S_{0,t}$ are pairwise equivalent. 

Furthermore, the map 
  \begin{equation*}
	[0,1]\to \boundedops(\mathfrak{D}_{\link},  \LL), \quad t \mapsto S_{0,t},
  \end{equation*}
  is continuous.
  The spectrum of $\bar S_{0,t}$ has empty intersection with $[-\tfrac{1}{2}, + \tfrac{1}{2}]$ by our choice of $\gamma$.
 Note also that 
	  \begin{equation*} [0,1] \to L^{\infty}((0,\thet'), \boundedops(D_{\link} , \LL)), \quad t \mapsto S_{1,t} , 
	  \end{equation*}
 is continuous.
 Moreover, the constant $0 < \thet' \leq 1$ such that \ref{boundedperturbtwo} holds over some $(0,\thet')$, can be chosen independent of $t$.

Each $\Dirac_t$ has an interior parametrix so that the closures $\bar \Dirac_t$ are Fredholm.

 Proposition \ref{lem:deformationLemma} shows that $\mathfrak{D}:= \dom(\bar \Dirac_t)$ is independent of $t$, the graph norms of $\bar \Dirac_t$ on $\mathfrak{D}$ are independent of $t$, up to equivalence, and the map 
\[
    [0,1] \to \boundedops(\mathfrak{D}, L^2_{G_1}(N, (\spinor \otimes E)^-)), \qquad t \mapsto \bar \Dirac_t, 
\]
is continuous.
In particular, the Fredholm indices of $\bar \Dirac_0 = \Psi_0^-\circ \bar \Dirac\vphantom{\Dirac}^+_E\circ (\Psi_0^+)^{-1}$ and of $\bar \Dirac_1 = \bar \Dirac\vphantom{\Dirac}^+_{E,1}$ are equal. 
Since $\Psi_t^\pm$ are isometries, we have
\begin{equation*}
   \ind(\bar \Dirac\vphantom{\Dirac}^+_E) = \ind(\bar \Dirac_0)  = \ind(\bar \Dirac\vphantom{\Dirac}^+_{E,1}).
\end{equation*}

Using \cite{Chou}*{Remark 5.25} and the same index computation as in the case of the spherical suspension we obtain
\begin{equation*}
	\ind (\bar\Dirac\vphantom{\Dirac}^+_E) = \int_N \omega_{E^+} - \omega_{E^-} = \int_N f^*(\ch(E_0^+) - \ch(E_0^-)) = (-1)^{\frac{n+1}2}\deg(f)\chi(\SSS^{n+1}).
\end{equation*}
In the last step we use that the connection $\nabla_1^{E^\pm}$ is flat over $\conic$ by construction. 
Hence, the integral over $N$ restricts to an integral over the compact manifold $\mathcal{K} \subset N$ with boundary and we can consider the degree of $f$ as the homological mapping degree for the map of pairs $(\mathcal{K}, \partial \mathcal{K}) \to (\SSS^{n+1}, V)$.
\end{proof}

\subsection{Proof of Theorem \ref{LlarullGeneralizedCone}}

We work in Setup \ref{setupGeneralizedCone}, where $E_0 = E_0^+ \oplus E_0^- \to \SSS^{n+1}$ is the spinor bundle over the even dimensional sphere $\SSS^{n+1}$.

Due to Proposition \ref{prop:IndexFormulaGenCone}, the index of the twisted Dirac operator 
\begin{equation*}
  \bar\Dirac_E \colon L^2_G( N, \spinor \otimes E) \supset \dom(\bar\Dirac_E) \to L^2_G(N, \spinor\otimes E).
\end{equation*}
is given by
\begin{equation*}
  \ind (\bar\Dirac\vphantom{\Dirac}^+_E) = (-1)^{\tfrac{n+1}2}\deg (f)\chi (\SSS^{n+1}) .
\end{equation*}
By assumption the degree of $f$ is non-zero and the Euler characteristic of the even dimensional sphere is 2. 
Possibly after changing the orientation of $N$, the index $\ind (\bar\Dirac\vphantom{\Dirac}^+_E)$ is positive.
We obtain a non-zero harmonic spinor field in $\dom (\bar \Dirac\vphantom{\Dirac}^+_E)\subset L^2_G(N, (\spinor\otimes E)^+)$, hence a non-zero harmonic spinor field $\psi \in \dom (\bar \Dirac_E)$.

With  \cite{CHS}*{Proposition 6.1} we obtain a pointwise estimate
\begin{equation}\label{eq:curv_lowerboundGenCone}
  \langle \mathcal{R}^E \omega, \omega \rangle\geq -\tfrac14 (n+1)n |\omega|^2 \qquad \forall \omega \in (\spinor\otimes E)_x\  
\end{equation}
at every point $x\in N$ where $f$ is differentiable.

Plugging the harmonic spinor into the integral Schrödinger-Lichnerowicz formula, Theorem \ref{thm:SL_GeneralizedCone}, and using \eqref{eq:curv_lowerboundGenCone} we get
\begin{align*}
  0 = \| \bar \Dirac_E\psi \|_{L^2}^2 &= \| \nabla_E \psi \|_{L^2}^2 +   \tfrac14\left(\scal_G \psi, \psi\right)_{L^2}   + \bigl( \mathcal{R}^E \psi, \psi \bigr)_{L^2} \\
  & \geq   \tfrac14\left(\scal_G \psi, \psi\right)_{L^2}   -\tfrac14 (n+1)n \|\psi \|_{L^2}^2 \\
  &\geq 0 .
\end{align*}
Hence, we have
\begin{align}
 &\|  \nabla_E \psi \|_{L^2} = 0 \label{eq:parallelharmspinorGenCone},\\
  &\langle \mathcal{R}^E \psi, \psi \rangle = -\tfrac14 (n+1) n |\psi |^2 \quad \mathrm{a.e.},  \label{eq:curv_equalityGenCone}\\
  &\scal_G = (n+1)n. \label{eq:equality_scalcurvGenCone}  
\end{align}

From \eqref{eq:parallelharmspinorGenCone} it follows that for $|\psi|\in H^1(\susp M)$ we have
\begin{equation*}
  d | \psi |^2 = 2 \langle \nabla_E\psi,\psi  \rangle = 0\qquad \text{a.e.}
\end{equation*}
and therefore, that $|\psi| $ is constant almost everywhere. Consequently, $\psi$ only vanishes on a set of measure zero.

Equation \eqref{eq:curv_equalityGenCone} says that we are in the equality case of \cite{CHS}*{Proposition 6.1}.
This implies that for almost all $x\in N$ where $f$ is differentiable the differential $d_x f$ is an isometry and for all $\Gamma$-orthonormal vectors $v,w \in T_xN$ we have
\begin{equation}\label{eq:doubleCM_invarianceGenCone}
  \bigl( v\cdot w \otimes d_x f(v^G)\cdot d_x f (w^G) \bigr) \cdot \psi = \psi .
\end{equation}

\begin{proposition}\label{prop:differential_orientpresGenCone}
At all points $x$ where the differential $d_x f$ is defined, it is orientation preserving.
  \end{proposition}
  
  \begin{proof} The proof is analogous to the proof of Proposition \ref{prop:differential_orientpres}.
  \end{proof}

To complete the proof of Theorem \ref{LlarullGeneralizedCone}, set 
\[
   \Sigma := f^{-1}(\{\bar f(x_1), \ldots, \bar f(x_k) \}) \subset N . 
\]
Since $d f$ is an orientation preserving isometry almost everywhere, Proposition 2.9 and Lemma 2.12 in \cite{CHS} imply that $f$ is a local homeomorphism. 
In particular, $\Sigma \subset N$ is a discrete subset and $N \setminus \Sigma$ is path connected. 
Since $\bar f \colon \bar N \to \SSS^{n+1}$ is proper, the restriction
\[
      f|_{N \setminus \Sigma} \colon N  \setminus \Sigma \to \SSS^{n+1} \setminus \{\bar f(x_1), \ldots, \bar f(x_n) \}
 \]
is proper.
Using that  $N \setminus \Sigma$ is path connected, that $df$ is an orientation preserving isometry almost everywhere and that $\SSS^{n+1} \setminus \{\bar f(x_1), \ldots, \bar f(x_n) \}$ is simply connected (as $n \geq 3$), Theorem 2.4 in \cite{CHS}  implies that $f|_{N \setminus \Sigma}$ is a metric isometry with respect to the path metrics on $N \setminus \Sigma$ and on $\SSS^{n+1} \setminus \{\bar f(x_1), \ldots, \bar f(x_n) \}$.

We claim that $\Sigma = \emptyset$.
Suppose that this is not the case.
Let   $p \in \Sigma$ such that $f(p) = \bar f(x_i)$ for some $i \in \{1, \ldots, k\}$.
Let $d := \dist_{\bar N}(p, x_i)$. 
By the continuity of $\bar f$, there exists a point $q \in N$ near $x_i \in \bar N$ with $\dist_N(p,q) > \tfrac{d}{2}$ and $\dist_{\SSS^{n+1} \setminus \{\bar f(x_1), \ldots, \bar f(x_n) \}}(f(p), f(q)) < \tfrac{d}{2}$.
This is a contradiction.
So  $\Sigma = \emptyset$.

In summary,  $f \colon N \to \SSS^{n+1} \setminus \{\bar f(x_1), \ldots, \bar f(x_n) \}$ is a metric isometry. 
Since the Riemannian metrics on domain and target are smooth, we conclude by the Myers-Steenrod theorem  that $f$ is a smooth Riemannian isometry. 
This completes the proof of Theorem \ref{LlarullGeneralizedCone}.

\appendix

\section{Smooth structures on Lipschitz bundles}

\begin{proposition}\label{lem:smoothStructureOnE}
Let $M$ be a smooth (not necessarily compact) manifold, let $E \to M$ be a Lipschitz bundle, let $K \subset M$ be a compact subset and let 
\[
    \psi \colon E|_{M \setminus K} \cong (M \setminus K) \times \CC^r 
\]
be a Lipschitz trivialization.

Then there is a smooth structure on $E$ which is compatible with the given Lipschitz structure and with the local trivialization $\psi$.
\end{proposition}

\begin{proof}
Let $r := \rk E$.
Let $V_0, V_1 \subset M$ be open subsets of $M$ with compact closures and such that $K \subset V_0$ and $\bar V_0  \subset V_1$.

There is a finite family $(U_i)_{i\in I}$ of open subsets of $M$ covering $K$ such that each $U_i$ is contained in $V_0$ and for each $i \in I$, there exists a Lipschitz trivialization $\psi_i \colon E|_{U_i} \cong U_i \times \CC^r$.
The family $(U_i)$ together with the open subset $M \setminus K \subset M$ form a finite open cover of $M$.
Using a finite partition of unity subordinate to this cover, using the Lipschitz trivializations $\psi_i$ of $E|_{U_i}$ and  the Lipschitz trivialization $\psi$ of $E|_{M \setminus K}$, we obtain a   Lipschitz embedding $\Phi \colon E \to M \times \CC^N$ into the trivial vector bundle of rank $N$ over $M$, for some $N \in \mathbb{N}$.
By construction, the  composition
\[
 \Phi|_{M \setminus \bar V'} \circ \psi|_{E|_{M \setminus \bar V_0}}^{-1} \colon (M \setminus \bar V_0) \times \CC^r \to (M \setminus \bar V_0) \times \CC^N
\]
is induced by the canonical embedding $\CC^r = \CC^r \times 0 \subset \CC^N$.
For $x \in M$, let $p(x) \in \CC^{N \times N}$ be the orthogonal projection onto $\Phi(E_x) \subset \CC^N$.
The map 
\[
    p \colon  M  \to \CC^{N \times N}, \quad x \mapsto p(x),
\]
is locally Lipschitz continuous and takes values in self-adjoint projections.

Let $\gamma$ be a smooth Riemannian metric on $M$ and let $d_\gamma$ be the induced metric on $M$.
There exists $\eps_0 > 0$ such that the $\eps_0$-neighborhood of $K$ in $(M, d_{\gamma})$ is contained in $V_0$ and the exponential map is defined on the $\eps_0$-ball in $(T_x M,  \gamma_x)$ for all $x \in \bar V_0$.
Let  $\kappa \colon \R^{n+1} \to \R_{\geq 0}$ be a smooth function supported in the unit ball such that $\int_{\R^{n+1}} \kappa(v) d v = 1$.
For $0 < \eps < \eps_0$, we obtain a smooth map $p_{\eps} \colon M \to \CC^{N \times N}$ such that 
\begin{enumerate}[label=\myicon]
   \item on $M \setminus V_0$ , the map $p_{\eps}$  is equal to $ \mathbb{1}_{r} \oplus  0$,
   \item on $\bar V_0$, the map $p_{\eps}$  is  equal to the convolution
\[
   p_{\eps}(x) = \frac{1}{\eps^{n+1}} \int_{T_x M} \kappa\left( \frac{|v|}{\eps} \right) p( \exp_x v) d v .
\]
\end{enumerate}

For $0 < \eps < \eps_0$, put $\tilde p_{\eps} := \frac{1}{2}( p_{\eps} + p^*_{\eps}) \colon  M \to \CC^{N \times N}$.
The map $\tilde p_{\eps}$ is smooth and equal to $ \mathbb{1}_{r} \oplus  0$ on $M \setminus V_0$.
Furthermore, each $\tilde p_{\eps}(x)$ is self-adjoint.

Choosing $\eps$ small enough we can assume that $ \| \tilde p_{\eps} - p \|_{C^0(M, \CC^{N \times N})} < \tfrac{1}{3}$.
Define a Lipschitz map $P \colon M \times [0,1] \to \CC^{N \times N}$,
\[
    P(x,t) := (1-t) \tilde p_{\eps}(x) + t p(x) .
\] 
Then each $P(x,t) \in \CC^{N \times N}$ is self-adjoint and $P(x,t) =  \mathbb{1}_{r} \oplus  0$ for $(x,t) \in (M \setminus V_0) \times  [0,1]$.
Moreover, $\|P(-, t) - p \|_{C^0(M, \CC^{N \times N})} < \tfrac{1}{3}$ for each $t \in [0,1]$.
This implies that  the spectrum of  $P|_{ \bar V_1 \times [0,1]}$, considered as an element in the unital $C^*$-algebra $C^0( \bar V_1 \times [0,1], \CC^{N \times N})$, is contained in $(-\tfrac{1}{3}, \tfrac{1}{3}) \cup (\tfrac{2}{3}, \tfrac{4}{3} )$, see  \cite{Rieffel}*{Lemma 3.2}.

Let $f=\chi_{\left[\nicefrac{2}{3}, \infty\right)} \colon \R \to \{0,1\}$ be the characteristic function.
Functional calculus yields a Lipschitz map $f(P) \colon (\bar V_1, d_{\gamma})  \times [0,1] \to \CC^{N \times N}$ with values in the self-adjoint projections, see \cite{Li}*{Proposition 3.1}.
By the naturality of functional calculus under unital $C^*$-homomorphisms and since $C^{\infty}( \bar V_1 , \CC^{N \times N})$ is closed under holomorphic functional calculus, $f(P)|_{ \bar V_1 \times \{1\}}  \colon  \bar V_1 \to \CC^{N \times N}$ is smooth. 
Furthermore, for $(x,t) \in (V_1 \setminus \bar V_0) \times  [0,1]$, we have $f(P)(x, t) =  \mathbb{1}_{r} \oplus  0$.
Hence, $f(P)$ extends to a Lipschitz map $\tilde P \colon (M, d_\gamma) \times [0,1] \to \CC^{N \times N}$, setting it equal to  $ \mathbb{1}_{r} \oplus  0$ on $ (M \setminus V_0) \times [0,1]$.

We conclude that  $\tilde P|_{M \times \{1\}}$ defines a smooth complex vector bundle on $M$  which is Lipschitz isomorphic to $E$.
Pulling back the smooth structure along this Lipschitz isomorphism defines a smooth structure on $E$ as needed.
\end{proof}

\end{document}